\documentclass[12pt]{amsart}
\usepackage{hyperref,cleveref,amssymb,mathtools,color}

\theoremstyle{plain}
\newtheorem{theorem}{Theorem}[section]
\newtheorem{proposition}[theorem]{Proposition}
\newtheorem{corollary}[theorem]{Corollary}
\newtheorem{lemma}[theorem]{Lemma}

\theoremstyle{definition}
\newtheorem{remark}[theorem]{Remark}
\newtheorem{question}[theorem]{Question}

\newtheorem{example}[theorem]{Example}

\newcommand{\abs}[1]{\lvert#1\rvert}
\newcommand{\norm}[1]{\lVert#1\rVert}
\newcommand{\bigabs}[1]{\bigl\lvert#1\bigr\rvert}

\newcommand{\Bigabs}[1]{\Bigl\lvert#1\Bigr\rvert}

\newcommand{\term}[1]{{\textit{\textbf{#1}}}}   
\renewcommand{\mid}{\::\:}

\newcommand{\goestau}{\xrightarrow{\tau}}	
\newcommand{\goeseta}{\xrightarrow{\eta}}	
\newcommand{\goesl}{\xrightarrow{\lambda}}	
\newcommand{\goesp}{\xrightarrow{\mathrm{p}}}	
\newcommand{\goeslB}{\xrightarrow{\lambda_{\mathcal B}}}	
\newcommand{\goesueta}{\xrightarrow{\mathrm{u}\eta}}	
\newcommand{\goesuIe}{\xrightarrow{\mathrm{u}_I\eta}}	
\newcommand{\goesceta}{\xrightarrow{\mathrm{c}\eta}}	
\newcommand{\goeso}{\xrightarrow{\mathrm{o}}}	

\newcommand{\goesfd}{\xrightarrow{\mathrm{fd}}}
\newcommand{\goesc}{\xrightarrow{\mathrm{c}}}
\newcommand{\goesuc}{\xrightarrow{\mathrm{uc}}}
\newcommand{\goesso}{\xrightarrow{\sigma\mathrm{o}}}	
\newcommand{\goesu}{\xrightarrow{\mathrm{ru}}} 
\newcommand{\goesuo}{\xrightarrow{\mathrm{uo}}}	
\newcommand{\goesun}{\xrightarrow{\mathrm{un}}}	
\newcommand{\goesnorm}{\xrightarrow{\norm{\cdot}}} 
\newcommand{\goesucc}{\xrightarrow{\mathrm{ucc}}} 
\newcommand{\goesmuB}{\xrightarrow{\mu_\mathcal B}} 

\DeclareSymbolFont{bbold}{U}{bbold}{m}{n}
\DeclareSymbolFontAlphabet{\mathbbold}{bbold}
\def\one{\mathbbold{1}}

\DeclareMathOperator{\Range}{Range}
\DeclareMathOperator{\Sol}{Sol}
\DeclareMathOperator{\Span}{span}
\DeclareMathOperator{\Int}{Int}

\renewcommand{\le}{\leqslant}
\renewcommand{\ge}{\geqslant}

\begin{document}

\title{Locally solid convergence structures}

\author[Bilokopytov]{E. Bilokopytov}
\email{bilokopy@ualberta.ca}
\address{Department of Mathematical and Statistical Sciences,
  University of Alberta, Edmonton, Canada}

\author[Conradie]{J. Conradie}
\email{jurie.conradie@uct.ac.za}
\address{Department of Mathematics and Applied Mathematics, University
  of Cape Town, Cape Town, South Africa}

\author[Troitsky]{V.G. Troitsky}
\email{troitsky@ualberta.ca}
\address{Department of Mathematical and Statistical Sciences,
  University of Alberta, Edmonton, Canada.}

\author[van der Walt]{J.H. van der Walt}
\email{janharm.vanderwalt@up.ac.za}
\address{Department of Mathematics and Applied Mathematics, University
  of Pretoria, Pretoria, South Africa}


\date{\today}

\begin{abstract}
  While there is a well developed theory of locally solid topologies,
  many important convergences in vector lattice theory are not
  topological. Yet they share many properties with locally solid
  topologies. Building upon the theory of convergence structures, we
  develop a theory of locally solid convergences, which generalize
  locally solid topologies but also includes many important
  non-topological convergences on a vector lattice. We consider some
  natural modifications of such structures: unbounded, bounded, and
  Choquet. We also study some specific convergences in vector lattices
  from the perspective of locally solid convergence structures.
\end{abstract}

\maketitle

\setcounter{tocdepth}{1} 
\tableofcontents

\section{Introduction}

In a topological space~$X$, there is a natural way to define
convergence, for both nets and filters. A net $(x_\alpha)$ converges
to a point $x\in X$ if every neighborhood of $x$ contains a tail of
the net. A filter in the power set $\mathcal P(X)$ converges to $x$ if
it contains all neighborhoods of~$x$.

There are, however, many examples of convergences occurring in
analysis that are not derived from a topology in this way. For this
reason, the concept of convergence has been extended beyond
topological spaces, leading to so called \term{convergence structures}
on sets and hence to \term{convergence spaces}. This theory can be
formulated in the language of filters, as in ~\cite{Beattie:02}, or
``translated'' into the language of nets, as was done
in~\cite{OBrien:23}.

Many useful convergences occur in the context of vector
lattices. There is a well-developed theory of topologies compatible
with the vector lattice structure, and of these the locally solid
topolgies are the most important (see, for example,
~\cite{Aliprantis:03}). There are examples of important
non-topological convergences in vector lattices that share many
properties with those of topological locally solid convergences. This
motivates the generalization of the notion of a locally solid topology
to that of a locally solid convergence. This was first explored in
~\cite{VanderWalt:13} in the context of filter convergence
structures. In this paper we continue the investigation of locally
solid convergences, mostly, but not exclusively using net convergence.

We give a brief overview of the paper. Net and filter convergence
structures are introduced and their relationship clarified in
Section~\ref{sec:prelim}, which also contains results on adherences
and uniform continuity. Basic properties of locally solid
convergences are discussed in Section~\ref{sec:loc-sol}. A thread
that runs through the paper is that of a modification of a given
convergence. We look at a number of such modifications: the unbounded
modification (Section~\ref{sec:unbdd}), the Mackey modification
(Section~\ref{sec:born-conv}), the bounded modification, also known as
specified sets convergence (Section~\ref{sec:bddmod}), the Choquet
modification (Section~\ref{sec:Choquet}) and, in various places
throughout the paper, the topological modification.

A further concept that plays an important role in some of the
convergences we discuss is that of a linear bornology; one can think of
this as a generalization of the collection of bounded sets in a
topological vector space. Bounded sets in convergence vector spaces
are defined and characterized in Section~\ref{sec:bddsets}. In
Section~\ref{sec:born-conv}, we investigate convergences induced by
linear bornologies.

Two specific convergence structures are discussed in more detail:
relative uniform convergence in Section~\ref{sec:ru} and
finite-dimensional convergence in Section~\ref{sec:born-conv}. We investigate
properties of relative uniform convergence, and study its relationship
to uniform convergence on compact sets in certain spaces of continuous
functions. Finite-dimensional convergence can be defined on any vector
space $X$ ($x_\alpha\goesfd x$ if a tail of $(x_\alpha)$ is contained
in a finite-dimensional subspace of $X$ and converges to $x$ there in
the usual sense) and is the finest linear convergence on~$X$. On
vector lattices, the Archimedean property can be characterized in
terms of it.

In Section~\ref{sec:minimal}, we consider minimal elements in the set
of all Hausdorff locally solid convergences on a vector lattice and
list a number of open problems. Finally, in Section~\ref{sec:u-ideal},
we show that an ideal $J$ is a projection band if for some
(equivalently, for all) complete locally solid
convergence structure, the operation of taking the unbounded
modification commutes with that of taking the restriction to~$J$; this
answers a question from~\cite{Kandic:17}.

\section{Preliminaries}
\label{sec:prelim}
\subsection*{Nets}
Let $X$ be a set. A net in $X$ is a function $x\colon A\to X$, where
$A$ is a directed set; we denote this net by $(x_\alpha)_{\alpha\in A}$, or simply by $(x_\alpha)$. We use the
convention of assuming that if the index set of a net is not
explicitly given, it is the capital Greek letter corresponding to the
lower case Greek letter used as subscript for the terms of the
net. Thus, for example, the index set for the net $(x_\beta)$ is
assumed to be~$B$, and that for the net $(x_\gamma)$ is~$\Gamma$.

A tail set of the net $(x_\alpha)$ is a set of the form
$\{x_\alpha\mid\alpha\ge \alpha_0\}$ for some $\alpha_0\in A$. We say
that a net $(y_\beta)$ is a \term{quasi-subnet} of $(x_\alpha)$ if
every tail set of $(x_\alpha)$ contains a tail set of $(y_\beta)$ as a
subset. Note that $(x_\alpha)$
and $(y_\beta)$ may have different index sets. Two nets $(x_\alpha)$
and $(y_\beta)$ are \term{tail equivalent} if each is a quasi-subnet
of the other.

\subsection*{Convergence structures}
We define a \term{net convergence structure} on a set $X$ by
specifying, for each $x\in X$, the nets $(x_\alpha)$ of elements of
$X$ which converge to~$x$; we denote this convergence by
$x_\alpha\to x$. To qualify as a net convergence structure, the
following natural axioms must be satisfied:
\begin{enumerate}
\item Constant nets converge: if $x_\alpha=x$ for every
  $\alpha$ then $x_\alpha\to x$;
\item If a net converges to $x$ then every
  quasi-subnet of it converges to~$x$;
\item  Suppose that $(x_\alpha)_{\alpha\in A}$
  and $(y_\alpha)_{\alpha\in A}$ both converge to~$x$. Let
  $(z_\alpha)_{\alpha\in A}$ be a net in $X$ such that
  $z_\alpha\in\{x_\alpha,y_\alpha\}$ for every~$\alpha$. Then
  $z_\alpha\to x$.
\end{enumerate}
It is easy to see that tail equivalent nets have the same limits.
When dealing with a pair of nets, we will often use the trick
described in Remark~2.5 in~\cite{OBrien:23} to replace them with a
pair of nets with the same index set, which are tail equivalent to the
original nets.

\medskip

A \term{filter convergence structure} on $X$ is defined by specifying, for each $x\in X$, the filters $\mathcal F$ in $\mathcal
P(X)$ which converge to~$x$; we denote this convergence by
$\mathcal F\to x$. To qualify as a filter convergence structure, the
following axioms must be satisfied:
\begin{enumerate}
  \item $[x]\to x$ for every $x\in X$, where $[x]=\{F\subseteq X\mid x\in F\}$ is the principal ultrafilter of~$x$;
  \item  If $\mathcal F\to x$ and $\mathcal F\subseteq\mathcal
    G$ then $\mathcal G\to x$;
  \item If $\mathcal F\to x$ and $\mathcal G\to x$ then
    $\mathcal F\cap\mathcal G\to x$.
\end{enumerate}

\medskip

Filters and nets in a set $X$ are related in a natural way: for any
net $(x_\alpha)$ its tail sets form a base of a filter, called the
\term{tail filter} of $(x_\alpha)$ and denoted by $[x_\alpha]$, while
every filter in $\mathcal P(X)$ is the tail filter of some net. It was
shown in~\cite{OBrien:23} that there is a one-to-one correspondence
between filter convergence structures and net convergence
structures. Namely, given a filter convergence structure, one may
define a net convergence structure on the same space by requiring that
a net $(x_\alpha)$ converges to $x$ whenever its tail filter
$[x_\alpha]$ converges to~$x$. Conversely, given a net convergence
structure, we define a filter convergence structure as follows: a
filter $\mathcal F$ converges to $x$ if there is a net $(x_\alpha)$
such that $x_\alpha\to x$ and $\mathcal F=[x_\alpha]$. This allows us
to just talk about \term{convergence structures}.  We will use net and
filter terminologies interchangeably.  When dealing with more than one
convergence structure on $X$ at once, we will give names to them. For
example, if we denote a convergence structure on $X$ by $\lambda$ then
we write $x_\alpha\goesl x$ instead of just  $x_\alpha\to x$, and we call the
resulting pair $(X,\lambda)$ a \term{convergence space}. We refer the reader to~\cite{Beattie:02,OBrien:23} for unexplained
notation and terminology.

Given two convergence structures $\lambda$ and $\eta$ on~$X$, we say
that $\lambda$ is stronger than $\eta$ (or $\eta$ is weaker that
$\lambda$) and write $\lambda\ge\eta$ if $x_\alpha\goesl x$ implies
$x_\alpha\goeseta x$ for every net $(x_\alpha)$ in $X$ or,
equivalently, $\mathcal F\goesl x$ implies $\mathcal F\goeseta x$ for
every filter $\mathcal F$ on~$X$.  For a filter base $\mathcal B$ of
subsets of~$X$, we write $[\mathcal B]$ for the filter generated
by~$\mathcal B$.  Recall that if $f\colon X\to Y$ is a function and
$\mathcal F$ is a filter on $X$ then we write $f(\mathcal F)$ for the
filter generated by $\bigl\{f(A)\mid A\in\mathcal F\bigr\}$. For a net
$(x_\alpha)$ in $X$ we have
$f\bigl([x_\alpha]\bigr)=\bigl[f(x_\alpha)\bigr]$. A function $f$
between two convergence spaces is \term{continuous} if $x_\alpha\to x$
implies $f(x_\alpha)\to f(x)$ for every net
$(x_\alpha)$. Equivalently, if $\mathcal F\to x$ implies
$f(\mathcal F)\to f(x)$ for every filter~$\mathcal F$. A convergence
structure on a vector space is \term{linear} if the vector space
operations are continuous; we call such a space a \term{convergence
  vector space}.

\subsection*{Closure and adherence}
Recall that a set $A$ in a convergence space $(X,\eta)$ is
\term{closed} if it contains the limit of each convergent net in
it. For a set~$A$, the \term{closure} of $A$ is defined as the
intersection of all the closed sets containing~$A$, or, equivalently,
the least closed set containing~$A$. We denote the closure of $A$ by
$\overline{A}^\eta$ or just by~$\overline{A}$. By the \term{adherence}
of $A$ we mean the set of limits of all the nets in~$A$. We write
$\overline{A}^{1,\eta}$ or just $\overline{A}^1$ for the adherence
of~$A$. Note that we always have
$A\subseteq\overline{A}^1\subseteq\overline{A}$. Further, we write
$\overline{A}^2$ for the adherence of $\overline{A}^1$, etc. In
general, for an ordinal~$\kappa$, we define $\overline{A}^{\kappa+1}$
as the adherence of $\overline{A}^\kappa$; if $\kappa$ is a limit
ordinal then we define $\overline{A}^\kappa$ as the union of
$\overline{A}^\iota$ for all the ordinals $\iota$ that are less
than~$\kappa$.

\begin{proposition}\label{k-adh-clos}
  Let $A$ be a subset of~$X$. Then there exists an ordinal number
  $\kappa$ whose cardinality does not exceed that of $X$ such that
  $\overline{A}^\kappa=\overline{A}$.
\end{proposition}

\begin{proof}
  Suppose not. Let $\alpha$ be the first ordinal number whose
  cardinality is greater than that of~$X$. Then for every ordinal
  number $\kappa<\alpha$ we have 
  $\overline{A}^\kappa\subsetneq\overline{A}^{\kappa+1}$. Pick any point in
  $\overline{A}^{\kappa+1}\setminus\overline{A}^\kappa$ and denote it
  by $f(\kappa)$. Then $f$ is an injection from $[0,\alpha)$ into~$X$,
  which is impossible because the cardinality of $[0,\alpha)$ is
  greater than that of~$X$.
\end{proof}

\subsection*{Uniform continuity} In order to translate the concept of
uniform continuity of a function on a vector space into the language
of nets, we first review the filter definition (cf.~\cite{Beattie:02},
Chapter 2).  Let $X$ be a convergence vector space. For
$A\subseteq X$, we define $\Delta A\subseteq X^2$ by
$\Delta A=\bigl\{(x,y)\mid x-y\in A\bigr\}$. For a filter $\mathcal F$
on~$X$, we define $\Delta\mathcal F$ to be the filter on $X^2$
generated by $\{\Delta A\mid A\in\mathcal F\}$. We define the
\term{induced uniform convergence structure} $\mathfrak A$ as the set
of filters $\mathcal U$ on $X^2$ such that
$\mathcal U\supseteq\Delta\mathcal F$ for some filter $\mathcal F$ on
$X$ with $\mathcal F\to 0$. Let $Y$ be another convergence vector
space and $\mathfrak B$ be the induced uniform convergence structure on~$Y^2$,
defined in a similar way. Let $f\colon X\to Y$. We define
$f^{(2)}\colon X^2\to Y^2$ by $f^{(2)}(x,y)=\bigl(f(x),f(y)\bigr)$. As
usual, for a filter $\mathcal U$ on~$X^2$, we write
$f^{(2)}(\mathcal U)$ for $\bigl[f^{(2)}(H)\mid H\in\mathcal
U\bigr]$. We say that $f$ is \term{uniformly continuous} if
$f^{(2)}(\mathcal U)\in\mathfrak B$ whenever
$\mathcal U\in\mathfrak A$. As before, every filter $\mathcal U$ on
$X^2$ is the tail filter of some net in~$X^2$. That is, we can find
two nets $(x_\alpha)$ and $(y_\alpha)$ in~$X$, indexed by the same
directed set, such that $\mathcal
U=\bigl[(x_\alpha,y_\alpha)\bigr]$. It follows that
$f^{(2)}(\mathcal U)=\bigl[\bigl(f(x_\alpha),f(y_\alpha)\bigr)\bigr]$.

\begin{lemma}\label{uniformity}
  Let $\mathcal U$ be a filter on $X^2$ with
  $\mathcal U=\bigl[(x_\alpha,y_\alpha)\bigr]$ as above. Then
  $\mathcal U\in\mathfrak A$ iff $x_\alpha-y_\alpha\to 0$ in~$X$.
\end{lemma}

\begin{proof}
  Suppose $\mathcal U\in\mathfrak A$. Then $\mathcal U\supseteq\Delta
  \mathcal F$ for some filter $\mathcal F$ on $X$ with $\mathcal F\to
  0$. For every $A\in\mathcal F$, we have $\Delta A\in\mathcal U$,
  hence $\Delta A$ contains a tail of
  $\bigl((x_\alpha,y_\alpha)\bigr)$. We conclude that $A$ contains a
  tail of $(x_\alpha-y_\alpha)$. This yields $A\in[x_\alpha-y_\alpha]$
  and, therefore, $\mathcal F\subseteq[x_\alpha-y_\alpha]$. It follows
  from $\mathcal F\to 0$ that $[x_\alpha-y_\alpha]\to 0$ and then
  $x_\alpha-y_\alpha\to 0$.
  
  Conversely, suppose that $x_\alpha-y_\alpha\to
  0$. Fix~$\alpha_0$. It follows from
 \begin{displaymath}
    \bigl\{(x_\alpha,y_\alpha)\mid\alpha\ge\alpha_0\bigr\}
    \subseteq\Delta\{x_\alpha-y_\alpha\mid\alpha\ge\alpha_0\}
  \end{displaymath}
  that 
  \begin{math}
    \Delta\{x_\alpha-y_\alpha\mid\alpha\ge\alpha_0\}
    \in\bigl[(x_\alpha,y_\alpha)\bigr]
  \end{math}
  and, therefore,
  \begin{math}
    \Delta[x_\alpha-y_\alpha]
    \subseteq\bigl[(x_\alpha,y_\alpha)\bigr]=\mathcal U.
  \end{math}
  Now $[x_\alpha-y_\alpha]\to 0$ implies that
  $\mathcal U\in\mathfrak A$.
\end{proof}

\begin{proposition}\label{unif-cont-nets}
  Let $f\colon X\to Y$ be a function between two convergence vector
  spaces. Then $f$ is uniformly continuous iff for every two nets
  $(x_\alpha)$ and $(y_\alpha)$ in $X$ with $x_\alpha-y_\alpha\to 0$
  we have $f(x_\alpha)-f(y_\alpha)\to 0$.
\end{proposition}

\begin{proof}
  Suppose that $f$ is uniformly continuous, and $x_\alpha-y_\alpha\to
  0$. Put $\mathcal U=\bigl[(x_\alpha,y_\alpha)\bigr]$. By the lemma,
  $\mathcal U\in\mathfrak A$. It follows that
  \begin{math}
    \bigl[\bigl(f(x_\alpha),f(y_\alpha)\bigr)\bigr]
    =f^{(2)}(\mathcal U)\in\mathfrak B,
  \end{math}
  and, by the lemma again, $f(x_\alpha)-f(y_\alpha)\to 0$. 

  To prove the converse, suppose that $\mathcal U\in\mathfrak
  A$. Represent $\mathcal U$ as $\mathcal
  U=\bigl[(x_\alpha,y_\alpha)\bigr]$ as before. By the lemma,
  $x_\alpha-y_\alpha\to 0$. By the assumption,
  $f(x_\alpha)-f(y_\alpha)\to 0$. Applying the lemma again, we conclude
  that 
  \begin{math}
    f^{(2)}(\mathcal U)=
    \bigl[\bigl(f(x_\alpha),f(y_\alpha)\bigr)\bigr]
    \in\mathfrak B.
  \end{math}
\end{proof}

\begin{remark}
  If $f$ is uniformly continuous and $(x_\alpha)_{\alpha\in A}$ and
  $(y_\beta)_{\beta\in B}$ are two nets such that
  $x_\alpha-y_\beta\to 0$, where we consider $x_\alpha-y_\beta$ as a
  net indexed by $A\times B$, then $f(x_\alpha)-f(y_\beta)\to
  0$. Indeed, define new nets indexed by $A\times B$ via
  $\tilde x_{\alpha,\beta}=x_\alpha$ and
  $\tilde y_{\alpha,\beta}=y_\beta$. Then
  $\tilde x_{\alpha,\beta}-\tilde y_{\alpha,\beta}\to 0$, hence, by
  the proposition,
  $f(\tilde x_{\alpha,\beta})-f(\tilde y_{\alpha,\beta})\to 0$, which
  just means $f(x_\alpha)-f(y_\beta)\to 0$.

  However, this condition is not equivalent to uniform
  continuity. Indeed, in~$\mathbb R$, $x_\alpha-y_\beta\to 0$ means
  that the two nets converge to the same limit, hence for every
  continuous $f\colon\mathbb R\to\mathbb R$ we have
  $f(x_\alpha)-f(y_\beta)\to 0$.
\end{remark}

\begin{remark}
  Let $\mathcal F$ be a filter on~$X$. Since $\Delta\mathcal F$ is a
  filter on~$X^2$, we can find two nets $(x_\alpha)$ and $(y_\alpha)$
  in $X$ such that $\Delta\mathcal
  F=\bigl[(x_\alpha,y_\alpha)\bigr]$. As in
  the proof of Lemma~\ref{uniformity}, one can prove 
  that $\mathcal F=[x_\alpha-y_\alpha]$.  

\end{remark}

\subsection*{Vector lattices}
For vector lattices (Riesz spaces) we will, unless otherwise
indicated, use the notation and terminology of ~\cite{Aliprantis:03},
to which we refer the reader for further information.  Throughout the
paper, all vector lattices are assumed to be Archimedean unless
specified otherwise.

\section{Locally solid convergence structures}
\label{sec:loc-sol}

Recall that a linear topology on a vector lattice $X$ is said to be
locally solid if it has a base of solid neighborhoods of zero. We will
now extend the concept of local solidity to convergence spaces. For a
filter $\mathcal F$ on~$X$, we write $\Sol(\mathcal F)$ for the filter
generated by the solid hulls of the elements of~$\mathcal F$. It is
clear that $\Sol(\mathcal F)\subseteq\mathcal F$.

\begin{theorem}\label{def-ls}
  For a linear convergence structure on a vector lattice~$X$, the
  following are equivalent:
  \begin{enumerate}
  \item\label{def-ls-fil} If $\mathcal F\to 0$ then
    there is a filter $\mathcal G$ with a base of solid sets such
    that $\mathcal G\subseteq\mathcal F$ and $\mathcal G\to 0$;
  \item\label{def-ls-fil-hull} If $\mathcal F\to 0$ then
    $\Sol(\mathcal F)\to 0$;
  \item\label{def-ls-net} If $y_\alpha\to 0$ and
  $\abs{x_\alpha}\le\abs{y_\alpha}$ for every $\alpha$ then
  $x_\alpha\to 0$; 
  \item\label{def-ls-net-dom} If $y_\gamma\to 0$ and for every
  $\gamma_0$ there exists $\alpha_0$ such that for every
  $\alpha\ge\alpha_0$ there exists $\gamma\ge\gamma_0$ with
  $\abs{x_\alpha}\le\abs{y_\gamma}$, then $x_\alpha\to 0$.
  \end{enumerate}
\end{theorem}

We call a convergence structure that satisfies any, hence all, of the
condition in the theorem \term{locally solid}. By a \term{locally solid
convergence space} we will mean a vector lattice equipped with a
locally solid convergence structure. In the rest of the
paper, we will mostly use \eqref{def-ls-net} (cf.~Definition~2.1
in~\cite{Aydin:21}). Condition \eqref{def-ls-fil} was first used
in~\cite{VanderWalt:13}. While in~\eqref{def-ls-net}, the two nets
have the same index sets, we no longer assume this in~
\eqref{def-ls-net-dom}.

\begin{proof}
  \eqref{def-ls-fil}$\Rightarrow$\eqref{def-ls-fil-hull} Let
  $\mathcal F\to 0$. Find a filter $\mathcal G\subseteq\mathcal F$
  with a basis consisting of solid sets so that $\mathcal G\to
  0$. Then $\mathcal G=\Sol(\mathcal G)\subseteq\Sol(\mathcal F)$, so
  that $\Sol(\mathcal F)\to 0$.

  \eqref{def-ls-fil-hull}$\Rightarrow$\eqref{def-ls-net-dom} It
  follows from $y_\gamma\to 0$ that $[y_\gamma]\to
  0$. By~\eqref{def-ls-fil-hull}, we have $\Sol[y_\gamma]\to 0$. It
  can be easily verified that $[x_\alpha]\supseteq\Sol[y_\gamma]$. We
  conclude that $[x_\alpha]\to 0$ and, therefore, $x_\alpha\to 0$.

  \eqref{def-ls-net-dom}$\Rightarrow$\eqref{def-ls-net} is trivial.

  \eqref{def-ls-net}$\Rightarrow$\eqref{def-ls-fil} Let
  $\mathcal F\to 0$. Find a net $(x_\alpha)_{\alpha\in A}$ such
  that $\mathcal F=[x_\alpha]$. Let
  \begin{math}
    \Gamma=\bigl\{(\alpha,z)\mid\alpha\in A,
       \abs{z}\le\abs{x_\alpha}\bigr\}.
  \end{math}
  We pre-order $\Gamma$ by the first component, making it into a
  directed set. Given $\gamma\in\Gamma$ with $\gamma=(\alpha,z)$, we
  put $u_\gamma=x_\alpha$ and $v_\gamma=z$. Then
  $\abs{v_\gamma}\le\abs{u_\gamma}$ for all $\gamma\in\Gamma$. It
  follows from $x_\alpha\to 0$ and $[u_\gamma]=[x_\alpha]$ that
  $u_\gamma\to 0$ and, therefore, $v_\gamma\to 0$. It is also easy to
  see that tails of $(v_\gamma)$ are solid sets. Let
  $\mathcal G=[v_\gamma]$. Then $\mathcal G\to 0$ and $\mathcal G$ has
  a base of solid sets. Let $U\in\mathcal G$. Then $U$ contains a tail
  $\{v_\gamma\}_{\gamma\ge\gamma_0}$, which, in turn, contains
  $\{u_\gamma\}_{\gamma\ge\gamma_0}$. The later agrees with a tail of
  $(x_\alpha)$. It follows that $U\in\mathcal F$. Hence,
  $\mathcal G\subseteq\mathcal F$.
\end{proof}

We mention a few important examples. Clearly, a linear topology on a
vector lattice $X$ is locally solid iff the corresponding convergence
structure is locally solid. In particular, norm convergence and
absolute weak (aw) convergence (see Definition~2.32
in~\cite{Aliprantis:03}) on a Banach lattice are locally solid.

One of the most important examples of a locally solid convergence is
\term{order convergence}: a net $(x_\alpha)_{\alpha\in A}$ in a vector
lattice $X$ \term{converges in order} to $x$ (written
$x_\alpha\goeso x$) if there exists a net
$(u_\gamma)_{\gamma\in\Gamma}$ such that $u_\gamma\downarrow 0$ and
for every $\gamma\in\Gamma$ there exists $\alpha_0\in\Lambda$ such
that $\abs{x_\alpha-x}\le u_\gamma$ for all $\alpha\ge\alpha_0$.  In
the language of filters, $\mathcal F\goeso x$ if $\mathcal F$ contains
a collection of order intervals whose intersection is~$\{x\}$. Order
convergence as a convergence structure was investigated in
~\cite{OBrien:23}.

Other examples of locally solid convergences include relative
uniform, uo, un, and uaw convergence. These will be defined and
discussed later throughout the paper.

The following proposition is
the net version of Proposition~2.5 in \cite{VanderWalt:13}; it
generalizes many known results for order, relative uniform, norm, uo,
un, and other specific locally solid convergences.

\begin{proposition}\label{op-cont}
  Let $(X,\to)$ be a locally solid convergence vector lattice. Then
  the maps which take $x$ to $\abs{x}$, $x^+$, and $x^-$ are
  continuous, and the maps which take $(x,y)$ to $x\vee y$ and
  $x\wedge y$ are jointly continuous.
\end{proposition}

\begin{proof}
  If $x_\alpha\to x$ then $x_\alpha-x\to 0$. By the triangle
  inequality,
  \begin{math}
    \bigabs{\abs{x_\alpha}-\abs{x}}\le\abs{x_\alpha-x},
  \end{math}
  so that $\abs{x_\alpha}-\abs{x}\to 0$ and, therefore,
  $\abs{x_\alpha}\to\abs{x}$. The proofs for $x^+$ and $x^-$ are
  similar. To prove joint continuity of $x\vee y$ and $x\wedge
  y$ one can use the identities
  \begin{displaymath}
    x\vee y=\tfrac{1}{2}(x+y)+\tfrac{1}{2}\abs{x-y}\quad\mbox{and}\quad
    x\wedge y=\tfrac{1}{2}(x+y)-\tfrac{1}{2}\abs{x-y}.
  \end{displaymath}
\end{proof}

We will now strengthen Proposition~\ref{op-cont} to uniform continuity.
Recall that the Robertson-Namioka Theorem asserts that a linear topology
on a vector lattice is locally solid iff the lattice operations are
uniformly continuous; see,
e.g.,~\cite[Theorem~2.17]{Aliprantis:03}. We will now generalize this
fact to convergence structures (cf.\ Section~3 in~\cite{Aydin:21}).

A linear convergence on a vector lattice $X$ is said to be
\term{locally full} if it satisfies the ``squeeze law'': if
$x_\alpha\le y_\alpha\le z_\alpha$ for all~$\alpha$, $x_\alpha\to a$
and $z_\alpha\to a$ then $y_\alpha\to a$; equivalently, if
$0\le x_\alpha\le y_\alpha\to 0$ implies $x_\alpha\to 0$. The latter
condition in the context of ordered convergence vector spaces
was called a \emph{full convergence} in~\cite{Aydin:21}; note
that~\cite{Aydin:21} uses a slightly different definition of a
convergence structure.

\begin{proposition}\label{mod-u-cont}
  Let $X$ be a convergence vector space and a vector lattice. The
  following are equivalent:
  \begin{enumerate}
  \item\label{uc-loc-sol} The convergence is locally solid;
  \item\label{uc-uc} The modulus operation is uniformly continuous;
  \item\label{uc-full} The modulus operation is continuous and the
    convergence is locally full.
  \end{enumerate}
\end{proposition}

\begin{proof}
  \eqref{uc-loc-sol}$\Rightarrow$\eqref{uc-uc}
  follows from Proposition~\ref{unif-cont-nets} and the triangle inequality 
  \begin{math}
    \bigabs{\abs{x_\alpha}-\abs{y_\alpha}}\le\abs{x_\alpha-y_\alpha}.
  \end{math}

  \eqref{uc-uc}$\Rightarrow$\eqref{uc-full} Suppose that
  $0\le y_\alpha\le x_\alpha\to 0$. Since the modulus is uniformly
  continuous, it follows from $x_\alpha=y_\alpha-(y_\alpha-x_\alpha)$
  that $\abs{y_\alpha}-\abs{y_\alpha-x_\alpha}\to 0$, hence
  $y_\alpha-(x_\alpha-y_\alpha)\to 0$, so that
  $2y_\alpha-x_\alpha\to 0$ and, therefore, $y_\alpha\to 0$.

  \eqref{uc-full}$\Rightarrow$\eqref{uc-loc-sol}
  Suppose that $\abs{y_\alpha}\le\abs{x_\alpha}$ for every
  $\alpha$ and $x_\alpha\to 0$. Continuity of the modulus yields
  $\abs{x_\alpha}\to 0$. It now follows from $-\abs{x_\alpha}\le
  y_\alpha\le\abs{x_\alpha}$ that $y_\alpha\to 0$.
\end{proof}

As in the proof of Proposition~\ref{op-cont}, the modulus operation in
Proposition~\ref{mod-u-cont} may be replaced with any other lattice
operation.

Recall that a convergence structure is \term{Hausdorff} when every net has at
most one limit.

\begin{proposition}\label{Haus}
  For a locally solid convergence space~$X$, the following are equivalent:
  \begin{enumerate}
  \item\label{Haus-Haus}  $X$ is Hausdorff;
  \item\label{Haus-const} Constant nets have unique limits;
  \item\label{Haus-zero} Every constant zero net converges to zero only;
  \item\label{Haus-cone} $X^+$ is closed;
  \item\label{Haus-int} Order intervals are closed;
  \item\label{Haus-mon} If $x_\alpha\to x$, $y_\alpha\to y$, and
    $x_\alpha\le y_\alpha$ for all $\alpha$ then $x\le y$.
  \end{enumerate}
\end{proposition}

\begin{proof}
  The implicatons
  \eqref{Haus-Haus}$\Rightarrow$\eqref{Haus-const}$\Leftrightarrow$\eqref{Haus-zero} and
  \eqref{Haus-cone}$\Leftrightarrow$\eqref{Haus-mon}
  are straightforward.

  \eqref{Haus-const}$\Rightarrow$\eqref{Haus-Haus} Suppose that
  $x_\alpha\to x$ and $x_\alpha\to y$. It follows from
  \begin{displaymath}
    \abs{x-y}\le\abs{x_\alpha-x}+\abs{x_\alpha-y}\to 0
  \end{displaymath}
  that $\abs{x-y}\to 0$, i.e., the constant net $\abs{x-y}$ converges
  to zero. Since we also have $\abs{x-y}\to\abs{x-y}$, our assumption
  yields $\abs{x-y}=0$ and, therefore, $x=y$.

  \eqref{Haus-zero}$\Rightarrow$\eqref{Haus-cone} Let $(x_\alpha)$ be
  a net in $X^+$ with $x_\alpha\to x$. Then $0=x_\alpha^-\to x^-$; it
  follows that $x^-=0$ and, therefore, $x\in X^+$. Hence $X^+$ is
  closed.

  \eqref{Haus-cone}$\Rightarrow$\eqref{Haus-int} because
  $[a,b]=(a+X^+)\cap(b-X^+)$. We get
  \eqref{Haus-int}$\Rightarrow$\eqref{Haus-zero} by considering the
  trivial interval $[0,0]$.
\end{proof}

For every locally solid convergence, we have $x_\alpha\to x$ iff
$\abs{x_\alpha-x}\to 0$. Thus, the convergence is determined by
positive nets converging to zero. In~\cite[Theorem~2.1]{Bilokopytov:23a},
this idea was developed as follows:

\begin{theorem}\label{plsc}
  Consider a non-empty family of nets in~$X_+$; we write
  $x_\alpha\to 0$ to indicate that $(x_\alpha)$ is in the
  family. Suppose that the following axioms are satisfied:
  \begin{enumerate}
  \item\label{plsc-sub} if $x_\alpha\to 0$ and $(y_\beta)$ is a quasi-subnet of
    $(x_\alpha)$ then $y_\beta\to 0$;
  \item\label{plsc-dom} If $x_\alpha\to 0$ and $0\le y_\alpha\le x_\alpha$ for every
    $\alpha$ then $y_\alpha\to 0$;
  \item\label{plsc-sum} If $x_\alpha\to 0$ and $y_\alpha\to 0$ then
    $x_\alpha+y_\alpha\to 0$;
  \item\label{plsc-1n} $\frac1n x\to 0$ for every $x\in X_+$.
  \end{enumerate}
  Then $\to$ extends to a locally solid convergence structure on all
  of $X$ via $x_\alpha\to x$ if $\abs{x_\alpha-x}\to 0$. In this case,
  the resulting convergence on $X$ is Hausdorff iff no constant net in
  $X_+\setminus\{0\}$ converges to zero.
\end{theorem}

\subsection*{Properties of locally solid convergences}

\begin{proposition}\label{MCT-dom}
  Let $(x_\alpha)$ be a net in a Hausdorff locally solid
  convergence space~$X$. If $x_\alpha\le y$ for all $\alpha$ and
  $x_\alpha\to y$ then $y=\sup x_\alpha$.
\end{proposition}

\begin{proof}
  Clearly, $y$ is an upper bound of the net. If $z$ is another upper
  bound then, passing to the limit, we get $y\le z$ by Proposition~\ref{Haus}.
\end{proof}

The following is motivated by the definition of order convergence; it
follows immediately from Theorem~\ref{def-ls}\eqref{def-ls-net-dom}:

\begin{lemma}\label{net-dom}
  Let $(x_\alpha)_{\alpha\in A}$ and
  $(y_\gamma)_{\gamma\in\Gamma}$ be two nets in a locally solid
  convergence space~$X$. Suppose that $y_\gamma\to 0$ and for every
  $\gamma\in\Gamma$ there exists $\alpha_0\in A$ such that
  $\abs{x_\alpha}\le\abs{y_\gamma}$ whenever $\alpha\ge\alpha_0$. Then
  $x_\alpha\to 0$.
\end{lemma}

The next three statements are concerned with monotone nets.  Note that
$x_\alpha\uparrow x$ need not imply $x_\alpha\to x$.  For example, the
implication fails for norm convergence on
$\ell_\infty$. Theorem~\ref{ocn-mon} will provide a condition under
which the statement will be true.

\begin{proposition}[\cite{VanderWalt:13}]\label{MCT}
  Let $(x_\alpha)$ be an increasing net in a Hausdorff locally solid
  convergence space~$X$. Then $x_\alpha\to x$ implies
  $x_\alpha\uparrow x$.
\end{proposition}

\begin{proof}
  Fix any $\alpha_0$. Then $x_\alpha\ge x_{\alpha_0}$ for every
  $\alpha\ge\alpha_0$, hence $x\ge x_{\alpha_0}$. It follows that $x$
  is an upper bound of the net. By Proposition~\ref{MCT-dom}, we have
  $x_\alpha\uparrow x$.
\end{proof}

Note that in the following two results we do not assume $X$ to be
Hausdorff.

\begin{proposition}\label{sset-mon-conv}
  Let $X$ be a locally solid convergence space, and
  $(x_\alpha)_{\alpha\in A}$ and $(y_\gamma)_{\gamma\in\Gamma}$ two
  nets in $X_+$ such that $x_\alpha\!\downarrow$,
  $\{y_\gamma\mid\gamma\in\Gamma\}\subseteq\{x_\alpha\mid\alpha\in
  A\}$, and $y_\gamma\to 0$. Then $x_\alpha\to 0$.
\end{proposition}

\begin{proof}
  It suffices to show that the assumptions of
  Lemma~\ref{net-dom} are satisfied. For every
  $\gamma\in\Gamma$ there exists $\alpha_0\in A$ such that
  $y_\gamma=x_{\alpha_0}$. It follows that for all $\alpha\ge\alpha_0$
  we have
  \begin{math}
    0\le x_\alpha\le x_{\alpha_0}=y_\gamma.
  \end{math}
\end{proof}

\begin{corollary}\label{qsnet-mon-conv}
  Let $X$ be a locally solid convergence space. If a
  monotone net in $X$ has a convergent quasi-subnet then the entire
  net converges to the same limit.
\end{corollary}

\begin{proof}
  Suppose that $(x_\alpha)$ is decreasing, $(y_\gamma)$ is a
  quasi-subnet of~$(x_\alpha)$, and $y_\gamma\to x$. Subtracting $x$
  and passing to a tail of~$(y_\gamma)$, we may assume that
  $y_\gamma\to 0$ and
  $\{y_\gamma\mid\gamma\in\Gamma\}\subseteq\{x_\alpha\mid\alpha\in
  A\}$. From the continuity of lattice operations, we get
  $y_\gamma^+\to 0$ and $y_\gamma^-\to 0$. It suffices to prove that
  $x_\alpha^+\to 0$ and $x_\alpha^-\to 0$.

  Applying Proposition~\ref{sset-mon-conv} to the nets $(x_\alpha^+)$
  and $(y_\gamma^+)$, we conclude that $x_\alpha^+\to 0$. Note that
  $x_\alpha^-\uparrow$. It is enough to show that the nets
  $(x_\alpha^-)$ and $(y_\gamma^-)$ satisfy the assumptions of
  Theorem~\ref{def-ls}\eqref{def-ls-net-dom}. Fix~$\gamma_0$. Find
  $\alpha_0$ such that $x_{\alpha_0}=y_{\gamma_0}$. Take
  $\alpha\ge\alpha_0$. Since $(y_\gamma)$ is a quasi-subnet
  of~$(x_\alpha)$, we can find $\gamma_1$ such that
  $\{y_\gamma\mid\gamma\ge\gamma_1\}\subseteq\{x_\beta\mid\beta\ge\alpha\}$. Take
  $\gamma\ge\gamma_0,\gamma_1$. Then there exists $\beta\ge\alpha$
  such that $y_\gamma^-=x_\beta^-\ge x_\alpha^-\ge 0$.
\end{proof}

A locally solid convergence space $(X,\eta)$ is \term{order
  continuous} or \term{Lebesgue} if $x_\alpha\goeso 0$ implies
$x_\alpha\goeseta 0$ for every net~$(x_\alpha)$, that is,
$\eta\le\mathrm{o}$. The following result, which originally appeared
as Theorem~4.1 in~\cite{Bilokopytov:23a}, follows immediately from
Lemma~\ref{net-dom}.

\begin{theorem}\label{ocn-mon}
  A locally solid convergence space $(X,\eta)$ is order continuous iff
  $x_\alpha\downarrow 0$ implies $x_\alpha\goeseta 0$ for every
  net~$(x_\alpha)$.
\end{theorem}

\begin{proposition}\label{lip}
  Let $X$ be a Hausdorff locally solid convergence space. If every order
  interval in $X$ is compact then $X$ is order complete and order continuous.
\end{proposition}

\begin{proof}
  Suppose that $0\le x_\alpha\downarrow$. Without loss of generality,
  $(x_\alpha)$ is contained in an order interval. Then it has a
  convergent quasi-subnet, say, $y_\gamma\to x$. By
  Corollary~\ref{qsnet-mon-conv}, $x_\alpha\to
  x$. Proposition~\ref{MCT} yields $x=\inf x_\alpha$. This proves that
  $X$ is order complete. If $x_\alpha\downarrow 0$ then the preceding
  argument yields $x_\alpha\to x$ with $x=\inf x_\alpha=0$, hence $X$
  is order continuous.
\end{proof}

For topologies, Proposition~\ref{lip} was proved in~\cite{Lipecki:22}.

\begin{proposition}\label{adh-sublat}
  Let $Y$ be a sublattice of a locally solid Hausdorff convergence vector
  lattice~$X$. Then the adherence $\overline{Y}^1$ is again
  a sublattice and $\overline{(Y_+)}^1=(\overline{Y}^1)_+$.
\end{proposition}

\begin{proof}
  $\overline{Y}^1$ is a sublattice and
  $(\overline{Y}^1)_+\subseteq\overline{(Y_+)}^1$ because
  lattice operations are continuous. Suppose now that
  $x\in \overline{(Y_+)}^1$. Then there exists a net $(y_\alpha)$
  in $Y_+$ such that $y_\alpha\to x$. It follows that
  $y_\alpha=y_\alpha^+\to x^+$. Since $X$ is Hausdorff,
  $x=x^+$ and, therefore, $x\ge 0$. It follows that
  $x\in (\overline{Y}^1)_+$.
\end{proof}

The following is a consequence of the continuity of lattice
operations:

\begin{proposition}\label{Jd-closed}
  Let $X$ be a Hausdorff locally solid convergence space. For every
  set $A$ in~$X$, its disjoint complement $A^d$ is closed. It follows
  that every band in $X$ is closed.
\end{proposition}

The proof of the following statement is routine.

\begin{proposition}[\cite{Bilokopytov:23a}]\label{conv-ideal}
  Let $X$ be a locally solid convergence space and
  $(x_\alpha)$ a net in~$X_+$. The set
  \begin{math}
    \bigl\{v\in X\mid \abs{v}\wedge x_\alpha\to 0\bigr\} 
  \end{math}
  is an ideal.
\end{proposition}

\begin{example}
  Recall that a net $(x_\alpha)$ in a vector lattice $X$
  \term{$\sigma$-order converges} to~$x$, written $x_\alpha\goesso x$,
  if there exists a sequence $(u_n)$ in $X_+$ such that
  $u_n\downarrow 0$ and for every $n$ there exists $\alpha_n$ such
  that $\abs{x_\alpha-x}\le u_n$ whenever $\alpha\ge\alpha_n$.  In
  filter language, $\mathcal F\goesso x$ if $\mathcal F$ contains a
  sequence of order intervals whose intersection is~$\{x\}$. It is
  easy to see that this is a Hausdorff locally solid convergence
  structure and it is stronger than order convergence. If $X$ has the
  countable sup property then $\sigma\mathrm{o}$ convergence agrees
  with order convergence. The convergence structure $\sigma\mathrm{o}$
  was investigated in \cite{Anguelov:05}.
\end{example}

Let $X$ be a convergence space and $Y$ a convergence vector space. We
may equip the vector space $C(X,Y)$ of all continuous functions from
$X$ to $Y$ with the convergence structure of \term{continuous
  convergence}: $f_\alpha\to f$ if for every net $(x_\gamma)$ in~$X$,
if $x_\gamma\to x$ then the double net
$\bigl(f_\alpha(x_\gamma)\bigr)_{(\alpha,\gamma)}$ converges to
$f(x)$. In the language of filters, a filter $\Phi$ in $C(X,Y)$
converges to $f\in C(X,Y)$ if $\Phi[\mathcal F]\to f(x)$ whenever
$\mathcal F\to x$ in~$X$; see~\cite{OBrien:23} for details.  It was
observed in~\cite[Proposition~8.1]{OBrien:23} that if $Y$ is a
complete Hausdorff topological vector space then continuous
convergence on $C(X,Y)$ is complete. We now show that local solidness
also passes to $C(X,Y)$:

\begin{proposition}
  Let $X$ be a convergence space and $Y$ a locally solid convergence
  space. Then continuous convergence on $C(X,Y)$ is locally solid.
\end{proposition}

\begin{proof}
  It is easy to see that $C(X,Y)$ is a vector lattice under the
  point-wise order. Suppose that $f_\alpha\to 0$ and
  $\abs{g_\alpha}\le\abs{f_\alpha}$ for every~$\alpha$. If
  $x_\gamma\to x$ then $f_\alpha(x_\gamma)\to 0$; it follows from
  \begin{math}
    \bigabs{g_\alpha(x_\gamma)}\le\bigabs{f_\alpha(x_\gamma)}\to 0
  \end{math}
  that $g_\alpha(x_\gamma)\to 0$ and, therefore, $g_\alpha\to 0$.
\end{proof}

\section{Unbounded modification}
\label{sec:unbdd}

Let $(X,\eta)$ be a locally solid convergence vector
lattice. Define a new convergence structure $\mathrm{u}\eta$ on $X$ as
follows: $x_\alpha\goesueta x$ if $\abs{x_\alpha-x}\wedge u\goeseta 0$
for every $u\in X_+$. It is easy to see that $\goesueta$ is again a
locally solid convergence structure. We call it the \term{unbounded
  modification} of~$\eta$, or just unbounded $\eta$-convergence.
Clearly, $\eta$ and $\mathrm{u}\eta$ agree on order bounded sets.
It follows that unbounded modification is idempotent, i.e.,
$\mathrm{u}(\mathrm{u}\eta)=\mathrm{u}\eta$.

Unbounded order convergence (i.e., the unbounded modification of order
convergence), abbreviated as uo-convergence, has been extensively
studied; see
\cite{Nakano:48,Papangelou:64,Kaplan:97,Fremlin:04,Gao:17} and
references therein. The unbounded modification of norm convergence in
a Banach lattice, called un-convergence, was studied in
\cite{Troitsky:04,Deng:17,Kandic:17}. The unbounded modification of
absolute weak convergence in Banach lattices was studied
in~\cite{Zabeti:18} under the name of uaw-convergence. Note that the
unbounded modification of a convergence induced by locally solid
topology is again topological; in particular, un-convergence and
uaw-convergence are topological. M.Taylor
in~\cite{Taylor:18,Taylor:19} studied unbounded modifications of
locally solid topologies on vector lattices. Unbounded modifications
of convergence structures were studied in~\cite{Aydin:21} under a
slightly different definition of a convergent structure.

More generally, let $\eta$ be a locally solid convergence structure
on~$X$, $I$ an ideal in~$X$, and $(x_\alpha)$ a net in~$X$; we say
that $(x_\alpha)$ unboundedly $\eta$-converges with respect to~$I$,
and write $x_\alpha\goesuIe x$ if $\abs{x_\alpha-x}\wedge u\goeseta 0$
for every $u\in I_+$. It follows from Theorem~\ref{plsc} (or it can be
verified directly) that $\mathrm{u}_I\eta$ is again a locally solid
convergence structure. Hence, by Proposition~\ref{op-cont}, lattice
operations are $\mathrm{u}_I\eta$-continuous. We clearly have
$\mathrm{u}_I\eta\le\mathrm{u}\eta\le\eta$ and
$\mathrm{u}_I(\mathrm{u}_I\eta)=\mathrm{u}_I\eta$. In the special case
when $I=X$, we get $\mathrm{u}_I\eta=\mathrm{u}\eta$. We require $I$
to be an ideal and not just a subset of $X$ because if $A\subseteq X$
is an arbitrary subset then, by Proposition~\ref{conv-ideal}, the
condition
\begin{displaymath}
  \forall u\in A\quad\abs{x_\alpha-x}\wedge\abs{u}\goeseta 0
\end{displaymath}
is equivalent to $x_\alpha\goesuIe 0$ where $I=I(A)$ is the ideal
generated by~$A$.

\begin{proposition}
  Let $\eta$ be a locally solid convergence on~$X$. Then
  $\mathrm{u}_I\eta$ is Hausdorff iff $\eta$ is Hausdorff and $I$ is
  order dense. In particular, if $\eta$ is Hausdorff then
  $\mathrm{u}\eta$ is Hausdorff.
\end{proposition}

\begin{proof}
  Suppose that $\eta$ is Hausdorff and $I$ is order dense. We will
  show that $u_I\eta$ is Hausdorff.  Suppose that $(x_\alpha)$ is a
  constant zero net and $x_\alpha\goesuIe x$. According to
  Proposition~\ref{Haus}, it suffices to show that then $x=0$. Let
  $u\in I^+$. By assumption, the constant net $\abs{x}\wedge u$ is
  $\eta$-convergent to zero. Since $\eta$ is Hausdorff, we have
  $\abs{x}\wedge u=0$. Since $u\in I^+$ is arbitrary and $I$ is order
  dense, we conclude that $x=0$.

  Suppose now that $\mathrm{u}_I\eta$ is Hausdorff. Then $\eta$ is
  Hausdorff because it is a stronger convergence.  Suppose that
  $I$ fails to be order dense. There exists $x>0$ in $X$ with
  $x\perp I$. It follows that every constant net with $x$ as the
  constant term $\mathrm{u}_I\eta$-converges to zero, hence
  $\mathrm{u}_I\eta$ fails to be Hausdorff.
\end{proof}

\begin{proposition}[\cite{Bilokopytov:23a}]\label{unb-equiv}
  Let $\eta$ be a Hausdorff locally solid convergence on~$X$. Then
    $x_\alpha\goesueta x$ iff $a\vee(x_\alpha\wedge b)\goeseta
    a\vee(x\wedge b)$ for all $a,b\in X$.
\end{proposition}

\begin{proof}
  If $x_\alpha\goesueta x$ and $a,b\in X$, continuity of lattice
  operations yields $a\vee(x_\alpha\wedge b)\goesueta a\vee(x\wedge
  b)$. Since the net $\bigl(a\vee(x_\alpha\wedge b)\bigr)$ is order
  bounded, $\mathrm{u}\eta$ and $\eta$ agree on it.
  %
  %

  The converse follows from the fact that for every $u\ge 0$ we have
  \begin{multline*}
    \abs{x_\alpha-x}\wedge u
    =\Bigabs{(x-u)\vee\bigl(x_\alpha\wedge(x+u)\bigr)-x}\\
    =\Bigabs{(x-u)\vee\bigl(x_\alpha\wedge(x+u)\bigr)-
      (x-u)\vee\bigl(x\wedge(x+u)\bigr)}\goeseta 0.
  \end{multline*}
\end{proof}

Recall that $\overline{A}^\eta$ and $\overline{A}^{1,\eta}$ stand for
the closure and the adherence of $A$ with respect to~$\eta$,
respectively, while $\overline{A}^{2,\eta}$ is the adherence of
$\overline{A}^{1,\eta}$. The proof of the following proposition was
motivated by an unpublished manuscript by M.~Taylor and by
Proposition~4.1 in~\cite{Aydin:21}.

\begin{proposition}\label{sublat-adh}
  Let $Y$ be a sublattice of $X$ and $\eta$ a locally solid Hausdorff
  convergence structure on~$X$. Then
  \begin{math}
    \overline{Y}^{1,\eta}\subseteq\overline{Y}^{1,\mathrm{u}\eta}
    \subseteq\overline{Y}^{2,\eta}.
  \end{math}
  Furthermore, $\overline{Y}^\eta=\overline{Y}^{\mathrm{u}\eta}$.
\end{proposition}

\begin{proof}
  By Proposition~\ref{adh-sublat}, the adherences are sublattices.
  Since $\mathrm{u}\eta\le\eta$, we clearly have
  $\overline{Y}^{1,\eta}\subseteq\overline{Y}^{1,\mathrm{u}\eta}$. Let
  \begin{math}
    0\le x\in\overline{Y}^{1,\mathrm{u}\eta}.
  \end{math}
  By Proposition~\ref{adh-sublat}, there exists a net $(y_\alpha)$ in
  $Y_+$ such that $y_\alpha\goesueta x$.
  Proposition~\ref{unb-equiv} yields
  $u\wedge y_\alpha\goeseta u\wedge x$ for every
  $u\in X_+$. For every $y\in Y_+$, we have $y\wedge y_\alpha\in Y$
  for every~$\alpha$; it follows from
  $y\wedge y_\alpha\goeseta y\wedge x$ that
  $y\wedge x\in\overline{Y}^{1,\eta}$. In particular,
  $y_\alpha\wedge x\in \overline{Y}^{1,\eta}$ for every~$\alpha$. It
  now follows from $y_\alpha\wedge x\goeseta x\wedge x=x$ that
  $x\in\overline{Y}^{2,\eta}$. Hence,
  $\overline{Y}^{u\eta}_+\subseteq\overline{Y}^{2,\eta}$.
  Proposition~\ref{adh-sublat} yields
  \begin{math}
    \overline{Y}^{u\eta}
    \subseteq\overline{Y}^{2,\eta}
  \end{math}  

  It follows that $\eta$ and $\mathrm{u}\eta$ have the same closed
  sublattices. Thus, $\overline{Y}^\eta=\overline{Y}^{\mathrm{u}\eta}$.
\end{proof}

\begin{question}
  Lemma~2.6 of~\cite{Gao:18a} provides an example that the first
  inclusion in Proposition~\ref{sublat-adh} may be proper. Is there 
  an example where the second inclusion is proper?
\end{question}

We now present a filter characterization of unbounded
modifications. For $u\in X_+$ and $A\subseteq X$, we write
$\abs{A}\wedge u=\bigl\{\abs{a}\wedge u\mid a\in A\bigr\}$. For a
filter $\mathcal F$ in~$X$, we write $\abs{\mathcal F}\wedge u$ for
the filter generated by $\bigl\{\abs{A}\wedge u\mid A\in\mathcal
F\bigr\}$. Note that $\abs{\mathcal F}\wedge u=f(\mathcal F)$ where
$f\colon X\to X$ is given by $f(x)=\abs{x}\wedge u$. In particular, 
  \begin{math}
    \bigabs{[x_\alpha]}\wedge u=\bigl[\abs{x_\alpha}\wedge u\bigr]
  \end{math}
for every net $(x_\alpha)$ in~$X$.
\begin{proposition}\label{filter-u}
  Let $\eta$ be a locally solid convergence structure on~$X$, $I$ an
  ideal in~$X$, and $\mathcal F$ a filter on~$X$. Then $\mathcal
  F\goesuIe 0$ iff $\abs{\mathcal F}\wedge u\goeseta 0$ for every
  $u\in I_+$.
\end{proposition}

\begin{proof}
  Find a net $(x_\alpha)$ with $\mathcal F=[x_\alpha]$. Then
  \begin{multline*}
    \mathcal F\goesuIe 0\qquad\Leftrightarrow\qquad
    x_\alpha\goesuIe 0\qquad\Leftrightarrow\qquad
    \forall u\in I_+\quad\abs{x_\alpha}\wedge u\goeseta 0\\
    \qquad\Leftrightarrow\qquad
    \forall u\in I_+\quad\bigl[\abs{x_\alpha}\wedge u\bigr]\goeseta 0
    \qquad\Leftrightarrow\qquad
    \forall u\in I_+\quad\abs{\mathcal F}\wedge u\goeseta 0.
  \end{multline*}
\end{proof}

Here is a sketch of an alternative proof of
Proposition~\ref{filter-u} in the special case $I=X$. It follows from
Proposition~\ref{unb-equiv} that $\mathrm{u}\eta$ is the weakest
convergence that makes all the functions of the form
$x\mapsto a\vee(x\wedge b)$ continuous as maps to $(X,\eta)$. Thus,
$\mathcal F\goesueta x$ iff $a\vee(\mathcal F\wedge b)\goeseta x$ for
all $a$ and~$b$. The latter is equivalent to
$\abs{\mathcal F}\wedge u\goeseta 0$ for all $u\ge 0$.

\begin{remark}\label{fil-uA}
  Just as we did for nets, one can consider unbounded convergence of filters
  with respect to an arbitrary subset $A$ instead of an ideal. And,
  just as for nets, this is easily seen to be equivalent to
  unbounding with respect to $I(A)$.
\end{remark}

\begin{proposition}\label{ub-filter-mon}
  Let $\mathcal F$ be a filter with a base of solid sets. Then
  $\mathcal F\subseteq\abs{\mathcal F}\wedge u$ for every $u\ge
  0$. If $0\le v\le u$ then
  $\abs{\mathcal F}\wedge u\subseteq\abs{\mathcal F}\wedge v$. 
\end{proposition}

\begin{proof}
  Let $\mathcal B$ be a base of $\mathcal F$ consisting of solid
  sets. If $A\in\mathcal F$ then there exists $B\in\mathcal B$ with
  $A\supseteq B\supseteq\abs{B}\wedge u\in\abs{\mathcal F}\wedge u$; it
  follows that $A\in\abs{\mathcal F}\wedge u$ and, therefore,
  $\mathcal F\subseteq\abs{\mathcal F}\wedge u$.

  Moreover, since
  $\abs{B}\wedge v\subseteq\abs{B}\wedge u\subseteq\abs{A}\wedge u$
  and $\abs{B}\wedge v\in\abs{\mathcal F}\wedge v$, it follows that
  $\abs{A}\wedge u\in \abs{\mathcal F}\wedge v$ and, therefore,
  $\abs{\mathcal F}\wedge u\subseteq\abs{\mathcal F}\wedge v$.
\end{proof}

A locally solid convergence is said to be
\term{unbounded} if it equals its unbounded modification or,
equivalently, if it is the unbounded modification of some locally
solid convergence. Obviously, uo and un convergences are unbounded.

Let $X$ be a convergence space. We claim that the continuous
convergence on $C(X)$ is unbounded. In fact, a more general statement
is true:

\begin{proposition}
  Let $X$ be a convergence space and $Y$ an unbounded locally solid
  convergence vector lattice. Then continuous convergence on
  $C(X,Y)$ is unbounded.
\end{proposition}

\begin{proof}
  Suppose that $f_\alpha\goesuc 0$ in $C(X,Y)$; we need to show that
  $f_\alpha\goesc 0$. Fix $y\in Y_+$ and let $g\in C(X,Y)$ be the
  constant $y$ function. Then $\abs{f_\alpha}\wedge g\goesc
  0$. Suppose that $x_\gamma\to x$ in~$X$. Then
  $\abs{f_\alpha(x_\gamma)}\wedge y\to 0$. Since $y$ is
  arbitrary and the convergence in $Y$ is unbounded, we conclude that
  $f_\alpha(x_\gamma)\to 0$ and, therefore, $f_\alpha\goesc 0$.
\end{proof}

We defined a locally solid convergence $\eta$ to be unbounded if
$\mathrm{u}\eta=\eta$.
More generally, given a locally solid convergence~$\eta$,
one can ask when $\mathrm{u}_I\eta=\eta$ for
some ideal~$I$. In Section~\ref{sec:u-ideal}, we will provide a
partial answer to this question (Proposition~\ref{uI-unbdd}), as well
as some additional properties of unbounded convergence structures.

\subsection*{Some properties of uo-convergence}
Recall that a sublattice $Y$ of $X$ is regular iff it is a subspace of
$X$ in the sense of the uo-convergence structure, that is,
$x_\alpha\goesuo x$ in $Y$ iff $x_\alpha\goesuo x$ in $X$ for every
net $(x_\alpha)$ in $Y$ and every $x\in Y$; see
\cite[Theorem~7.5]{Papangelou:64} or
\cite[Theorem~3.2]{Gao:17}.

If $X$ is universally complete then
order convergence and unbounded order convergence agree on sequences
in~$X$, see, e.g., Theorem~3.2 in~\cite{Kaplan:97}. However, this
fails for nets, even for double sequences:

\begin{theorem}[\cite{Taylor:TH}]\label{uo-o-fin-dim}
  Uo-convergence and order convergence agree
  iff $X$ is finite-dimensional.
\end{theorem}

\begin{proof}
  It is easy to see that if $X$ is
  finite-dimensional then the two convergences agree. Suppose that $X$
  is infinite-dimensional. Let $(x_n)$ be a disjoint sequence of
  non-zero vectors. Define a double sequence $y_{n,m}=mx_n$. Since $X$
  is Archimedean, no tail of this net is order bounded, hence it fails
  to converge in order. We claim that it is uo-null.

  Replacing $X$ with~$X^\delta$, we may assume without loss of
  generality that $X$ is order complete.  Fix $u\in X_+$. For
  every~$n$, consider the band projection of $u$ onto the band
  $B_{x_n}$; it is given by $P_{x_n}u=\sup_m(u\wedge mx_n)$.  The
  sequence $(P_{x_n}u)$ is disjoint and order bounded by~$u$, hence
  $P_{x_n}u\goeso 0$. On the other hand, $y_{n,m}\wedge u\le P_{x_n}u$
  for all $n$ and~$m$; it follows that $y_{n,m}\wedge u\goeso 0$ and,
  therefore, $y_{n,m}\goesuo 0$.
\end{proof}

It follows that if $X$ is infinite-dimensional then order convergence
is not unbounded. It was observed in \cite[Theorem~2.3]{Kandic:17}
that the norm topology on a Banach lattice $X$ is unbounded iff $X$
has a strong unit.

Let $X$ be a discrete vector lattice.  For every atom $a$ in~$X_+$,
consider the corresponding band projection~$P_a$. A net $(x_\alpha)$
converges to zero \term{component-wise} (or \term{point-wise}) if
$(P_ax_\alpha)$ converges to zero for every atom $a$ (in the
one-dimensional subspace spanned by $a$). We write $x_\alpha\goesp x$
if $x_\alpha-x$ converges to zero coordinate-wise. It is easy to see
that this is a Hausdorff locally solid convergence.  By
\cite[Theorem~1.78]{Aliprantis:03}, one may view $X$ as an order dense
(hence regular) sublattice of~$\mathbb R^A$, where $A$ is a maximal
disjoint collection of atoms in~$X$. Then coordinate-wise
convergence on $X$ agrees with point-wise convergence in~$\mathbb R^A$.

\begin{proposition}\label{discr-uo-p}
  In a discrete vector lattice, uo-convergence agrees with coordinate-wise
  convergence.
\end{proposition}

\begin{proof}
  Let $X$ be a discrete vector lattice. We view $X$ as an order dense
  (hence regular) sublattice of~$\mathbb R^A$, where $A$ is a maximal
  disjoint collection of atoms in~$X$. Therefore, it suffices to show
  that uo-convergence agrees with point-wise convergence
  in~$\mathbb R^A$. Let $(x_\alpha)$ be a net in~$\mathbb R^A$; it
  suffices to show that $x_\alpha\goesuo x$ in $\mathbb R^A$ iff
  $(x_\alpha)$ converges to $x$ point-wise. We may assume without loss
  of generality that $x=0$ and $x_\alpha\ge 0$ for all~$\alpha$.

  Suppose that $x_\alpha\goesuo 0$. Fix $a\in A$. Then
  $x_\alpha\wedge\one_{\{a\}}\goeso 0$. However,
  $x_\alpha\wedge\one_{\{a\}}$ is a scalar multiple of~$\one_{\{a\}}$;
  namely, $x_\alpha\wedge\one_{\{a\}}=\bigl(x_\alpha(a)\wedge
  1\bigr)\one_{\{a\}}$. It follows that $x_\alpha(a)\wedge
  1\to 0$, hence $x_\alpha(a)\to 0$.

  Suppose now that $(x_\alpha)$ converges to $0$ point-wise. Fix
  $0\le u\in\mathbb R^A$. We need to show that
  $x_\alpha\wedge u\goeso 0$. It suffices to show that
  $v_\alpha\downarrow 0$, where
  $v_\alpha=\sup_{\beta\ge\alpha}(x_\beta\wedge u)$. This follows
  immediately from the observation that $v_\alpha(a)\downarrow 0$ for
  every $a\in A$.
\end{proof}

Recall that given a convergence vector space $(X,\lambda)$, the
topology given by the collection of all $\lambda$-closed sets is
called the \term{topological modification} of~$\lambda$; we will
denote this topology by $\mathrm{t}(\lambda)$ or
just~$\mathrm{t}\lambda$.  See, for instance Definitions~1.3.1
and~1.3.8 in~\cite{Beattie:02} and \cite[page 249 \& Proposition
3.3]{OBrien:23}.  The topological modification of a locally solid
convergence need not be locally solid: it was shown in Example~10.7
in~\cite{OBrien:23} that $\mathrm{t}(\mathrm{o})$ need not even be
linear.

\begin{proposition}\label{uo-top-un}
  Let $X$ be a vector lattice with the countable supremum property and
  $\tau$ a Hausdorff locally solid order continuous 
  topology on~$X$. Then 
  \begin{math}
    \mathrm{t}(\mathrm{uo})=\mathrm{u}\tau.
  \end{math}
\end{proposition}

\begin{proof}
  It suffices to prove that uo-closed and $\mathrm{u}\tau$-closed sets
  agree. Since $\tau$ is an order continuous topology,
  $\tau\le \textrm{o}$, and so $\mathrm{u}\tau\le\mathrm{uo}$. It
  follows that every u$\tau$-closed set is uo-closed. To prove the
  converse, suppose that $A$ is uo-closed, $(x_\alpha)$ is a net in
  $A$ and $x_\alpha\xrightarrow{\textrm{u}\tau} x$. We need to show
  that $x\in A$. By \cite[Theorem~5.6]{Deng:23}, there exists a
  sequence $(\alpha_k)$ of indices such that
  $x_{\alpha_k}\xrightarrow{\mathrm{uo}} x$. It follows that $x\in A$.
\end{proof}

\begin{corollary}
  The topological modification of $\mathrm{uo}$-convergence on an order
  continuous Banach lattice is the $\mathrm{un}$-topology.
\end{corollary}

\begin{remark}\label{uo-top-un-B}
  Under the assumption of Proposition~\ref{uo-top-un}, we may restrict
  both convergences in the proposition to any subset $B$ of~$X$, that is,
  \begin{math}
    \mathrm{t}\bigl((\mathrm{uo})_{|B}\bigr)=(\mathrm{u}\tau)_{|B}.
  \end{math}
  The same proof works.
\end{remark}


\section{Relative uniform convergence}
\label{sec:ru}

Relative uniform convergence (or just uniform convergence) in an
(Archimedean) vector lattice $X$ was discussed
in~\cite{OBrien:23}. Recall that a net $(x_\alpha)$ \term{converges
  relatively uniformly} to~$x$, written $x_\alpha\goesu x$, if there
exists $e\in X_+$, called a \term{regulator} of convergence, such that
for every $\varepsilon>0$ there exists an index $\alpha_0$ suh that
$\abs{x_\alpha-x}<\varepsilon e$ whenever
$\alpha\ge\alpha_0$. Equivalently, for every $n\in\mathbb N$ there
exists an index $\alpha_0$ such that $\abs{x_\alpha-x}<\frac1n e$
whenever $\alpha\ge\alpha_0$.%
\footnote{Replacing the condition $\abs{x_\alpha-x}<\frac1n e$ with
  \begin{math}
    -\frac1n e\le x_\alpha-x\le\frac1n e,
  \end{math}
  one can extend the concept of relative uniform convergence to
  ordered vector spaces.}
The latter observation
suggests the sequential nature of relative uniform convergence. In the
filter language, it is easy to see that $\mathcal F\goesu 0$ whenever
there exists $e\in X_+$ such that
$[-\varepsilon e,\varepsilon e]\in\mathcal F$ for every
$\varepsilon>0$; equivalently, $[-\frac1n e,\frac1n e]\in\mathcal F$
for every $n\in\mathbb N$.

\begin{proposition}\label{u-strongest}
  Relative uniform convergence is the strongest locally solid
  convergence.
\end{proposition}

\begin{proof}
  We need to prove that for every locally solid convergence structure
  $\eta$ on~$X$, if $x_\alpha\goesu 0$ then $x_\alpha\goeseta 0$. For
  sequences, the proof is straightforward. For nets, it requires some
  detail. Let $e$ be a regulator for $x_\alpha\goesu 0$. For every
  $n\in\mathbb N$, there exists $\alpha_0$ such that
  $\abs{x_\alpha}\le\frac1n e$ whenever $\alpha\ge\alpha_0$. Since
  $\frac1n e\goeseta 0$, it follows from
  Lemma~\ref{net-dom} that
  $x_\alpha\goeseta 0$.
\end{proof}

Let $e\in X_+$. Recall that $x$ is in the principal ideal $I_e$ iff
$\abs{x}\le\lambda e$ for some $\lambda>0$. That is, $x\in I_e$ iff
$\norm{x}_e$ is finite, where
\begin{displaymath}
    \norm{x}_e=\inf\bigl\{\lambda\ge 0\mid\abs{x}\le\lambda e\bigr\}.
\end{displaymath}
It is easy to see that $\bigl(I_e,\norm{\cdot}_e\bigr)$ is a normed
lattice. Furthermore, $x_\alpha\goesu x$ iff
$\norm{x_\alpha-x}_e\to 0$ for some $e\in X_+$. The Krein-Kakutani
Representation Theorem asserts that the completion of
$\bigl(I_e,\norm{\cdot}_e\bigr)$ is lattice isometric to $C(K)$ for
some compact Hausdorff space~$K$, with $e$ corresponding
to~$\one$. Thus, a net $(x_\alpha)$ converges relatively uniformly in
$X$ iff it converges in the supremum norm in the appropriate $C(K)$
space.

Recall \cite[page 252]{OBrien:23} that a net $(x_\alpha)_{\alpha\in A}$ in a convergence vector space is
\term{Cauchy} if the double net $(x_\alpha-x_\beta)$ indexed by
$A^2$ converges to zero. A convergence vector space is \term{complete}
(or \term{sequentially complete}) if every Cauchy net (respectively,
sequence) is convergent.

\begin{theorem}\label{ru-compl}
  Let $X$ be a vector lattice. The following are equivalent:
  \begin{enumerate}
  \item\label{ru-compl-ru} $X$ is relatively uniformly complete, i.e.,
    complete with respect to relative uniform convergence;
  \item\label{ru-compl-Ie} $\bigl(I_e,\norm{\cdot}_e\bigr)$ is
    complete for every $e\in X_+$ (and is, therefore, lattice
    isometric to some $C(K)$ space);
  \item\label{ru-compl-sadm} $X$ admits a sequentially complete
    Hausdorff locally solid convergence structure.
  \item\label{ru-compl-adm} $X$ admits a complete
    Hausdorff locally solid convergence structure.    
  \end{enumerate}
\end{theorem}

\begin{proof}
  \eqref{ru-compl-ru}$\Rightarrow$\eqref{ru-compl-adm}$\Rightarrow$\eqref{ru-compl-sadm}  is trivial.

  \eqref{ru-compl-sadm}$\Rightarrow$\eqref{ru-compl-Ie} Let $\lambda$
  be a sequentially complete Hausdorff locally solid convergence
  structure. Let $e\in X_+$ and $(x_n)$ a Cauchy sequence in
  $\bigl(I_e,\norm{\cdot}_e\bigr)$. Then $(x_n)$ is relatively
  uniformly Cauchy in~$X$. By Proposition~\ref{u-strongest}, $(x_n)$
  is $\lambda$-Cauchy. By assumption, $x_n\goesl x$ for some $x\in X$.

  Fix $\varepsilon>0$. Since $(x_n)$ is $\norm{\cdot}_e$-Cauchy, there
  exists $n_0$ such that $\abs{x_n-x_m}\le\varepsilon e$ for every
  $n,m\ge n_0$ and, therefore, $x_n\in[x_m-\varepsilon
  e,x_m+\varepsilon e]$. Since order intervals are $\lambda$-closed,
  passing to the limit on $n$ we get
  $x\in[x_m-\varepsilon
  e,x_m+\varepsilon e]$. This yields $\norm{x_m-x}_e\le\varepsilon$. It
  follows that $(x_n)$ is $\norm{\cdot}_e$-convergent to~$x$.

  \eqref{ru-compl-Ie}$\Rightarrow$\eqref{ru-compl-ru} Suppose that
  $\bigl(I_e,\norm{\cdot}_e\bigr)$ is complete for every $e\in X_+$.
  Let $(x_\alpha)$ be a relatively uniformly Cauchy net in~$X$. That
  is, the double net $(x_\alpha-x_\beta)_{(\alpha,\beta)}$ converges
  relatively uniformly to zero. Let $e$ be a regulator of this
  convergence. There exists $\alpha_0$ such that for all
  $\alpha\ge\alpha_0$ we have $\norm{x_\alpha-x_{\alpha_0}}_e<1$. It
  follows that $\abs{x_\alpha}\le\abs{x_{\alpha_0}}+e=:u$. Hence, the
  tail $(x_\alpha)_{\alpha\ge\alpha_0}$ is contained in $I_u$ and is
  Cauchy with respect to $\norm{\cdot}_{u}$. Since
  $\bigl(I_u,\norm{\cdot}_u\bigr)$ is complete, it follows that
  $(x_\alpha)$ converges in $\bigl(I_u,\norm{\cdot}_u\bigr)$, hence the
  net converges relatively uniformly in~$X$.
 \end{proof}




 Note that relative uniform convergence depends on the ambient
 space. That is, given a sublattice $Y$ of~$X$, the subspace
 convergence structure induced on $Y$ by the relative uniform
 convergence on $X$ need not agree with the relative uniform
 convergence on~$Y$. It is easy to see that if $x_\alpha\goesu x$ in
 $Y$ then $x_\alpha\goesu x$ in~$X$, but the converse may fail. For
 example, the sequence $\frac1n e_n$ converges to zero relatively
 uniformly in~$\ell_\infty$, but not in~$\ell_1$. However, if $Y$ is
 majorizing in $X$ than $x_\alpha\goesu x$ in $Y$ iff
 $x_\alpha\goesu x$ in~$X$, assuming that $(x_\alpha)$ and $x$ are
 in~$Y$. It was observed in~\cite{Taylor:20} that if $Y$ is a
 relatively uniformly closed sublattice of a relatively uniformly
 complete vector lattice $X$ (in particular, if $Y$ is a norm closed
 sublattice of a Banach lattice $X$) then $x_n\goesu 0$ in $Y$ iff
 $x_n\goesu 0$ in $X$ for every sequence $(x_n)$ in~$Y$, yet this may
 fail for nets.
  
\begin{corollary}\label{c0-convs}
  Norm, order, and relative uniform convergences agree for sequences
  in~$c_0$.
\end{corollary}

\begin{proof}
  It is clear that relative uniform convergence implies order
  convergence. Since $c_0$ is order continuous, order convergence
  implies norm convergence. Suppose now that $x_n\goesnorm 0$ for a
  sequence $(x_n)$ in~$c_0$. Since $c_0$ is a norm closed subspace
  of~$\ell_\infty$, we have $x_n\goesnorm 0$ in~$\ell_\infty$. It
  follows that $x_n\goesu 0$ in~$\ell_\infty$. By the preceding
  discussion, we now have $x_n\goesu 0$ in~$c_0$.
\end{proof}

It was shown in~\cite{Wirth:75} that norm and relative uniform
convergence for sequences in a Banach lattice $X$ agree iff $X$ is
lattice isomorphic to an AM-space.

Every norm convergent sequence in a Banach lattice $X$ has a
relatively uniformly convergent subsequence; see, e.g., Lemma~1.1
in~\cite{Taylor:20}. It follows that relatively uniformly closed sets
in $X$ agree with norm closed sets. Hence, the topological
modification of ru in $X$ is the norm topology. This remains valid in
a locally solid Frechet lattice, as every convergent net in such a
space contains (as a subset) a sequence that is relatively uniformly convergent
to the same limit by Proposition~5.1 of~\cite{Deng:23}. It also
follows that if $X$ is an order continuous Banach lattice then order
closed sets in $X$ agree with norm closed sets and, therefore, the
topological modification of order convergence on $X$ is the norm
topology.

\subsection*{Unbounded modification of ru convergence.}
We write $\mathrm{ru}$ for the relative uniform convergence and
$\mathrm{u}(\mathrm{ru})$ for its unbounded modification.
It follows from Proposition~\ref{u-strongest} that
$\mathrm{u}(\mathrm{ru})$ is the strongest unbounded convergence on~$X$.

Recall that the norm topology on a Banach lattice $X$ is unbounded iff
$X$ has a strong unit.

\begin{proposition}\label{ru-stru}
  If $X$ has a strong unit then $\mathrm{ru}$ convergence is
  unbounded.
\end{proposition}

\begin{proof}
  Let $e$ be a strong unit in~$X$.
  Suppose that $(x_\alpha)$ is a net in $X_+$ such that
  $x_\alpha\wedge u\goesu 0$ for every $u\in X_+$. In particular,
  $x_\alpha\wedge e\goesu 0$. Since $e$ is a strong unit, one can
  choose $e$ as the regulator for the latter convergence. In
  particular, $x_\alpha\wedge e\le\frac12 e$ for all sufficiently
  large~$\alpha$, which implies $x_\alpha\le e$, so that
  $x_\alpha\wedge e=x_\alpha$. It follows that $x_\alpha\goesu 0$. 
\end{proof}

\begin{question}\label{q:ru-ubdd}
  Is the converse of Proposition~\ref{ru-stru} true? That is, if the
  relative uniform convergence $X$ is unbounded, does this imply that $X$
  has a strong unit?
\end{question}

In order continuous Banach lattices, ru-convergence agrees with order
convergence by, e.,g.,~\cite{Bedingfield:80}, hence
$\mathrm{u}(\mathrm{ru})$ agrees with uo. On the other hand, on $C(K)$
spaces, $\mathrm{u}(\mathrm{ru})=\mathrm{ru}$ by
Proposition~\ref{ru-stru}, and is, generally, different from~$\mathrm{uo}$.

\begin{example}
  \emph{Relative uniform convergence on $c_0$ is not unbounded.}
  Let $(e_n)$ be the standard unit basis of~$c_0$. For every $0\le
  u\in c_0$, we have $e_n\wedge u\goesnorm 0$ and, therefore,
  $e_n\wedge u\goesu 0$ by Corollary~\ref{c0-convs}. Yet, clearly,
  $(e_n)$ does not converge to zero relatively uniformly.
\end{example}

Let $\Omega$ be a Hausdorff locally compact topological space. The
space $C_0(\Omega)$ of continuous real-valued functions on $\Omega$
vanishing at infinity is a Banach lattice under the supremum norm. We
write $C_c(\Omega)$ for the space of all continuous real valued
functions on $\Omega$ with compact support. Clearly, $C_c(\Omega)$ is
an ideal in $C_0(\Omega)$, while the latter is an ideal in
$C(\Omega)$. We say that a net $(f_\alpha)$ in $C(\Omega)$ converges
to some $f\in C(\Omega)$ \term{uniformly on compact sets} and write
$f_\alpha\goesucc f$ if for every $\varepsilon>0$ and every compact
$K\subseteq\Omega$ there exists $\alpha_0$ such that
$\abs{f_\alpha(\omega)-f(\omega)}<\varepsilon$ whenever
$\alpha\ge\alpha_0$ and $\omega\in K$. It is easy to see that
$\mathrm{ucc}$ is a locally solid topological convergence. By
\cite[Corollary~1.5.17]{Beattie:02} or
\cite[Proposition~8.3]{OBrien:23}, $\mathrm{ucc}$ agrees with the
continuous convergence on $C_0(\Omega)$. It was observed in
\cite[Example~20]{Troitsky:04} that $\mathrm{ucc}=\mathrm{un}$ on
$C_0(\Omega)$.

\begin{proposition}\label{C0-uru}
  Let $\Omega$ be a locally compact topological space. Put
  $I=C_c(\Omega)$. The following convergences on $C_0(\Omega)$
  coincide: $\mathrm{u}(\mathrm{ru})$, $\mathrm{u}_I(\mathrm{ru})$,
  $\mathrm{un}$, $\mathrm{u}_I\mathrm{n}$, and~$\mathrm{ucc}$. In
  particular, $\mathrm{ucc}$ is unbounded.
\end{proposition}

\begin{proof}
  It is straightforward that
  $\mathrm{u}(\mathrm{ru})\ge\mathrm{u}_I(\mathrm{ru})\ge\mathrm{u}_I\mathrm{n}$
  and
  $\mathrm{u}(\mathrm{ru})\ge\mathrm{un}\ge\mathrm{u}_I\mathrm{n}$. It
  therefore suffices to show that
  $\mathrm{u}_I\mathrm{n}\ge\mathrm{ucc}\ge\mathrm{u}(\mathrm{ru})$.

  $\mathrm{u}_I\mathrm{n}\ge\mathrm{ucc}$: Let
  $0\le f_\alpha\xrightarrow{\mathrm{u}_I\mathrm{n}} 0$ and let $K$ be
  a compact set. Find $u\in C_c(\Omega)_+$ such that $u$ equals 1
  on~$K$. It follows from $f_\alpha\wedge u\goesnorm 0$ that
  $f_\alpha$ converges to zero uniformly on~$K$.
  
  $\mathrm{ucc}\ge\mathrm{u}(\mathrm{ru})$: Suppose that
  $0\le f_\alpha\goesucc 0$. Let $u\in C_0(\Omega)_+$. We want to show
  that $f_\alpha\wedge u\goesu 0$. We use $\sqrt{u}$ as a
  regulator. Fix $\varepsilon>0$.
  Since  $u\in C_0(\Omega)_+$,
  the set $K:=\bigl\{\omega\mid u(\omega)\ge\varepsilon^2\bigr\}$
  is compact. Find $\alpha_0$ such that $f_\alpha(\omega)<\varepsilon^2$
  for all $\omega\in K$ whenever $\alpha\ge\alpha_0$. Fix
  $\alpha\ge\alpha_0$ and let $\omega\in\Omega$. If $w\in K$ then
  $f_\alpha(\omega)<\varepsilon^2\le\varepsilon\sqrt{u(\omega)}$. If
  $\omega\notin K$ then $u(\omega)<\varepsilon^2$, so that
  $u(\omega)\le\varepsilon\sqrt{u(\omega)}$. In either case,
  $(f_\alpha\wedge u)(\omega)\le\varepsilon\sqrt{u(\omega)}$.
\end{proof}

Analogously, one can prove the following:

\begin{proposition}
  For a locally compact topological space~$\Omega$,
  $\mathrm{u}(\mathrm{ru})=\mathrm{un}=\mathrm{ucc}$ on
  $C_c(\Omega)$. In particular, $\mathrm{ucc}$ is unbounded.
\end{proposition}

\begin{remark}\label{Cc-ru}
  It is easy to see that in $C_c(\Omega)$, we have $f_\alpha\goesu 0$
  iff $f_\alpha\goesnorm 0$ and there exists $\alpha_0$ and a compact
  set $K$ such that $\operatorname{supp}f_\alpha\subseteq K$ for all
  $\alpha\ge\alpha_0$. Observe that, in this case, $f_\alpha\goesnorm 0$ is
  equivalent to $f_\alpha\goesucc 0$.
\end{remark}

We can now prove a somewhat similar result for
$C_0(\Omega)$:

\begin{corollary}\label{C0-ru}
  In $C_0(\Omega)$, we have $f_\alpha\goesu 0$ iff
  $f_\alpha\goesucc 0$ and there exists $\alpha_0$ such that
  $(f_\alpha)_{\alpha\ge\alpha_0}$ is norm bounded and for every
  $\varepsilon>0$ there exists a compact set $K$ such that
  $\abs{f_\alpha(\omega)}<\varepsilon$ whenever $\alpha\ge\alpha_0$ and
  $\omega\notin K$.
\end{corollary}

\begin{proof}
  It is easy to see that $f_\alpha\goesu 0$ iff
  $f_\alpha\xrightarrow{\mathrm{u}(\mathrm{ru})}0$ and the net is
  eventually order bounded. Now combine Proposition~\ref{C0-uru} with
  Proposition~3.5 of~\cite{Bilokopytov:22}.
\end{proof}

We can now completely characterize those $\Omega$ for which the
relative uniform convergence on $C_0(\Omega)$ is unbounded. In
particular, this yields an affirmative answer to
Question~\ref{q:ru-ubdd} for $C_0(\Omega)$ spaces.

\begin{corollary}\label{comp-0-ru}
  A locally compact topological space $\Omega$ is compact iff the
  relative uniform convergence on $C_0(\Omega)$ is unbounded.
\end{corollary}

\begin{proof}
  The forward implication follows from
  Proposition~\ref{ru-stru}. Suppose that
  $\mathrm{u}(\mathrm{ru})=\mathrm{ru}$. Proposition~\ref{C0-uru}
  implies that
  $\mathrm{ru}\ge\mathrm{n}\ge\mathrm{un}=\mathrm{u}(\mathrm{ru})=\mathrm{ru}$. It
  follows that $\mathrm{n}=\mathrm{un}$. Using
  \cite[Theorem~2.3]{Kandic:17} we conclude that $C_0(\Omega)$ has a
  strong unit, say,~$u$. Then $u(\omega)>0$ for every
  $\omega\in\Omega$. It follows that $\sqrt{u}\le\lambda u$ for some
  $\lambda>0$. We conclude that $u(\omega)\ge\frac{1}{\lambda^2}$ for
  every $\omega\in\Omega$. This yields that $\Omega$ is compact.
\end{proof}

In the preceding corollary, $C_0(\Omega)$ may be replaced with
$C_c(\Omega)$. Note that the relevant direction of the proof of
Theorem~2.3 in \cite{Kandic:17} remains valid for normed lattices.

\begin{question}
  Characterize vector lattices where $\mathrm{u}(\mathrm{ru})$ is
  topological. Characterize Banach lattices where
  $\mathrm{u}(\mathrm{ru})=\mathrm{un}$; do we have
  $\mathrm{t}\bigl(\mathrm{u}(\mathrm{ru})\bigr)=\mathrm{un}$?
\end{question}

Next, we extend some of the preceding results to the case when
$X=C(\Omega)$, where $\Omega$ is locally compact and $\sigma$-compact.
This is equivalent to $\Omega$ having a \term{compact exhaustion},
i.e., a sequence $(K_n)_{n\in\mathbb N}$ of compact subsets such that
$K_n\subseteq\Int K_{n+1}$ for every $n$ and
$\bigcup_{n=1}^\infty K_n=\Omega$; see, e.g., \cite[3.8.C]{Engelking:89}.
For example, our results will apply
when $\Omega=\mathbb R^N$. Note that in this case every compact
$K\subseteq\Omega$ is contained in $K_n$ for some~$n$. For
convenience, we put $K_0=\varnothing$.
It is easy to see that the ucc
topology is generated by the seminorms
\begin{displaymath}
  \norm{f}_n=\sup\bigl\{\bigabs{f(t)}
  \mid t\in K_n\setminus\Int K_{n-1}\bigr\},
  \ n\in\mathbb N,
\end{displaymath}
so $C(\Omega)$ is a Frechet space. A subset $A$ of $C(\Omega)$ is
bounded with respect to the ucc topology iff
$\sup_{f\in A}\norm{f}_n<\infty$ for every~$n$.

\begin{lemma}\label{exh}
  Let $(K_n)$ be a compact exhaustion of~$\Omega$, and let
  $(m_n)$ be a sequence in~$\mathbb N$. There exists $f\in C(\Omega)$
  such that $f_{|K_n\setminus\Int K_{n-1}}\ge m_n$ for all
  $n\in\mathbb N$.
\end{lemma}

\begin{proof}
  We may assume that $(m_n)$ is increasing; otherwise, replace $m_n$
  with $\max\{m_1,\dots,m_n\}$. Using Urysohn's lemma, find $f_n$ in
  $C(K_n\setminus\Int K_{n-1})$ such that $f_n$ equals $m_n$ on
  $\partial K_{n-1}$, $m_{n+1}$ on $\partial K_n$, and is between
  $m_n$ and $m_{n+1}$ on the rest of $K_n\setminus\Int K_{n-1}$ (we
  again put $K_0=\varnothing$). We now define $f$ on $\Omega$ so that
  it agrees with $f_n$ on $K_n\setminus\Int K_{n-1}$. It is easy to
  verify that $f$ satisfies the requirements of the lemma.
\end{proof}

\begin{proposition}\label{sc}
  Suppose that $\Omega$ is locally compact and
  $\sigma$-compact. Then
  \begin{enumerate}
  \item\label{sc-eq} $\mathrm{u}(\mathrm{ru})=\mathrm{ucc}$ on
    $C(\Omega)$;
  \item\label{sc-bdd} A subset of $C(\Omega)$ is order bounded iff it
    is bounded in the ucc topology;
  \item\label{sc-ru} A net in $C(\Omega)$ converges relatively
    uniformly iff it converges ucc and is eventually (order) bounded;
  \item\label{sc-comp} Relative uniform convergence on $C(\Omega)$ is unbounded
    iff $\Omega$ is compact.  
  \end{enumerate}
\end{proposition}

\begin{proof}
  \eqref{sc-eq} Observe that ucc is unbounded because if
  $0\le f_\alpha\xrightarrow{\mathrm{u}(\mathrm{ucc})}0$ then
  $\one\wedge f_\alpha\goesucc 0$; it is then easy to see that
  $f_\alpha\goesucc 0$. It now follows from
  $\mathrm{ru}\ge\mathrm{ucc}$ that
  $\mathrm{u}(\mathrm{ru})\ge\mathrm{ucc}$.

  We are now going to show that
  $\mathrm{u}(\mathrm{ru})\le\mathrm{ucc}$.  If $\Omega$ is compact,
  this is trivial, so we assume now that $\Omega$ is not compact.
  Suppose that $0\le f_\alpha\goesucc 0$. Let $u\ge 0$. We need to
  show that $f_\alpha\wedge u\goesu 0$. We may assume that $u\ge\one$
  and the set
  \begin{math}
    C_n=\bigl\{\omega\in\Omega\mid u(\omega)\le n\bigr\}
  \end{math}
  is compact for every~$n$, because if this is not the case, we may
  replace $u$ with $u\vee\one\vee f$, where $f$ is the function from
  Lemma~\ref{exh} with $m_n=n$. For every~$n$, there is an $\alpha_n$
  such that ${f_\alpha}(\omega)\le\frac1n\le\frac1n u^2(\omega)$ for
  all $\alpha\ge\alpha_n$ and all $\omega\in C_n$. If
  $\omega\notin C_n$ then
  $u(\omega)\le \frac1n u^2(\omega)$. It follows
  that $f_\alpha\wedge u\le\frac1nu^2$ for all $\alpha\ge\alpha_n$.
  We conclude that $f_\alpha\wedge u\goesu 0$.
 
  \eqref{sc-bdd} The forward implication is straightforward. For the
  converse, if $A$ is ucc-bounded, put $m_n=\sup_{f\in A}\norm{f}_n$
  and apply Lemma~\ref{exh}.
  
  \eqref{sc-ru} follows from combining \eqref{sc-eq} and~
  \eqref{sc-bdd}.

  \eqref{sc-comp} If $\Omega$ is compact then ru-convergence is
  unbounded by Proposition~\ref{ru-stru}. Assume that $\Omega$ is not
  compact. Then there is a compact exhaustion $(K_n)$ which does not
  stabilize, that is, $K_n\ne K_{n+1}$ for all $n\in\mathbb N$. Find a
  disjoint sequence $(f_n)$ in $C(\Omega)$ such that $f_n$ equals $1$
  on $\partial K_{2n-1}$ and vanishes outside of
  $\Int K_{2n}\setminus K_{2n-2}$, for $n\ge 2$. Then the double
  sequence $(mf_n)$ is ucc-null but not eventually order bounded. Therefore,
  it is $\mathrm{u}(\mathrm{ru})$-null by~\eqref{sc-eq}, but nor
  ru-null by~\eqref{sc-ru}.
\end{proof}

\begin{question}
  Characterize those locally solid topologies in which topologically bounded
  and order bounded sets agree.
\end{question}

\section{Bounded sets in convergence vector spaces}
\label{sec:bddsets}

Let $(X,\lambda)$ be a convergence vector space over~$\mathbb K$,
where $\mathbb K$ is $\mathbb R$ or~$\mathbb C$. We recall the concept
of a $\lambda$-bounded set; see, e.g.,~\cite{OBrien:23}, for details.
A subset $A$ of $X$ is $\lambda$\term{-bounded} if
$[\mathcal N_0A]\goesl 0$, where $\mathcal N_0$ is the neighbourhood
filter of $0$ in~$\mathbb K$. Equivalently, for every net
$(r_\gamma)_{\gamma\in\Gamma}$ in $\mathbb K$ with $r_\gamma\to 0$ we
have $(r_\gamma x)_{(\gamma,x)}\goesl 0$, where the latter net is
indexed by $\Gamma\times A$ ordered by the first component; we denote
this net by $(r_\gamma A)$. It was observed in~\cite{OBrien:23} that
if $A$ is circled then instead of a general net $(r_\gamma)$ in the
definition of a bounded set, it suffices to take the
sequence~$(\frac1n)$.  In general, it suffices to take $(0,1)$ viewed
as a net decreasing to zero:

\begin{theorem}\label{bdd}
  Let $A$ be a set in a convergence vector space~$X$. The following
  are equivalent:
  \begin{enumerate}
  \item\label{bdd-bd} $A$ is bounded;
  \item\label{bdd-01} $(0,1)A\to 0$, where $(0,1)$ is viewed as a
    decreasing net;
  \item\label{bdd-2nets} If $r_\gamma\to 0$ in $\mathbb K$ and
    $(x_\gamma)$ is a net in $A$ then $r_\gamma x_\gamma\to 0$.
  \end{enumerate}
\end{theorem}

\begin{proof}
  \eqref{bdd-bd}$\Rightarrow$\eqref{bdd-01} is trivial.

  \eqref{bdd-01}$\Rightarrow$\eqref{bdd-2nets} The net
  \begin{math}
    \bigl(\sqrt{\abs{r_\gamma}}x_\gamma\bigr)
  \end{math}
  is a quasi-subnet of $(0,1)A$, hence it converges to zero. Also, the
  net
  \begin{math}
    \bigl(\frac{r_\gamma}{\sqrt{\abs{r_\gamma}}}\bigr)
  \end{math}
  converges to zero (if $r_\gamma=0$ we interpret the quotient as zero).
  By the continuity of scalar multiplication, $r_\gamma x_\gamma\to 0$.
  
  \eqref{bdd-2nets}$\Rightarrow$\eqref{bdd-bd}
  Let $(r_\gamma)_{\gamma\in\Gamma}$ be a net in
  $\mathbb K$ with $r_\gamma\to 0$. Define a new index set
  $\Lambda=\Gamma\times A$, ordered by the first component. Define two
  nets indexed by $\Lambda$: for $(\gamma,a)\in\Lambda$, let
  $s_{(\gamma,a)}=r_\gamma$ and $x_{(\gamma,a)}=a$. It is easy to see
  that $(s_{(\gamma,a)})$ is tail equivalent to~$(r_\gamma)$. It
  follows that $s_{(\gamma,a)}\to 0$. Therefore,
  $s_{(\gamma,a)}x_{(\gamma,a)}\to 0$. Since the
  net $(s_{(\gamma,a)}x_{(\gamma,a)})$ is tail equivalent to
  $(r_\gamma A)$, we have $r_\gamma A\to 0$ and, therefore, $A$ is
  bounded.
\end{proof}

\begin{example}
  It is easy to see that for topological vector
  spaces, one may replace nets with sequences in
  Theorem~\ref{bdd}\eqref{bdd-2nets}.  However, this is not true for
  general linear convergences. Recall from~\cite{OBrien:23} that order
  bounded sets in a vector lattice are precisely the sets that are
  bounded with respect to order convergence. The set
  $\{e_n\}_{n\in\mathbb N}$ in $c_0$ is not order bounded, yet for
  every sequence $(r_k)$ in $\mathbb R$ with $r_k\to 0$ and for every
  sequence $(n_k)$ in~$\mathbb N$, the sequence $(r_ke_{n_k})$ is
  order null.
\end{example}

\begin{remark}
  It is claimed in \cite{Beattie:02,OBrien:23} that the convex hull of
  a bounded set in a convergence vector space is again bounded. This
  may, actually, fail even in the case of a topological convergence if
  the space is not locally convex. Indeed,
  let $A$ be the unit ball in $\ell_p$ with $0<p<1$. Then $A$ is
  bounded with respect to the quasi-norm $\norm{\cdot}_p$, hence with
  respect to the corresponding convergence, yet its convex hull is unbounded.
\end{remark}

Recall that a subset $A$ of a vector lattice $X$ is \term{dominable}
iff $A$ is order bounded as a subset of the universal completion
$X^u$ of~$X$.

\begin{theorem}\label{uo-bdd-dom}
  A subset $A$ of a vector lattice $X$ is uo-bounded iff it is dominable.
\end{theorem}

\begin{proof}
  Replacing $A$ with its solid hull, we may assume without loss of
  generality that $A$ is solid. In particular, it is circled. Thus,
  $A$ is uo-bounded in $X$ iff $\frac1n A\goesuo 0$ in~$X$.  Since $X$
  is regular in~$X^u$, $\frac1n A\goesuo 0$ in $X$ iff
  $\frac1n A\goesuo 0$ in~$X^u$, hence $A$ is uo-bounded in $X$ iff
  $A$ is uo-bounded in~$X^u$. Thus, without loss of generality, $X$ is
  universally complete. It now suffices to prove that if $A$ is
  uo-bounded then it is order bounded. The converse implication is
  trivial as $\mathrm{uo}\le\mathrm{o}$.

  Suppose that $A$ is uo-bounded. Being universally complete, $X$ has
  a weak unit, say,~$e$.
  Since $\frac1n A\goesuo 0$, the net
  \begin{math}
    \bigl(\abs{\frac1n a}\wedge e),
  \end{math}
  which is indexed by $\mathbb N\times A$, ordered by the first
  component, is order null. Hence
  \begin{displaymath}
    0=\inf_{(n,a)}\sup_{(m,b)\ge(n,a)}\abs{\tfrac1m b}\wedge e
    =\inf_{n}\sup_{m\ge n,b\in A}\tfrac1m\abs{b}\wedge e
    =\inf_n\sup_{a\in A}\tfrac1n\abs{a}\wedge e.
  \end{displaymath}
  Put $x_n=\sup_{a\in A}\tfrac1n\abs{a}\wedge e$. Then $0\le x_n\le e$
  and $x_n\downarrow 0$. Put $s_n=P_{e-x_n}e$ for all $n\in\mathbb
  N$. Then $e\ge s_n=\sup_me\wedge m(e-x_n)\ge e-x_n$.  It follows
  that $0\le s_n\uparrow e$. Let $d_1=s_1$ and $d_n=s_n-s_{n-1}$ for
  $n>1$; then $(d_n)$ is disjoint and $\sum_{n=1}^md_n=s_m$ for
  every~$m$. Put $h=\bigvee_n nd_n$; the supremum exists because $X$
  is universally complete. Let $0\le a\in A$; we will show that
  $a\le h$ and, therefore, $A\subseteq[-h,h]$.

  Fix~$n$. Clearly, $(\frac1n a-e)^+\perp(e-\frac1n a)^+$.
  It follows from $(e-\frac1n a)^+\ge
  (e-x_n)^+=e-x_n$ and the fact that $d_n$ is in the band generated by
  $e-x_n$ that $(\frac1n a-e)^+\perp
  d_n$. Therefore, $0=P_{d_n}(\frac1n
  a-e)^+=(\frac1n P_{d_n}a-P_{d_n}e)^+$, so that $P_{d_n}a\le
  nP_{d_n}e=nd_n$. We conclude that
  \begin{displaymath}
    P_{s_m}a=\sum_{n=1}^mP_{d_n}a\le\sum_{n=1}^mnd_n\le h.
  \end{displaymath}
  By Theorem~1.48(ii) in~\cite{Aliprantis:06}, it follows from
  $s_m\uparrow e$ that $P_{s_m}a\uparrow P_ea=a$
  because $e$ is a weak unit. This yields $a\le h$. 
\end{proof}

\begin{question}
  Characterize un-bounded subsets of a Banach lattice.
\end{question}

Recall that a linear convergence is \term{locally bounded} if every
convergent net has a bounded tail or, equivalently, if every
convergent filter contains a bounded set. It is well known that a
locally convex topology is locally bounded iff it is normable. It is
easy to see that order and relative uniform convergences are locally
bounded.

\begin{proposition}
  In a vector lattice~$X$,
  $\mathrm{uo}$-convergence is locally bounded iff $X$ is finite-dimensional.
\end{proposition}

\begin{proof}
  Follow the proof of Theorem~\ref{uo-o-fin-dim}; observe that tails
  of the double sequence $(y_{n,m})$ there are not dominable.
\end{proof}

\section{Bornological convergences}
\label{sec:born-conv}
Let $X$ be a set. A collection $\mathcal B$ of subsets of $X$ is a
\term{bornology} if it is closed under taking subsets and finite
unions and $X=\bigcup\mathcal B$. Members of a bornology are
referred to as \term{bounded} sets. A subset $\mathcal B_0$ of a
bornology $\mathcal B$ is a \term{base} of $\mathcal B$ if every set
in $\mathcal B$ is contained in a set from~$\mathcal
B_0$. It is easy to see that a collection $\mathcal B_0$ of subsets of
$X$ is a base for some bornology iff $X=\bigcup\mathcal B_0$ and
for every $A_1,A_2\in\mathcal B_0$ there exists $A\in\mathcal B_0$
such that $A_1\cup A_2\subseteq A$.

A bornology $\mathcal B$ on a vector space $X$ is said to be a
\term{linear} or \term{vector bornology} if
\begin{enumerate}
\item $A_1,A_2\in\mathcal B$ implies $A_1+A_2\in\mathcal B$;
\item if $A\in\mathcal B$ then $\alpha A\in\mathcal B$ for every
  non-zero scalar~$\alpha$;
\item if $A\in\mathcal B$ then the balanced (or circled) hull of $A$
  is in~$\mathcal B$.
\end{enumerate}
A vector space equipped with a linear bornology is said to be a
\term{bornological vector space}. A linear operator $T\colon X\to Y$
between two bornological vector spaces is said to be \term{bounded} if
it maps bounded sets to bounded sets. We refer the reader
to~\cite{Hogbe-Nleng:77} for further details on bornological spaces.

\begin{example}
  In a topological vector space, the bounded sets form a linear
  bornology. In the case of a normed space, the set of all balls
  centred at the origin of positive integer radius forms a countable
  base for this bornology.
\end{example}

\begin{example}
  In a vector lattice, all order bounded sets form a linear
  bornology. The set of all order intervals of the form $[-u,u]$ with
  $u\ge 0$ is a base of this bornology.
\end{example}

\begin{example}
  Let $(X,\lambda)$ be a convergence vector space. All
  $\lambda$-bounded sets form a linear bornology on~$X$. If
  $T\colon X\to Y$ is a continuous linear operator between two
  convergence vector spaces then $T$ is clearly bounded with respect to the
  induced bornologies.
\end{example}

\begin{example}\label{findim-born}
  Let $X$ be a vector space, then the set of all bounded subsets of
  finite-dimensional subspaces of $X$ is a linear bornology in~$X$.
\end{example}

Let $(X,\mathcal B)$ be a bornological vector space. We define the
\term{bornological convergence} structure induced by $\mathcal B$ on
$X$ as follows: $x_\alpha\goesmuB 0$ if there exists $A\in\mathcal B$
such that for every $\varepsilon>0$ the set $\varepsilon A$ contains a
tail of the net; we define $x_\alpha\goesmuB x$ if $x_\alpha-x\goesmuB
0$. It is clear that in the preceding definition one can replace the
bornology $\mathcal B$ with any base of it. In particular, it suffices
to consider only circled sets in~$\mathcal B$. In terms of filters,
$\mathcal F\goesmuB 0$ if $\mathcal F\supseteq[\mathcal N_0A]$ for
some $A\in\mathcal B$; here $\mathcal N_0$ stands for the usual base
of zero neighbourhoods in the scalar field. Indeed, for a net
$(x_\alpha)$ and a circled set $A\in\mathcal B$, $\varepsilon A$
contains a tail of the net for every $\varepsilon$ iff
$\mathcal N_0A\subseteq[x_\alpha]$.

\begin{example}
  If $\mathcal B$ is the bornology of norm bounded sets in a normed
  space then $\mu_{\mathcal B}$ is norm convergence. If $\mathcal B$
  is the bornology of all order bounded sets in a vector lattice then
  $\mu_{\mathcal B}$ is relative uniform convergence.
\end{example}

In the following theorem, we collect a few basic properties of~$\mu_{\mathcal B}$.

\begin{theorem}\label{muB-conv}
  Let $(X,\mathcal B)$ be a bornological vector space. Then
  \begin{enumerate}
  \item\label{muB-conv-lin} $\mu_{\mathcal B}$ is a linear
    convergence;
  \item\label{muB-conv-Haus} $\mu_{\mathcal B}$ is Hausdorff iff
  $\mathcal B$ contains no non-trivial subspaces;
  \item\label{muB-conv-bdd} a set $A$ belongs to $\mathcal B$ iff it
    is $\mu_{\mathcal B}$-bounded;
  \item\label{muB-conv-lbdd}  $\mu_{\mathcal B}$ is locally bounded;
  \item\label{muB-conv-str} $\mu_{\mathcal
        B}$ is the strongest linear convergence structure for which
      every set in $\mathcal B$ is bounded.
  \end{enumerate}
\end{theorem}

\begin{proof}
  \eqref{muB-conv-lin}
  Suppose that $x_\alpha\goesmuB x$ and $y_\alpha\goesmuB y$; we need
  to check that $x_\alpha+y_\alpha\goesmuB x+y$. Replacing $x_\alpha$
  and $y_\alpha$ with $x_\alpha-x$ and $y_\alpha-y$, we may assume that
  $x=y=0$. Find $A,B\in\mathcal B$ such that for every
  $\varepsilon>0$ there exist $\alpha_0$ such that
  $(x_\alpha)_{\alpha\ge\alpha_0}\subseteq \varepsilon A$ and
  $(y_\alpha)_{\alpha\ge\alpha_0}\subseteq \varepsilon B$. It follows
  that for all $\alpha\ge\alpha_0$ we have 
  $x_\alpha+y_\alpha\in \varepsilon A+\varepsilon
  B=\varepsilon(A+B)$. Since $A+B\in\mathcal B$, we have
  $x_\alpha+y_\alpha\goesmuB 0$.

  Suppose now that $x_\alpha\goesmuB x$ and $r_\alpha\to r$. Find a
  circled set $A\in\mathcal B$ such that for every $\varepsilon>0$
  there exists $\alpha_0$ such that $x_\alpha-x\in\varepsilon A$ for
  all $\alpha>0$. Let $B$ be the circled hull of~$\{x\}$. Then $B$ is
  bounded and so is $A+B$.  Fix $\varepsilon>0$. Find $\alpha_0$ such
  that $\abs{r_\alpha-r}<\varepsilon$ for all $\alpha\ge\alpha_0$. In
  particular, $\abs{r_\alpha}\le\abs{r}+\varepsilon$. Passing to a
  further tail, we may assume that
  $x_\alpha-x\in\frac{\varepsilon}{\abs{r}+\varepsilon}A$ for all
  $\alpha\ge\alpha_0$.  We then have
  \begin{math}
    r_\alpha(x_\alpha-x)\in
    r_\alpha\frac{\varepsilon}{\abs{r}+\varepsilon}A
    \subseteq\varepsilon A
  \end{math}
  and
  \begin{math}
    (r_\alpha-r)x\in\varepsilon B.
  \end{math}
  It follows that 
  \begin{math}
    r_\alpha x_\alpha-rx=r_\alpha(x_\alpha-x)+(r_\alpha-r)x\in\varepsilon(A+B).
  \end{math}

  \eqref{muB-conv-Haus} Suppose that $\mathcal B$ contains a
  non-trivial subspace. In particular, it contains $\Span\{a\}$ for
  some non-zero $a\in X$. Then the constant $a$ sequence converges
  both to $a$ and to $0$ in~$\mu_{\mathcal B}$, hence
  $\mu_{\mathcal B}$ fails to be Hausdorff. Conversely, suppose that
  $\mu_{\mathcal B}$ is not Hausdorff. Then there exists a net
  $(x_\alpha)$ such that $x_\alpha\goesmuB a$ and $x_\alpha\goesmuB b$
  for some $a\ne b$. Subtracting $b$ from everything, we may assume
  that $b=0$. It follows that there is a circled set $A\in\mathcal B$
  such that for every $\varepsilon>0$ the set $\varepsilon A$ contains
  both $x_\alpha$ and $x_\alpha-a$ for all sufficiently
  large~$\alpha$. This yields $a\in 2\varepsilon A$. Since
  $\varepsilon$ is arbitrary, we have $\Span\{a\}\subseteq A$, hence
  $\Span\{a\}\in\mathcal B$.

  \eqref{muB-conv-bdd} If $A\in\mathcal B$ then $[\mathcal
  N_0A]\goesmuB 0$, which means that $A$ is  $\mu_{\mathcal B}$-bounded.
  To prove the converse, suppose that $A$ is
  $\mu_{\mathcal B}$-bounded. Let $r_\alpha\to 0$ in
  $\mathbb K\setminus\{0\}$, then
  $r_\alpha A\xrightarrow{\mu_{\mathcal B}}0$. It follows that there
  exists $B\in\mathcal B$ such that $r_\alpha A\subseteq B$ for all
  sufficiently large~$\alpha$; this yields $A\in\mathcal B$.

  \eqref{muB-conv-lbdd} follows immediately from~\eqref{muB-conv-bdd}. 

  \eqref{muB-conv-str} Let $\lambda$ be a linear convergence structure
  for which every set in $\mathcal B$ is bounded. If
  $\mathcal F\goesmuB 0$ then $[\mathcal N_0A]\subseteq\mathcal F$ for
  some $A\in\mathcal B$. Since $A$ is $\lambda$-bounded, we have
  $\mathcal F\goesl 0$.
\end{proof}

Summarizing, every linear bornology $\mathcal B$ yields a linear
convergence structure~$\mu_{\mathcal B}$, while every linear
convergence structure $\lambda$ yields a linear bornology, namely, the
bornology of all $\lambda$-bounded sets. These two operations are not
quite inverses of each other. If we start with a linear bornology
$\mathcal B$ and form the associated convergence~$\mu_{\mathcal B}$,
then by Proposition~\ref{muB-conv}\eqref{muB-conv-bdd}, the bounded
sets for $\mu_{\mathcal B}$ are exactly~$\mathcal B$, so we get back
where we started.

However, if we start with a linear convergence~$\lambda$, and put
$\mathcal B$ to be the bornology of all $\lambda$-bounded sets, then
the convergence $\mu_{\mathcal B}$ is the \term{Mackey modification}
$\lambda_m$ of~$\lambda$. We generally have $\lambda\le\lambda_m$. A
linear convergence structure $\lambda$ is said to be \term{Mackey} if
$\lambda=\lambda_m$. See~\cite{Beattie:02,OBrien:23} for further
information about Mackey convergences and Mackey modifications;
including Proposition~3.7.16 in~\cite{Beattie:02}, which provides an
intrinsic characterization of Mackey convergences.

In an (Archimedean) vector lattice~$X$, order bounded sets are exactly
the sets that are bounded with respect to order convergence; it
follows that the Mackey modification of order convergence is relative
uniform convergence and therefore order bounded sets are exactly the
sets that are bounded with respect to relative uniform convergence.
Combining this with Theorem~\ref{uo-bdd-dom}, we immediately get the
following:

\begin{proposition}
  The Mackey modification of uo-convergence on a vector lattice $X$ is
  the restriction to $X$ of the relative uniform convergence on~$X^u$.
\end{proposition}

It follows from the preceding discussion that for a linear
bornology~$\mathcal B$, the Mackey modification of $\mu_{\mathcal B}$
is again~$\mu_{\mathcal B}$. Now Proposition~3.9 of~\cite{OBrien:23}
yields the following:

\begin{proposition}\label{born-cont}
  For a linear operator $T\colon X\to Y$  between bornological
  spaces, $T$ is bounded iff it is continuous with respect to the
  convergences induced by the bornologies.
\end{proposition}

\subsection*{Locally solid bornologies.} A linear bornology on a
vector lattice is said to be \term{locally solid} if it has a base of
solid sets. It is easy to see that, in this case, the solid hull of
every bounded set is bounded. It is clear that the bornology of all
norm bounded sets in a normed lattice is locally solid. The bornology
of all order bounded sets in a vector lattice is locally solid;
moreover, it is the least locally solid bornology. It is also easy to
see that dominable sets form a locally solid bornology.  It is
straightforward that if $\mathcal B$ is a locally solid bornology then
$\mu_{\mathcal B}$ is a locally solid convergence.

\begin{proposition}
  Let $(X,\eta)$ be a locally solid convergence space. The collection
  of all $\eta$-bounded sets is a locally solid bornology. Every order
  bounded set is $\eta$-bounded.
\end{proposition}

\begin{proof}
  Suppose that $A$ is $\eta$-bounded. Then $[\mathcal N_0A]\goeseta
  0$. Since $\eta$ is locally solid, $\Sol(\mathcal N_0A)\goeseta
  0$. It is easy to see that
  $\Sol(\mathcal N_0A)=\bigl[\mathcal N_0\Sol(A)\bigr]$. Hence,
  $\bigl[\mathcal N_0\Sol(A)\bigr]\goeseta 0$ and, therefore, $\Sol(A)$ is
  $\eta$-bounded. The second claim is straightforward.
\end{proof}

\begin{question}
  Recall that a subset $A$ in a normed lattice $X$ is \term{almost
    order bounded} if for every $\varepsilon>0$ there exists
  $u\in X_+$ such that $A\subseteq[-u,u]+\varepsilon B_X$. It is easy
  to see that almost order bounded sets form a locally solid
  bornology.  It would be interesting to describe the bornological
  convergence generated by the bornology of almost order bounded
  sets. Since this bornology contains all order bounded sets and is
  contained in the bornology of norm bounded sets, the corresponding
  bornological convergence is weaker than relative uniform
  convergence but stronger than norm convergence. Note that the
  bornology of almost order bounded sets may be generalized to locally
  solid topological vector lattices.
\end{question}

\begin{example}\label{ex:CK-equi}
  The collection of all equicontinuous sets in
  $C(\Omega)$, for a Tychonoff space~$\Omega$, is a linear
  bornology. We claim that if this bornology is locally solid then
  $\Omega$ is discrete (and then every set in $C(\Omega)$ is
  equicontinuous). Indeed, the solid hull of $\one$ is $[-\one,\one]$ and it
  should be equicontinuous. Fix $t\in\Omega$ and find its open
  neighborhood $U$ such that $\bigabs{f(s)-f(t)}<\frac{1}{2}$, for every
  $s\in U$ and every $-\one\le f\le \one$. If there were
  $s\in U\setminus\{t\}$, we could then find $f\in C(\Omega)$ with
  $-\one\le f\le \one$, $f(t)=1$ and $f(s)=-1$, which would lead to a
  contradiction. Hence, $U=\{t\}$.

  More generally, let $X$ be a convergence space and $A\subseteq
  C(X)$. We say that $A$ is \term{equicontinuous} if $x_\alpha\to x$
  in $X$ and $\varepsilon>0$ then there exists $\alpha_0$ such that
  $\bigabs{f(x_\alpha)-f(x)}<\varepsilon$ for all $\alpha\ge\alpha_0$
  and $f\in A$. It can be easily verified that this is consistent with
  the filter definition of equicontinuity in Definition~2.4.1
  of~\cite{Beattie:02}. It is also easy to see that equicontinuous
  sets form a linear bornology in $C(X)$.

  Suppose now that $X$ is a convergence vector space. Let
  $\mathcal L_cX$ be the space of all continuous linear functionals
  on~$X$. Equicontinuous subsets in $\mathcal L_cX$ form a linear
  bornology because it is a restriction of the bornology of
  equicontinuous sets in $C(X)$ to $\mathcal L_cX$.
\end{example}

\subsection*{Finite-dimensional convergence.}

Let $X$ be a vector space. Define \term{finite-dimensional
  convergence} on $X$ as follows: $x_\alpha\goesfd 0$ if a tail of
$(x_\alpha)$ is contained in a finite-dimensional subspace of $X$ and
converges to zero there in the usual sense; put $x_\alpha\goesfd x$
when $x_\alpha-x\goesfd 0$.

Let $(e_i)_{i\in I}$ be a Hamel basis in~$X$; we write $P_ix$ for the
$i$-th coordinate of~$x$. It is easy to see that if
$x_\alpha\goesfd x$ then $P_ix_\alpha\to P_ix$ for every~$i$, i.e.,
$(x_\alpha)$ converges to $x$ coordinate-wise. The converse is false:
let $X$ be a vector space with a countable Hamel basis
$(e_i)_{i\in\mathbb N}$, then, viewed as a sequence, $(e_i)$ converges
to zero coordinate-wise, but not in~$\mathrm{fd}$. 

Here we collect a few basic facts about this
convergence:

\begin{proposition}
  Finite-dimensional convergence on a vector space $X$ is Hausdorff
  and complete. It is the strongest linear convergence on~$X$. It
  agrees with the convergence induced by the bornology of all bounded
  subsets of finite-dimensional subspaces.
\end{proposition}

\begin{proof}
  It is easy to see that $\mathrm{fd}$ convergence is exactly the
  convergence induced by the bornology of all bounded subsets of
  finite-dimensional subspaces,
  cf.~Example~\ref{findim-born}. Therefore, it is Hausdorff by
  Theorem~\ref{muB-conv}\eqref{muB-conv-Haus}. It is the strongest
  linear convergence on $X$ because every linear
  convergence a on finite-dimensional vector spaces is weaker than
  coordinate-wise convergence.

  To show that $\mathrm{fd}$ convergence is complete, suppose that
  $(x_\alpha)$ is a Cauchy net in~$X$. Then a tail of this net is
  contained in a finite-dimensional subspace, and is Cauchy there in
  the usual sense. Hence this tail is convergent (in the subspace and,
  therefore, with respect to $\mathrm{fd}$ convergence).
\end{proof}

In an ordered vector space, finite dimensional convergence can be
used to characterize the Archimedean property.

\begin{proposition}
  Let $X$ be an ordered vector space. The following are equivalent:
  \begin{enumerate}
  \item\label{a-A} $X$ is Archimedean;
  \item\label{a-ru} $X_+$ is relatively uniformly closed;
  \item\label{a-a} $X_+$ is $\mathrm{fd}$-closed.
  \end{enumerate}
\end{proposition}

\begin{proof}
  \eqref{a-A}$\Rightarrow$\eqref{a-ru}
  Suppose $X$ is Archimedean, and $x_\alpha\goesu x$ for some net
  $(x_\alpha)$ in~$X^+$. It follows that there exists a regulator $u\ge
  0$ such that for every $n\in\mathbb N$ there
  exists $\alpha$ such that
  \begin{math}
    x_\alpha-\frac{u}{n}\le x\le x_\alpha+\frac{u}{n}.
  \end{math}
  Then $x\ge -\frac{u}{n}$ for every~$n$. It follows that
  $x\ge 0$, hence $x\in X_+$.

  \eqref{a-ru}$\Rightarrow$\eqref{a-a} is trivial as finite dimensional
  convergence is the strongest linear convergence on~$X$.

  \eqref{a-a}$\Rightarrow$\eqref{a-A} Suppose that $x\in X$ and
  $y\in X_+$ satisfy $nx\le y$ for all $n\in\mathbb N$. It follows
  that $x-\frac{y}{n}\le 0$ for all~$n$. Since the sequence
  $\bigl(x-\frac{y}{n}\bigr)$ is contained in the linear span of $x$
  and~$y$, passing to the $\rm{fd}$-limit, we conclude that $x\le 0$.
\end{proof}

\begin{proposition}
  Let $X$ be an ordered vector space and $e\in X_+$. The following are
  equivalent:
  \begin{enumerate}
  \item\label{e-e} $e$ is a strong order unit;
  \item\label{e-ru} $X_+$ is an $\mathrm{ru}$-neighborhood of~$e$;
  \item\label{e-a}  $X_+$ is an $\mathrm{fd}$-neighborhood of~$e$.
  \end{enumerate}
\end{proposition}

\begin{proof}
  \eqref{e-e}$\Rightarrow$\eqref{e-ru} Suppose that
  $x_\alpha\goesu e$. Without loss of generality, we may take $e$ as a
  regulator of this convergence. Then for all sufficiently large
  $\alpha$ we have $-\frac{e}{2}\le x_\alpha-e\le \frac{e}{2}$ and,
  therefore, $x_\alpha \ge\frac{e}{2}\ge 0$.

  \eqref{e-ru}$\Rightarrow$\eqref{e-a} is trivial.

  \eqref{e-a}$\Rightarrow$\eqref{e-e} Let $x\in X$. It follows from
  $\frac{x}{n}+e\goesfd e$ that for some sufficiently large $n$ we have
  $\frac{x}{n}+e\ge 0$ and, therefore, $x\ge -ne$. Applying the same
  argument to~$-x$, we find $m$ such that $x\le me$.
\end{proof}

Since fd is the strongest linear convergence, the two preceding
propositions yield the following:

\begin{corollary}
  Let $X$ be an ordered vector space equipped with a linear
  convergence. If $X_+$ is closed then $X$ is Archimedean. If $X_+$ is
  a neighborhood of $e$ then $e$ is a strong unit.
\end{corollary}

\section{Bounded modification}
\label{sec:bddmod}

Let $(X,\lambda)$ be a convergence space and $\mathcal B$ a bornology
on~$X$. One can define a new convergence $\lambda_\mathcal B$ as
follows: $x_\alpha\goeslB x$ if $x_\alpha\goesl x$ and a tail of
$(x_\alpha)$ is contained in a set in~$\mathcal B$.  In terms of
filters, we define $\mathcal F\goeslB x$ if $\mathcal F\goesl x$ and
$\mathcal F\cap\mathcal B\ne\varnothing$. It is easy to see that
$x_\alpha\goeslB x$ iff $[x_\alpha]\goeslB x$. We call this
convergence the \term{bounded modification} of $\lambda$
by~$\mathcal B$, or the $\mathcal B$-modification of~$\lambda$.  In
\cite[p.~3]{Beattie:02}, this convergence is called the
\term{specified sets convergence} for $\lambda$ with respect
to~$\mathcal B$.  This concept is motivated by the notion of a mixed
topology, see, e.g.,~\cite{Cooper:87,Conradie:06}. The notation
$\lambda_{\mathcal B}$ should not be confused
with~$\mu_{\mathcal B}$, which we have used to denote the bornological
convergence associated with the bornological vector space
$(X,\mathcal B)$. We will use $\mu_{\mathcal B}$ only for bornological
convergences.

The following properties of
$\lambda_{\mathcal B}$ are straightforward:

\begin{proposition}\label{bm}
  Let $(X,\lambda)$ be a convergence space and
  $\mathcal B$ a bornology on~$X$. Then
  \begin{enumerate}
  \item $\lambda_{\mathcal B}$ is the strongest convergence on $X$
    that agrees with $\lambda$ on sets in~$\mathcal B$.
  \item\label{bm-final} $\lambda_{\mathcal B}$ is the
  strongest convergence $\eta$ on $X$ for which the
  embedding $(A,\lambda_{|A})\hookrightarrow (X,\eta)$ is continuous
  for every $A\in\mathcal B$.
\item $(\lambda_{\mathcal B})_{\mathcal B}=\lambda_{\mathcal B}$.
  If $\mathcal C$ is another bornology on $X$ then
    $(\lambda_{\mathcal B})_{\mathcal C}=\lambda_{\mathcal
      B\cap\mathcal C}$.
  \item If $f\colon X\to Y$ is a function from $X$ to another
    convergence space~$Y$, then $f$ is $\lambda_{\mathcal
      B}$-continuous iff its restriction to $A$ is $\lambda$-continuous
    for every $A\in\mathcal B$.
  \end{enumerate}
\end{proposition}

If $X$ is a vector space and both $\lambda$ and $\mathcal B$ are
linear then so is $\lambda_{\mathcal B}$
(cf. \cite[p.~86]{Beattie:02}).
Indeed, if $x_\gamma\goeslB x$ and
$y_\gamma\goeslB y$ then $x_\gamma\goesl x$, $y_\gamma\goesl y$ and
the nets have tails in some sets $A$ and $B$ from~$\mathcal B$,
respectively. Then $x_\gamma+y_\gamma\goesl x+y$ and the net
$(x_\gamma+y_\gamma)$ has a tail in $A+B$, hence
$x_\gamma+y_\gamma\goeslB x+y$. The proof for scalar multiplication is
similar.

\begin{proposition}\label{bdd-bdd}
  Let $(X,\lambda)$ be a convergence vector space and $\mathcal B$ a
  linear bornology on~$X$. A set $A$ is $\lambda_{\mathcal B}$-bounded
  iff $A$ is $\lambda$-bounded and $A\in\mathcal B$.
\end{proposition}

\begin{proof}
  Let $r_\gamma\to 0$ in~$\mathbb K$. It is easy to see
  that $r_\gamma A\goeslB 0$ iff $r_\gamma A\goesl 0$ and
  $A\in\mathcal B$.
\end{proof}

Given a convergence structure $\lambda$ and a bornology~$\mathcal B$,
we say that $\lambda$ is \term{locally $\mathcal B$-bounded} if every
$\lambda$-convergent net has a $\mathcal B$-bounded tail or,
equivalently, if $\lambda=\lambda_{\mathcal B}$.
A topological space is locally compact iff its convergence structure
is locally $\mathcal B$-bounded, where $\mathcal B$ is the bornology
of relatively compact sets.  A linear convergence $\lambda$ is locally
bounded precisely when it is locally $\mathcal B$-bounded, where
$\mathcal B$ is the bornology of all $\lambda$-bounded sets.  If
$(X,\tau)$ is a topological vector space and $\mathcal B$ is a linear
bornology on~$X$, $\tau$ is locally $\mathcal B$-bounded iff every
point has a $\mathcal B$-bounded neighbourhood. Indeed, let
$\mathcal N_x$ be the set of all neighborhoods of~$x$. If
$\tau=\tau_{\mathcal B}$ then $\mathcal N_x\goestau x$ implies
$\mathcal N_x\xrightarrow{\tau_{\mathcal B}}x$ and, therefore,
$\mathcal N_x\cap\mathcal B\ne\varnothing$; the converse is trivial.

Proposition~\ref{bdd-bdd} yields the following:

\begin{corollary}\label{lB-bdd-B}
  Given a linear bornology $\mathcal B$ and a locally
  $\mathcal B$-bounded linear convergence~$\lambda$, every
  $\lambda$-bounded set is in~$\mathcal B$.
\end{corollary}

Let $(X,\lambda)$ be a convergence vector space.  Recall that
$\lambda_m$ stands for the Mackey modification of~$\lambda$; we always
have $\lambda\le\lambda_m$. Recall also that $\mu_{\mathcal B}$ stands
for the bornological convergence induced by
a bornology $\mathcal B$ as in Section~\ref{sec:born-conv}.

\begin{proposition}\label{compat}
  Let $(X,\lambda)$ be a convergence vector space, and $\mathcal B$ a
  linear bornology on $X$ such that every set in $\mathcal B$ is
  $\lambda$-bounded. Then
  \begin{enumerate}
  \item\label{compat-bdd} A subset $A$ of $X$ is
    $\lambda_{\mathcal B}$-bounded iff $A\in\mathcal B$;
  \item\label{compat-lbdd} $\lambda_{\mathcal B}$ is locally bounded;
  \item\label{compat-dom} $\lambda\le\lambda_{\mathcal
      B}\le\mu_{\mathcal B}$ and $\lambda_m\le \mu_{\mathcal B}$;
  \item\label{compat-cont} Every $\lambda_{\mathcal B}$-continuous
    functional on $X$ is $\mathcal B$-bounded.
  \end{enumerate}
\end{proposition}

\begin{proof}
  \eqref{compat-bdd} follows from Proposition~\ref{bdd-bdd}.

  \eqref{compat-lbdd} If $x_\gamma\goeslB 0$ then it has a tail in
  some $A\in\mathcal B$, which is $\lambda_{\mathcal B}$-bounded
  by~\eqref{compat-bdd}.

  \eqref{compat-dom} $\lambda\le\lambda_{\mathcal B}$ is
  immediate. The rest follows from  \eqref{compat-bdd},
  Proposition~\ref{muB-conv}\eqref{muB-conv-str}, and the fact that
  every set in $\mathcal B$ is $\lambda$-bounded.

  \eqref{compat-cont} Let $f\colon X\to\mathbb K$ be a
  $\lambda_{\mathcal B}$-continuous linear functional. It follows from
  \eqref{compat-dom} that $f$ is $\mu_{\mathcal B}$-continuous. Now
  Proposition~\ref{born-cont} yields that $f$ is $\mathcal B$-bounded.
\end{proof}

\begin{corollary}\label{BC-locbdd}
  Let $(X,\lambda)$ be a convergence vector space, $\mathcal C$ the
  bornology of $\lambda$-bounded sets and $\mathcal B$ a vector
  bornology in $X$ such that $\mathcal B\subseteq \mathcal C$. Then
  $\lambda_\mathcal B=\lambda$ iff $\lambda$ is locally bounded and
  $\mathcal B=\mathcal C$.
\end{corollary}

\begin{example}\label{lllB}
  Let~$\lambda$, $\mathcal B$, and $\mathcal C$ be as in the preceding
  corollary. Even if $\lambda$ is locally bounded, we need not have
  $\lambda=\lambda_{\mathcal B}$ if $\mathcal B\ne\mathcal C$. Indeed,
  take $X=\ell_1$, let $\lambda$ be norm convergence on~$X$, and
  $\mathcal B$ the bornology of all order bounded sets. It is easy to
  see that $\lambda<\lambda_{\mathcal B}$.
\end{example}

\begin{example}
  We observed in Proposition~\ref{compat} that in general
  $\lambda\le\lambda_{\mathcal B}$ and $\lambda\le\lambda_m$. What is
  the relationship between $\lambda_m$ and $\lambda_{\mathcal B}$? If
  $\lambda$ is Mackey (i.e., $\lambda_m=\lambda$), it is locally
  bounded by \cite[Proposition~3.7.16]{Beattie:02}, so if we take
  $\mathcal B$ to be the bornology of all $\lambda$-bounded sets then
  we have $\lambda_m=\lambda=\lambda_{\mathcal B}$ by Corollary~\ref{BC-locbdd}

  On the other hand, in Example~\ref{lllB}, we have
  $\lambda_m=\lambda<\lambda_{\mathcal B}$.  Furthermore, if we let
  $\lambda$ be order convergence on an infinite-dimensional vector
  lattice, and $\mathcal B$ the bornology of all order bounded (that
  is, $\lambda$-bounded) sets, then
  $\lambda_{\mathcal B}=\lambda$ while  $\lambda_m$ is relative
  uniform convergence, so $\lambda_m>\lambda_{\mathcal B}$.
\end{example}

\begin{example}\label{ex:tvs-equi}
  (Cf.~Example~\ref{ex:CK-equi}.)
  Let $X$ be a topological vector space. Let $\mathcal B$ be the
  vector bornology on $X^*$ consisting of all equicontinuous sets. It
  is easy to see that a subset $G$ of $X^*$ is equicontinuous iff it
  is contained in a polar of a neighborhood of zero in~$X$, hence such
  polars form a base of~$\mathcal B$.

  We write $\mathrm{w}^*$ for weak*-convergence on~$X^*$, i.e.,
  $\mathrm{w}^*=\sigma(X^*,X)$; we write $\mathrm{w}^*_{\mathcal B}$
  for its bounded modification by~$\mathcal B$. Every equicontinuous
  set in $X^*$ is weak*-bounded, so that the assumptions of
  Proposition~\ref{compat} are satisfied. Therefore,
  $\mathrm{w}^*_{\mathcal B}$-bounded sets are exactly the sets
  in~$\mathcal B$. By Proposition~8.4 in~\cite{OBrien:23},
  $\mathrm{w}^*_{\mathcal B}$ is continuous convergence on~$X^*$.

  Suppose now that, in addition, $X$ is a Hausdorff locally convex
  topological vector space. Then the space of all weak*-continuous
  linear functionals on $X^*$ can be identified with~$X$. Moreover,
  the set of all linear functionals on $X^*$ which are
  $\mathrm{w}^*$-continuous on equicontinuous subsets of $X^*$ equals
  the space of all $\mathrm{w}^*_{\mathcal B}$-continuous functionals
  on~$X^*$, which in turn equal to the second dual of $X$ in the sense
  of \cite{Beattie:02} and \cite{OBrien:23}. By \cite[Theorem
  4.3.20]{Beattie:02} or \cite[Theorem 8.12]{OBrien:23}, this second
  dual may be identified with the completion of~$X$; thus we have
  recovered Grothendieck's completion theorem, see, e.g., Theorem~2 in
  Chapter~III, Section~3.6 in~\cite{Bourbaki:87}.
\end{example}

\begin{example}
  The collection of all finite-dimensional subspaces of a vector space
  $X$ is a base for a bornology; let $\mathcal B$ be the resulting
  bornology. It is easy to see that the bounded modification of any
  Hausdorff linear convergence on $X$ by $\mathcal B$ is
  finite-dimensional convergence because all linear Hausdorff
  convergences agree on finite dimensional subspaces (see
  \cite[Theorem~3.3.19]{Beattie:02}).
\end{example}

The following result shows that a bounded
modification of a topological convergence is typically
non-topological.

\begin{proposition}\label{top-B}
  Let $(X,\tau)$ be a topological space and $\mathcal B$ a bornology
  on~$X$. If the bounded modification $\tau_{\mathcal B}$ of
  $\tau$-convergence by $\mathcal B$ is topological, then for every
  $x\in X$ there exists $B\in\mathcal B$ such that every
  eventually $\mathcal B$-bounded net that $\tau$-converges to $x$
  is eventually in~$B$.
\end{proposition}

\begin{proof}
  Suppose that $\tau_{\mathcal B}$ is topological. Let $\mathcal F$ be
  the filter of all neighborhoods of $x$ in this topology. Then
  $\mathcal F\xrightarrow{\tau_{\mathcal B}}x$, hence there exists a
  $B\in\mathcal B$ which is a neighborhood of $x$
  in~$\tau_{\mathcal B}$.  If $(x_\alpha)$ is net such that
  $x_\alpha\xrightarrow{\tau}x$ and such that a tail of it is in
  $\mathcal B$ then $x_\alpha\xrightarrow{\tau_{\mathcal B}}x$, hence
  $(x_\alpha)$ has a tail in~$B$.
\end{proof}

For a normed space~$X$, we write $S_X$ for the unit sphere of~$X$.

\begin{corollary}\label{top-Bn}
  Let $X$ be a normed space and $\tau$ a linear topology on~$X$. The
  bounded modification $\tau_{\mathcal B}$ of $\tau$ by the bornology
  $\mathcal B$ of all norm bounded sets is not topological provided
  that $0\in\overline{S_X}^\tau$.
\end{corollary}

\begin{proof}
  Let $B\in\mathcal B$. Take $r>0$ so that $B\cap
  (rS_X)=\varnothing$. Since $0\in\overline{S_X}^\tau$, one can find a
  net $(x_\alpha)$ is $rS_X$ such $x_\alpha\goestau 0$. Then
  $(x_\alpha)$ is bounded, yet misses~$B$. By Proposition~\ref{top-B},
  $\tau_{\mathcal B}$ fails to be topological.
\end{proof}

\begin{example}\label{ex:norm-equi}
  (Cf. Example~\ref{ex:tvs-equi}.)  Let $X$ be an infinite-dimensional
  normed space with dual~$X^*$. Let $\mathcal B$ be the bornology of
  all norm bounded sets in~$X^*$. These are the equicontinuous sets
  in~$X^*$; it follows that $\mathrm{w}^*_{\mathcal B}$ coincides with
  continuous convergence on~$X^*$. Since $0$ belongs to the
  weak*-closure of $S_{X^*}$, Corollary~\ref{top-Bn} implies that the
  bounded modification $\mathrm{w}^*_{\mathcal B}$ of the
  weak*-topology on $X^*$ by $\mathcal B$ is not topological.
 
  A similar argument shows that if we consider the bornology $\mathcal
  B$ of norm-bounded sets in $X$ then the bounded modification of the
  weak convergence on $X$ by $\mathcal B$ is not topological.

  Note that since every weakly convergent sequence is norm-bounded,
  $\mathrm{w}=\mathrm{w}_{\mathcal B}$ on sequences. This fails for
  nets as there are weakly null nets in $X$ that are not eventually
  norm bounded.
\end{example}

Next, we study the relationship between bounded and topological
modifications. Recall that the topological modification $\mathrm{t}\lambda$
of a convergence structure $\lambda$ is the strongest topology (whose
convergence is) weaker than~$\lambda$; see, e.g., Proposition~3.5
in~\cite{OBrien:23}.

\begin{proposition}\label{top-bor-mod} 
  Let $\mathcal B$ be a bornology on a set~$X$.
  \begin{enumerate}
  \item\label{top-bor-mod-conv} If $\lambda$ is a convergence
    structure on $X$ then $\mathrm{t}(\lambda_{\mathcal B})$ is the
    strongest topology on $X$ that is weaker that $\lambda$ on sets
    from~$\mathcal B$.
  \item\label{top-bor-mod-top} If $\tau$ is a topology on $X$ then
    $\mathrm{t}(\tau_{\mathcal B})$ is the strongest topology on $X$ that
    agrees with $\tau$ on sets from~$\mathcal B$.
  \end{enumerate}
\end{proposition}

\begin{proof}
  \eqref{top-bor-mod-conv} is straightforward.



  \eqref{top-bor-mod-top} First, we will show that $\mathrm{t}(\tau_{\mathcal
    B})$ agrees with $\tau$ on sets from~$\mathcal B$.
  By the discussion before the proposition,
  $\tau_{\mathcal B}\ge \mathrm{t}(\tau_{\mathcal B})$. It follows
  from~\eqref{top-bor-mod-conv} that $\mathrm{t}(\tau_{\mathcal
    B})\ge\tau$. Hence, $\tau_{\mathcal B}\ge \mathrm{t}(\tau_{\mathcal B})\ge\tau$.
  Since $\tau$ and $\tau_{\mathcal B}$ agree on 
  sets from~$\mathcal B$, all the three convergences agree there.

  Let $\sigma$ be a topology on $X$ which agrees with $\tau$ on sets
  from~$\mathcal B$. It follows from~\eqref{top-bor-mod-conv} that
  $\sigma\le \mathrm{t}(\tau_{\mathcal B})$. 
\end{proof}

\begin{proposition}\label{lb-clos}
  Let $(X,\lambda)$ be a convergence space and $\mathcal B$ a
  bornology on~$X$; let $A\subseteq X$. Then
  \begin{enumerate}
  \item\label{lb-clos-lb} $A$ is $\lambda_{\mathcal B}$-closed iff
    $\overline{A\cap B}^{1,\lambda}\subseteq A$ for every
    $B\in\mathcal B$;
  \item\label{lb-clos-tlb} $A$ is $(\mathrm{t}\lambda)_{\mathcal B}$-closed iff
    $\overline{A\cap B}^\lambda\subseteq A$ for every
    $B\in\mathcal B$.  
  \end{enumerate}
\end{proposition}

\begin{proof}
  \eqref{lb-clos-lb} is straightforward; \eqref{lb-clos-tlb} follows
  from \eqref{lb-clos-lb} and the observation that
  $\overline{C}^{1,\mathrm{t}\lambda}=\overline{C}^{\mathrm{t}\lambda}=\overline{C}^\lambda$
  for every set~$C$.
\end{proof}

Clearly,
$\mathrm{t}(\lambda_{\mathcal
  B})=\mathrm{t}\bigl((\mathrm{t}\lambda)_{\mathcal B}\bigr)$ iff
$\lambda_{\mathcal B}$ and $(\mathrm{t}\lambda)_{\mathcal B}$ have the
same closed sets. By Proposition~\ref{lb-clos}, this boils down to the
difference between the $\lambda$-closure and the $\lambda$-adherence
of certain sets. We now present an example to show that we may have
$\mathrm{t}(\lambda_{\mathcal
  B})\ne\mathrm{t}\bigl((\mathrm{t}\lambda)_{\mathcal B}\bigr)$.

\begin{example}
  Put $X=[0,1]^{2}$, let $\tau$ be the Euclidean topology on~$X$. Let
  $\mathcal D$ be the bornology on $X$ consisting of subsets of all
  sets of the form
  \begin{math}
    \bigl([0,1]\times\{0\}\bigr)\cup\bigl(F\times[0,1]\bigr),
  \end{math}
  where $F$ is a finite subset of $[0,1]$. Let
  $\lambda:=\tau_{\mathcal D}$, and $\mathcal{B}$ the bornology of all
  sets whose intersection with $\mathbb Q\times\{0\}$ is finite. Put
  $A:=\bigl(\mathbb Q\cap[0,1]\bigr)\times[0,1]$. Then $A$ satisfies
  the condition \eqref{lb-clos-lb} in Proposition \ref{lb-clos}, but
  fails \eqref{lb-clos-tlb} for $B=(0,1]\times[0,1]$ because
  $\overline{A\cap B}^{1,\lambda}$ contains
  $\bigl(\mathbb Q\cap[0,1]\bigr)\times\{0\}$, so that
  $\overline{A\cap B}^{2,\lambda}$ contains $[0,1]\times\{0\}$ and,
  therefore, $\overline{A\cap B}^{2,\lambda}$ is not contained in~$A$.
\end{example}

\begin{proposition}\label{tt}
  Let $(X,\lambda)$ be a convergence space and $\mathcal B$ a
  bornology on $X$ with a base $\mathcal B_0$ of $\lambda$-closed
  sets. Then $\mathrm{t}(\lambda_{\mathcal
    B})=\mathrm{t}\bigl((\mathrm{t}\lambda)_{\mathcal B}\bigr)$.
\end{proposition}

\begin{proof}
  Observe that $\lambda_{\mathcal B}$ is the final convergence for the
  family of inclusions $(B,\lambda_{|B})\hookrightarrow X$, where
  $B$ runs over $\mathcal B$ and $\lambda_{|B}$ is the restriction of $\lambda$
  to~$B$. Moreover, since $\mathcal B_0$ is a base of~$\mathcal B$, it
  suffices to consider all $B\in\mathcal B_0$.

  Proposition~1.3.14
  of~\cite{Beattie:02} asserts that topological modification commutes
  with taking the final convergence structure; it follows that
  $\mathrm{t}(\lambda_{\mathcal B})$ is the final convergence structure for the
  inclusions
  $\bigl(B,\mathrm{t}(\lambda_{|B})\bigr)\hookrightarrow X$ with
  $B\in\mathcal B_0$. 

  Let $B\in\mathcal B_0$. Since $B$ is closed, the restriction
  operation commutes with topological modification, i.e.,
  $\mathrm{t}(\lambda_{|B})=(\mathrm{t}\lambda)_{|B}$ by Proposition~1.3.10
  of~\cite{Beattie:02}. It follows that $\mathrm{t}(\lambda_{\mathcal B})$ is
  the final convergence structure for the inclusions
  $\bigl(B,(\mathrm{t}\lambda)_{|B}\bigr)\hookrightarrow X$ as
  $B\in\mathcal B_0$. Applying
  Proposition~\ref{top-bor-mod}\eqref{top-bor-mod-top} with
  $\tau=\mathrm{t}\lambda$, we conclude that the latter final convergence is
  $\mathrm{t}\bigl((\mathrm{t}\lambda)_{\mathcal B}\bigr)$.
\end{proof}

\begin{example}\label{ex:bw*}
  Let $X$ be an infinite-dimensional normed space. As in
  Example~\ref{ex:norm-equi}, the convergence
  $\mathrm{w}^*_{\mathcal B}$ coincides with continuous convergence
  on~$X^*$, and is not topological. This implies that
  $\mathrm{w}^*_{\mathcal B}$ is strictly stronger than
  $\mathrm{bw}^*:=\mathrm{t}(\mathrm{w}^*_{\mathcal B})$. The latter
  is the mixed topology, which is defined as the strongest topology
  that agrees with the weak* topology on norm bounded sets. By
  Banach-Dieudonn\'{e} theorem this topology is the restriction of the
  topology of uniform convergence on compact sets. In particular,
  $\mathrm{bw}^*$ is a linear topology; see, e.g.,
  \cite[I.2.A]{Cooper:87}, Section~3.5 in Chapter~IV
  of~\cite{Bourbaki:87} or \cite{Bruguera:03}.
\end{example}

\subsection*{The locally solid case}

\begin{proposition}
  Let $X$ be a vector lattice, $\lambda$ a locally solid convergence
  on~$X$, and $\mathcal B$ a locally solid bornology on~$X$. Then the
  bounded modification $\lambda_{\mathcal B}$ is again locally
  solid. If, in addition, $\lambda$ is order continuous then so
  is~$\lambda_{\mathcal B}$.
\end{proposition}

\begin{proof}
  The first claim is straightforward. Suppose that $\lambda$ is order
  continuous and $x_\alpha\goeso 0$. Then $x_\alpha\goesl 0$. Passing
  to a tail, we may assume that $(x_\alpha)$ is contained in an order
  interval. Since every order interval is contained in~$\mathcal B$,
  we conclude that $x_\alpha\goeslB 0$.
\end{proof}

\begin{example}\label{cru-bm}
  Remark~\ref{Cc-ru}, Corollary~\ref{C0-ru}, and Theorem~\ref{sc} may
  be interpreted as follows: relative uniform convergence on the appropriate
  spaces of continuous functions may be obtained as a bounded
  modification of ucc convergence. In particular, Theorem~\ref{sc}
  says that if $\Omega$ is locally compact and $\sigma$-compact,
  then relative uniform convergence on $C(\Omega)$ is the
  bounded modification of ucc convergence by the bornology of all ucc
  (or, equivalently, order) bounded sets.
\end{example}

\begin{example}\label{uob-o}
  It is easy to see that the bounded modification of
  uo-convergence by the bornology of all order bounded sets is
  order convergence.
\end{example}

\begin{example}\label{top-bdd-nontop}
  In the preceding example, suppose that $X$ is discrete. In this case,
  by Proposition~\ref{discr-uo-p}, uo-convergence agrees with
  point-wise convergence, hence is topological. On the other hand, 
  order convergence is not topological unless $X$ is
  finite-dimensional; see~\cite{Dabboorasad:20} or Corollary~10.11
  in~\cite{OBrien:23}. This yields another example of a topological
  convergence, which has a non-topological bounded modification.
  (See also Proposition~\ref{top-B}.)
\end{example}

\begin{example}
  Let $X$ be a normed lattice, and consider uo-convergence on~$X$. The
  set of all uo-continuous linear functionals on $X$ is rather small;
  it consists only of linear combinations of the coordinate
  functionals of atoms in~$X$; see, e.g., Proposition~2.2
  in~\cite{Gao:18}. Let $\mathcal B$ be the bornology of all norm
  bounded sets in~$X$. The bounded modification $\rm{uo}_{\mathcal B}$
  yields a richer dual space. It was shown in Theorem~2.3
  in~\cite{Gao:18} that the dual $X^\sim_{\mathrm{uo}}$ of $X$ with
  respect to this convergence is the order continuous part of the
  order continuous dual of~$X$.

  In this example, the condition in Proposition~\ref{compat} that
  every set in $\mathcal B$ should be $\lambda$-bounded is not
  satisfied. It follows from Theorem~\ref{uo-bdd-dom} and
  Proposition~\ref{bdd-bdd} that $(\rm{uo})_{\mathcal B}$-bounded sets
  in $X$ are the sets that are both dominable and norm bounded.
\end{example}

\begin{proposition}\label{uo-utau}
  Let $X$ be a vector lattice with the countable supremum property,
  $\tau$ a Hausdorff locally solid order continuous topology on~$X$,
  and $\mathcal B$ a locally solid bornology on~$X$. Then
  \begin{enumerate}
  \item\label{uo-utau-top} The topological modifications of $(\mathrm{u}\tau)_{\mathcal B}$ and of
    $(\rm{uo})_{\mathcal B}$ are equal and agree with the strongest
    topology coinciding with $\mathrm{u}\tau$ on members of~$\mathcal
    B$;
  \item\label{uo-utau-func} A linear functional on $X$ is
    $(\rm{uo})_{\mathcal B}$-continuous iff it is
    $(\rm{u}\tau)_{\mathcal B}$-continuous.
  \end{enumerate}
\end{proposition}

\begin{proof}
  \eqref{uo-utau-top} By
  Proposition~\ref{top-bor-mod}\eqref{top-bor-mod-conv}, the
  topological modifications of $(\mathrm{uo})_{\mathcal B}$ and of
  $(\mathrm{u}\tau)_{\mathcal B}$ are the strongest topologies on $X$
  that are weaker than $\mathrm{uo}$ and~$\mathrm{u}\tau$,
  respectively, on sets in~$\mathcal B$. But these two topologies
  agree because by Remark~\ref{uo-top-un-B}, for every
  $B\in\mathcal B$, a topology on $B$ is weaker than
  $(\mathrm{uo})_{|B}$ iff it is weaker than $(\mathrm{u}\tau)_{|B}$.
  
   
  \eqref{uo-utau-func} follows from \eqref{uo-utau-top} and
  Proposition~1.3.9(iii) of~\cite{Beattie:02}.
\end{proof}



Analogously to $X^\sim_{\mathrm{uo}}$, we write
$X^\sim_{\mathrm{un}}$ for the dual of $X$ with respect to
$\mathrm{un}_{\mathcal B}$, where $\mathcal B$ is the bornology of the
norm bounded sets on a Banach lattice $X$
and $\mathrm{n}$ is convergence with
respect to the norm on~$X$.

\begin{corollary}
  Let $X$ be an order continuous Banach lattice. Then
  $X^\sim_{\mathrm{uo}}=X^\sim_{\mathrm{un}}$.
\end{corollary}

\begin{example}\label{ex:L01}
  One may ask whether, under the assumptions of
  Proposition~\ref{uo-utau}, the topological modification of
  $(\mathrm{uo})_\mathcal B$ is always a linear topology. The following example
  shows that the answer is ``no''.

  Let $X=L_\infty[0,1]$, $\tau$ be the topology of convergence in
  measure, and $\mathcal B$ the bornology of all order (or,
  equivalently, norm) bounded sets in~$X$. Then $X$, $\tau$, and
  $\mathcal B$ satisfy the assumptions of the
  proposition. It is easy to see that $\tau$ is unbounded and that
  $(\mathrm{uo})_{\mathcal B}=\mathrm{o}$. Hence, it follows from the
  proposition that
  $\mathrm{t}(\mathrm{o})$ is the strongest topology coinciding
  with $\tau$ on bounded sets. Suppose, for the sake of contradiction,
  that $\mathrm{t}(\mathrm{o})$ is
  linear. Then it would be the strongest linear topology coinciding
  with $\tau$ on bounded sets; it would then equal the Mackey topology
  $\tau(L_\infty,L_1)$ by Theorem~5 of~\cite{Nowak:89}; see also a
  discussion before Corollary~3.8 in~\cite{Conradie:06}. However, it
  was shown in Theorem~16 in~\cite{Abela:22} that
  $\tau(L_\infty,L_1)\ne\mathrm{t}(\mathrm{o})$ on~$X$.
\end{example}

\begin{example}
  Let $\tau$ be a locally solid order continuous topology on a vector
  lattice~$X$, let $\mathcal B$ be the bornology of all order bounded
  sets in~$X$. It is easy to see that
  $\tau_{\mathcal B}=(\mathrm{u}\tau)_{\mathcal B}$. Since all
  Hausdorff order continuous locally solid topologies agree on order
  bounded sets (see \cite{Aliprantis:03}, Theorem 4.22),
  $\tau_{\mathcal B}$ is independent of the topology chosen. It
  follows from $\tau\le\mathrm{o}$ that
  $\tau_{\mathcal B}\le\mathrm{o}$. We will see that, in general,
  $\tau_{\mathcal B}$ need not agree with order convergence.

  To see this, let $X=L_p[0,1]$ with $1\le p<\infty$, and let $\tau$
  be the norm topology of~$X$. In this case, $\mathrm{u}\tau$ is the
  topology of convergence in measure. The typewriter sequence as in, e.g.,
  \cite[Example~10.6]{OBrien:23} is an order
  bounded sequence in $X$ which converges in measure to zero, and
  hence also with respect to~$\tau_{\mathcal B}$, but is not order
  convergent to zero.

  In the preceding example, $\mathrm{u}\tau$ is the topology of
  convergence in measure and, therefore, does not depend
  on~$p$. However, $\tau_{\mathcal B}$ does. Indeed, let
  $f_n=\sqrt{n}\one_{[\frac{1}{n+1},\frac1n]}$. Then $(f_n)$ is order
  bounded in $L_1[0,1]$; since it converges to zero in measure, we
  have $f_n\xrightarrow{\tau_{\mathcal B}}0$ in $L_1[0,1]$. Yet
  $(f_n)$ is not order bounded in $L_2[0,1]$, hence it fails to
  $\tau_{\mathcal B}$-converge to zero in $L_2[0,1]$.
\end{example}



Example~\ref{uob-o} suggests a natural conjecture
regarding the interaction between locally solid convergences and
bornologies. Let $\lambda$ be a locally solid convergence on a vector
lattice~$X$, and $\mathrm{u}\lambda$ be its unbounded
modification. When can one recover $\lambda$ from $\mathrm{u}\lambda$
by taking the bounded modification of $\mathrm{u}\lambda$ by an
appropriate bornology, i.e., when do we have
$(\mathrm{u}\lambda)_{\mathcal B}=\lambda$? Two natural bornologies in
this context appear to be the bornology of all order bounded sets and
of all $\lambda$-bounded sets.

\begin{proposition}\label{bddrecover}
 Let $\lambda$ a locally solid convergence on a vector lattice~$X$.
 \begin{enumerate}
 \item\label{bddrecover-le} If $\mathcal B$ is the bornology of
   $\lambda$-bounded sets in $X$ and $\lambda$ is locally bounded,
   then $(\mathrm{u}\lambda)_\mathcal B\le\lambda$.
 \item\label{bddrecover-ge} If $\mathcal B$ is the bornology of order
   bounded sets in~$X$, then
   $ \lambda\le (\mathrm{u}\lambda)_\mathcal B$.
   \end{enumerate}
\end{proposition}


A locally solid convergence on a vector lattice is said to be
\term{locally order bounded} if it is locally $\mathcal B$-bounded,
where $\mathcal B$ is the bornology of all order bounded sets.

\begin{proposition}\label{lo-lB}
  Let $\lambda$ be a locally solid convergence on a vector lattice~$X$,
  and $\mathcal B$ the bornology of all order bounded sets. The
  following are equivalent:
  \begin{enumerate}
  \item\label{lo-lB-lo} $\lambda$ is locally order bounded;
  \item\label{lo-lB-u} $\lambda=(\mathrm{u}\lambda)_{\mathcal B}$;
  \item\label{lo-lB-lB} $\lambda$ is locally bounded, and every
    $\lambda$-bounded set is order bounded.
  \end{enumerate}
\end{proposition}

\begin{proof}
  \eqref{lo-lB-u}$\Rightarrow$\eqref{lo-lB-lo} and
  \eqref{lo-lB-lB}$\Rightarrow$\eqref{lo-lB-lo} are trivial.
  \eqref{lo-lB-lo}$\Rightarrow$\eqref{lo-lB-u} because $\lambda$ and
  $\mathrm{u}\lambda$ agree on order bounded
  sets. \eqref{lo-lB-lo}$\Rightarrow$\eqref{lo-lB-lB} because every
  $\lambda$-bounded set is order bounded by Corollary~\ref{lB-bdd-B}.
\end{proof}

\begin{corollary}
  Let $X$ be a normed lattice, $\mathrm{n}$ convergence in norm, and
  $\mathcal B$ the bornology of all order bounded sets. Then
  $(\mathrm{un})_{\mathcal B}=\mathrm{n}$ iff $X$ has a strong unit.
\end{corollary}


\begin{proposition}
  Let $X$ be a normed lattice without a strong unit and $\mathcal B$
  the bornology of all norm bounded sets in~$X$. The bounded
  modification $(\mathrm{un})_{\mathcal B}$ of
  $\mathrm{un}$-convergence by $\mathcal B$ fails to be
  topological. In particular,
  $(\mathrm{un})_{\mathcal B}\ne\mathrm{n}$.
\end{proposition}

\begin{proof}
  By \cite[Theorem~2.3]{Kandic:17}, we have
  $\mathrm{un}<\mathrm{n}$. It follows that there exists a net
  $(x_\alpha)$ in $S_X$ such that $x_\alpha\goesun 0$. Using
  Corollary~\ref{top-Bn}, we get that $(\mathrm{un})_{\mathcal B}$
  fails to be topological.
\end{proof}

\begin{example}
  The following example shows that both
  $(\mathrm{u}\lambda)_\mathcal B<\lambda$ and
  $(\mathrm{u}\lambda)_\mathcal B>\lambda$ are possible in
  Proposition~\ref{bddrecover}. Let $\lambda$ be norm convergence on
  $L_1[0,1]$. Then $\mathrm{u}\lambda$ is un-convergence. Let
  $f_n=n\one_{[0,\frac1n]}$. Clearly, $\norm{f_n}=1$ for all~$n$,
  hence $(f_n)$ fails to $\lambda$-converge to 0. However, if
  $\mathcal B$ is the bornology of all $\lambda$-bounded sets (that
  is, of all norm bounded sets) then
  $f_n\xrightarrow{(\mathrm{un})_{\mathcal B}}0$, so
  $(\mathrm{u}\lambda)_{\mathcal B}\ne\lambda$. Since
  $(\mathrm{u}\lambda)_{\mathcal B}\le\lambda$ by
  Proposition~\ref{bddrecover}\eqref{bddrecover-le}, we get
  $(\mathrm{u}\lambda)_{\mathcal B}<\lambda$.

  Now let $g_n=n\one_{[\frac{1}{n},\frac{1}{n-1}]}$, and let $\mathcal
  B$ be the bornology of all order bounded sets. Then $g_n\goesnorm
  0$, but no tail of $(g_n)$ is order bounded, hence $(g_n)$ fails to
  converge to zero in  $(\mathrm{u}\lambda)_{\mathcal B}$. Hence, again,
  $(\mathrm{u}\lambda)_{\mathcal B}\ne\lambda$. Since
  $(\mathrm{u}\lambda)_{\mathcal B}\ge\lambda$
  by Proposition~\ref{bddrecover}\eqref{bddrecover-ge},
  we get
  $(\mathrm{u}\lambda)_{\mathcal B}>\lambda$. 
\end{example}

\begin{proposition}\label{lob-top-su}
  Let $\tau$ be a locally solid locally order bounded linear topology
  on a vector lattice~$X$. Then $X$ has a strong unit and $\tau$ is
  relative uniform convergence.
\end{proposition}

\begin{proof}
  Since $\tau$ is locally order bounded, the neighborhood filter
  of zero,~$\mathcal N_0$, contains an order bounded set, hence an order
  interval $[-e,e]$ for some $e\in X_+$. It follows that $e$ is a
  strong unit. It follows from $[-e,e]\in\mathcal N_0$ that $\tau$ is
  stronger than the topology induced by the norm
  $\norm{\cdot}_e$. On the other hand,
  the latter is the strongest locally solid convergence on~$X$.
\end{proof}

We now look at the locally solid aspect of Examples~\ref{ex:tvs-equi},
\ref{ex:norm-equi}, and~\ref{ex:bw*}.

\begin{example}\label{loc-sol-equi}
  Let $X$ be a locally solid topological vector lattice. By
  Theorem~2.22 in~\cite{Aliprantis:03}, its topological dual $X^*$ is
  an ideal in its order dual~$X^\sim$, hence $X^*$ is a vector
  lattice. We claim that the linear bornology $\mathcal B$ of all
  equicontinuous sets in $X^*$ is locally solid. Indeed, $\mathcal B$
  has a base formed by polars of neighborhoods of zero
  in~$X$. Therefore, for every $B\in\mathcal B$ there exists a
  neighborhood $U$ of zero in $X$ such that $B\subseteq
  U^\circ$. Since $X$ is locally solid, we can find a solid
  neighborhood $V$ of zero in $X$ such that $V\subseteq U$; it follows
  that $V^\circ\supseteq U^\circ$, hence $B\subseteq V^\circ$. Note
  that $V^\circ$ is solid; see, e.g., the paragraph after Theorem~6.1
  in~\cite{Aliprantis:03}.
\end{example}

\begin{example}
  Let $X=\ell_p$ with $1\le p\le \infty$, let $Y$ be $c_{00}$
  considered as a subspace of~$\ell_q$, where $p^{-1}+q^{-1}=1$, and
  let $Z=\ell_q$ (for $p=1$ put $Z:=c_{0}$). We know that
  uo-convergence and coordinate-wise convergence on $X$ agree; see
  Proposition~\ref{discr-uo-p}. Moreover, this convergence is equal to
  $\sigma(X,Y)$, and so is weak* when $X$ is considered a dual
  of~$Y$. Since $Y$ is dense in~$Z$, it follows that $\sigma(X,Z)$ and
  $\sigma(X,Y)$ agree on the norm-bounded subsets of~$X$.  Consider
  the bounded modification of this convergence by the bornology
  $\mathcal B$ of all norm bounded sets. This is exactly the
  $\mathrm{w}^*_{\mathcal B}$ convergence described in
  Example~\ref{ex:norm-equi} (in this case, we may view $X=Y^*$ or
  $X=Z^*$).  This convergence is not topological, but its topological
  modification is $\mathrm{bw}^*$, which is a linear topology (see
  Example \ref{ex:bw*}). When $1<p<\infty$, $\sigma(X,Z)$ is the weak
  topology, hence for sequences $\mathrm{w}^*_{\mathcal B}$ agrees
  with weak convergence. When $p=\infty$, $\mathcal B$ is the
  bornology of order bounded sets, and so
  $\mathrm{bw}^*=\mathrm{t}(\mathrm{w}^*_{\mathcal
    B})=\mathrm{t}(\mathrm{uo}_{\mathcal
    B})=\mathrm{t}(\mathrm{o})$. Applying the same argument to
  $\ell_{\infty}(\Omega)$, for an arbitrary set~$\Omega$, we conclude
  that the topological modification of order convergence on
  $\ell_{\infty}(\Omega)$ is linear.
\end{example}

While the topological modification of order convergence on
$\ell_{\infty}(\Omega)$ is linear, this is not the case for
$L_\infty[0,1]$, as we proved in Example~\ref{ex:L01}. It is also not
the case for $C[0,1]$, according to \cite[Example
10.7]{OBrien:23}. Since both $\ell_{\infty}(\Omega)$ and
$L_\infty[0,1]$ are isomorphic to $C(K)$ spaces, this motivates
the following question.

\begin{question}
  Is it true that if $\mathrm{t}(\mathrm{o})$ is linear on $C(K)$,
  then $K$ has a dense set of isolated points?
\end{question}




\section{Choquet modification}
\label{sec:Choquet}

Recall that a filter $\mathcal U$ is an \term{ultrafilter} if it is
maximal, i.e., not properly contained in another filter. Equivalently,
for every set~$A$, either $A\in\mathcal U$ or
$X\setminus A\in\mathcal U$. Furthermore, a filter $\mathcal U$ is an
ultrafilter iff it is \term{prime}, i.e., if
$\mathcal F\cap\mathcal G\subseteq\mathcal U$ for some filters
$\mathcal F$ and $\mathcal G$ then $\mathcal F\subseteq\mathcal U$ or
$\mathcal G\subseteq\mathcal U$. Every filter is contained in an
ultrafilter.

Let $(X,\eta)$ be a convergence space. We say that it is
\term{Choquet} if $\mathcal F\goeseta x$ whenever
$\mathcal U\goeseta x$ for every ultrafilter
$\mathcal U\supseteq\mathcal F$; see pp.~17--18
in~\cite{Beattie:02}. It is easy to see that every topological
convergence is Choquet. Given an arbitrary convergence structure
$(X,\eta)$, its \term{Choquet modification} $\mathrm{c}\eta$ is
defined as follows: $\mathcal F\goesceta x$ if $\mathcal U\goeseta x$
for every ultrafilter $\mathcal U$ containing~$\mathcal F$. Using the
fact that every ultrafilter is prime, it is easy to see that
$\mathrm{c}\eta$ is a Choquet convergence structure and
$\mathrm{c}\eta\le\eta$.

We will now ``translate'' these two terms into the net language.  We
write $(y_\beta)\preceq(x_\alpha)$ to indicate that $(y_\beta)$ is a
quasi-subnet of~$(x_\alpha)$.

\begin{lemma}\label{Ch-lem}
  For a net $(x_\alpha)$ in a convergence space $X$ and $x\in X$,
  the following are equivalent:
  \begin{enumerate}
  \item\label{Ch-lem-fil} Every ultrafilter containing $[x_\alpha]$
    converges to~$x$;
  \item\label{Ch-lem-net} Every quasi-subnet of $(x_\alpha)$ has a further
    quasi-subnet that converges to~$x$.
  \end{enumerate}
\end{lemma}

\begin{proof}
  \eqref{Ch-lem-fil}$\Rightarrow$\eqref{Ch-lem-net} Let $(y_\beta)$ be
  a quasi-subnet of~$(x_\alpha)$. Note that $[y_\beta]$ is contained
  in some ultrafilter, say,~$\mathcal U$. Find a net $(z_\gamma)$ such
  that $\mathcal U=[z_\gamma]$. Observe that
  $[x_\alpha]\subseteq[y_\beta]\subseteq[z_\gamma]$. It follows that
  $(z_\gamma)\preceq(y_\beta)$ and by~\eqref{Ch-lem-fil},
  $\mathcal U\to x$. Therefore, $z_\gamma\to x$.

  \eqref{Ch-lem-net}$\Rightarrow$\eqref{Ch-lem-fil} Let $\mathcal U$
  be an ultrafilter such that $[x_\alpha]\subseteq\mathcal U$. Find a
  net $(y_\beta)$ with $[y_\beta]=\mathcal U$. Then
  $(y_\beta)\preceq(x_\alpha)$. By \eqref{Ch-lem-net}, there exists a
  net $(z_\gamma)$ such that $(z_\gamma)\preceq(y_\beta)$ and
  $z_\gamma\to x$. It follows that $[z_\gamma]\to x$. Finally, it
  follows from $\mathcal U=[y_\beta]\subseteq[z_\gamma]$ and the maximality of
  $\mathcal U$ that $[z_\gamma]=\mathcal U$, hence $\mathcal U\to x$.
\end{proof}

In \eqref{Ch-lem-net}, one may replace ``quasi-subnets'' with
``subnets'', in either one or in both occurrences. This is because it
was observed in \cite[Proposition~2.2]{OBrien:23} that every subnet is
a quasi-subnet, and every subnet is tail equivalent to a quasi-subnet.

The following two propositions are immediate consequences of
Lemma~\ref{Ch-lem}. 

\begin{proposition}
  Let $X$ be a convergence space. $X$ is Choquet iff for every net
  $(x_\alpha)$ in $X$ and $x\in X$, if every (quasi-)subnet of
  $(x_\alpha)$ has a further (quasi-)subnet that converges to~$x$, then
  $x_\alpha\to x$.
\end{proposition}

\begin{proposition}
  Let $(X,\eta)$ be a convergence space. Then $x_\alpha\goesceta x$
  iff every (quasi-)subnet of $(x_\alpha)$ has a further
  (quasi-)subnet that $\eta$-converges to~$x$.
\end{proposition}

It is easy to see that if $\eta$ is Hausdorff then so is
$\mathrm{c}\eta$. It was observed in
\cite[Proposition~1.3.25(ii)]{Beattie:02} that if a function
$f\colon X\to Y$ between two convergence spaces $(X,\eta)$ and
$(Y,\lambda)$ is $\eta$-to-$\lambda$ continuous then it is also
$\mathrm{c}\eta$-to-$\mathrm{c}\lambda$ continuous. If $\eta$ is
linear then so is~$\mathrm{c}\eta$; see
\cite[Proposition~3.1.4]{Beattie:02}.

\medskip

Let $\mathcal B$ be a bornology on a set~$X$; suppose that
$\mathcal B$ is proper in the sense that
$\mathcal B\ne \mathcal P(X)$. Then $\mathcal B$ is a proper ideal in
$\mathcal P(X)$. Let $\mathcal F$ be a filter on~$X$. If
$\mathcal F\cap\mathcal B=\varnothing$ then there exists an
ultrafilter $\mathcal U$ containing $\mathcal F$ such that
$\mathcal U\cap\mathcal B=\varnothing$, see, e.g., Lemma~2.1
in~\cite{Aviles:comp-int}. That is, $\mathcal F$ contains a bounded set
iff every ultrafilter containing $\mathcal F$ contains a bounded set. 
It follows that Choquet modification commutes with bornological
modification:

\begin{proposition}\label{Choq-born}
  Let $(X,\eta)$ be a convergence space, and $\mathcal B$ a bornology
  on~$X$. Then
  $(\mathrm{c}\eta)_{\mathcal B}=\mathrm{c}(\eta_{\mathcal B})$. In
  particular, if $\eta$ is Choquet then so is~$\eta_{\mathcal B}$.
\end{proposition}

\begin{proof}
  For a filter~$\mathcal F$,
  $\mathcal F\xrightarrow{(\mathrm{c}\eta)_{\mathcal B}}x$ iff
  $\mathcal F\xrightarrow{\mathrm{c}\eta}x$ and $\mathcal F$
  meets~$\mathcal B$. By the remark preceding the proposition, this is
  equivalent to saying that for every ultrafilter $\mathcal U$
  containing $\mathcal F$ we have $\mathcal U\goeseta x$ and
  $\mathcal U$ meets~$\mathcal B$, hence
  $\mathcal U\xrightarrow{\eta_{\mathcal B}}x$. This means
  $\mathcal F\xrightarrow{\mathrm{c}(\eta_{\mathcal B})}x$.
\end{proof}

\begin{proposition}
  Let $\mathcal B$ be a bornology on a set~$X$, and $(x_\alpha)$ a net
  in~$X$. Every quasi-subnet of $(x_\alpha)$ in turn has a bounded
  quasi-subnet iff $(x_\alpha)$ is eventually bounded.
\end{proposition}

\begin{proof}
  Suppose $(x_\alpha)$ is not eventually bounded. Then
  $[x_\alpha]\cap\mathcal B=\varnothing$. As before, find an
  ultrafilter $\mathcal U$ containing $[x_\alpha]$ such that $\mathcal
  U\cap\mathcal B=\varnothing$. Then $\mathcal U=[y_\beta]$ for some
  quasi-subnet $(y_\beta)$ of $(x_\alpha)$. By assumption, there
  exists a bounded net $(z_\gamma)$ such that
  $(z_\gamma)\preceq(y_\beta)$. Since $[y_\beta]\subseteq[z_\gamma]$,
  maximality of $\mathcal U$ yields $[y_\beta]=[z_\gamma]$ and,
  therefore, $[z_\gamma]\cap\mathcal B=\varnothing$, which is a
  contradiction.

  The converse implication is straightforward.
\end{proof}

\begin{remark}\label{lb-Choq}
  Let $\eta$ be a linear convergence and $\mathcal B$ a linear
  bornology on a vector space~$X$. If $\eta$ is locally
  $\mathcal B$-bounded, i.e., $\eta=\eta_{\mathcal B}$, then it
  follows from Proposition~\ref{Choq-born} that
  $(\mathrm{c}\eta)_{\mathcal B}=\mathrm{c}(\eta_{\mathcal
    B})=\mathrm{c}\eta$, hence $\mathrm{c}\eta$ is also locally
  $\mathcal B$-bounded.

  Suppose now that $\eta$ is locally bounded, that is, it is locally
  $\mathcal B$-bounded where $\mathcal B$ is the bornology of all
  $\eta$-bounded sets. By the preceding paragraph, $\mathrm{c}\eta$ is
  also locally $\mathcal B$-bounded. By Corollary~\ref{lB-bdd-B},
  every $\mathrm{c}\eta$-bounded set is $\eta$-bounded. On the other
  hand, since $\mathrm{c}\eta\le\eta$, $\eta$-bounded sets are
  $\mathrm{c}\eta$-bounded. Thus, if $\eta$ is locally bounded then
  $\eta$ and $\mathrm{c}\eta$ have the same bounded sets.
\end{remark}

\begin{proposition}\label{Ch-ls}
  Let $\eta$ be a convergence structure on a vector lattice~$X$. If
  $\eta$ is locally solid then so is~$\mathrm{c}\eta$. Moreover,
  $\eta$ and $\mathrm{c}\eta$ agree on monotone nets.
\end{proposition}

\begin{proof}
  We will use Proposition~\ref{mod-u-cont}\eqref{uc-full}. Since the
  modulus operation is $\eta$-to-$\eta$ continuous, it is also
  $\mathrm{c}\eta$-to-$\mathrm{c}\eta$ continuous. So it suffices to
  verify that $\mathrm{c}\eta$ is locally full. Suppose that $0\le
  y_\alpha\le x_\alpha\goesceta 0$. Let $(y_{\alpha_\beta})$ be a
  subnet of~$(y_\alpha)$. Consider the corresponding subnet
  $(x_{\alpha_\beta})$ of~$(x_\alpha)$. Since $x_\alpha\goesceta 0$,
  there exists a further subnet
  $(x_{\alpha_{\beta_\gamma}})$ of $(x_{\alpha_\beta})$
  such that $x_{\alpha_{\beta_\gamma}}\goeseta 0$. Since
  $\eta$ is locally solid, we have
  $y_{\alpha_{\beta_\gamma}}\goeseta 0$ and, therefore,
  $y_\alpha\goesceta 0$.

  For the ``moreover'' clause, we always have
  $\mathrm{c}\eta\le\eta$. Suppose that $x_\alpha\downarrow$ and
  $x_\alpha\goesceta x$. Then there exists a quasi-subnet $(y_\beta)$
  of $(x_\alpha)$ such that $y_\beta\goeseta x$. Then
  $x_\alpha\goeseta x$ by Corollary~\ref{qsnet-mon-conv}.
\end{proof}

It was shown in Theorem 1.5.5. in \cite{Beattie:02} that continuous
convergence on $C(\Omega)$ is Choquet. By Example~1.5.26(i) and
Theorem~4.3.9 in \cite{Beattie:02}, if $\Omega$ is not locally compact then
this provides an example of a
locally solid convergence which is Choquet but not topological.

Recall that if $f\colon X\to Y$ is a function and $\mathcal F$ is a
filter on $X$ then we write $f(\mathcal F)$ for the filter generated
by $\bigl\{f(A)\mid A\in\mathcal F\bigr\}$. It is easy to see that if
$\mathcal F$ is an ultrafilter then so is $f(\mathcal F)$. If $X$ is a
vector lattice and $u\in X_+$ then $\abs{\mathcal F}\wedge u$ equals
$f(\mathcal F)$ where $f(x)=\abs{x}\wedge u$. In particular, if
$\mathcal F$ is an ultrafilter then so is $\abs{\mathcal F}\wedge u$.

The following proposition shows that the unbounded modification
commutes with the Choquet modification.

\begin{proposition}
  Let $(X,\eta)$ be a locally solid convergence space. Then
  \begin{math}
    \mathrm{u}(\mathrm{c}\eta)=
        \mathrm{c}(\mathrm{u}\eta).
  \end{math}
\end{proposition}

\begin{proof}
  Suppose that
  \begin{math}
    0\le x_\alpha\xrightarrow{\mathrm{c}(\mathrm{u}\eta)}0.
  \end{math}
  Fix $u\ge 0$. Take any quasi-subnet $(y_\beta)$ of $(u\wedge
  x_\alpha)$.  It follows from 
   \begin{math}
    u\wedge x_\alpha\xrightarrow{\mathrm{c}(\mathrm{u}\eta)}0
  \end{math}
  that there exists a quasi-subnet $(z_\gamma)$ of $(y_\beta)$ such
  that $z_\gamma\goesueta 0$. Since the net $(z_\gamma)$ is eventually
  order bounded, we have $z_\gamma\goeseta 0$. It follows that
  $u\wedge x_\alpha\goesceta 0$ and, therefore,
  \begin{math}
    0\le x_\alpha\xrightarrow{\mathrm{u}(\mathrm{c}\eta)}0.
  \end{math}

  For the converse, we will use filter language. Suppose that
  \begin{math}
    \mathcal F\xrightarrow{\mathrm{u}(\mathrm{c}\eta)}0.
  \end{math}
  Let $\mathcal U$ be an ultrafilter with
  $\mathcal F\subseteq\mathcal U$. Fix $u\ge 0$. Then
  $\abs{\mathcal F}\wedge u\subseteq\abs{\mathcal U}\wedge u$. By the
  comment preceding the proposition, $\abs{\mathcal U}\wedge u$ is an
  ultrafilter. It follows from 
  \begin{math}
    \mathcal F\xrightarrow{\mathrm{u}(\mathrm{c}\eta)}0
  \end{math}
  that $\abs{\mathcal F}\wedge u\goesceta 0$, hence
  $\abs{\mathcal U}\wedge u\goeseta 0$ and, therefore, $\mathcal
  U\goesueta 0$. We now conclude that 
  \begin{math}
    \mathcal F\xrightarrow{\mathrm{c}(\mathrm{u}\eta)}0.
  \end{math}
\end{proof}

\begin{remark}
  Using the same proof, one can prove that
  \begin{math}
    \mathrm{u}_I(\mathrm{c}\eta)=
        \mathrm{c}(\mathrm{u}_I\eta)
  \end{math}
  for every ideal~$I$.
\end{remark}

\begin{proposition}\label{ceta-top}
  Let $\eta$ be a locally solid locally order bounded convergence on
  $X$ such that $\mathrm{c}\eta$ is topological. Then $X$ has a strong
  unit and $\eta$ is relative uniform convergence.
\end{proposition}

\begin{proof}
  By Remark~\ref{lb-Choq} and Proposition~\ref{Choq-born},
  $\mathrm{c}\eta$ is locally solid and locally order
  bounded. By Proposition~\ref{lob-top-su}, $X$ has a strong unit and
  $\mathrm{c}\eta=\mathrm{ru}$. It now follows from
  \begin{math}
    \mathrm{ru}=\mathrm{c}\eta\le\eta\le\mathrm{ru}
  \end{math}
  that $\eta=\mathrm{ru}$.
\end{proof}

\begin{example}
  Since every norm convergent sequence in a Banach
  lattice has a subsequence that converges relatively uniformly, one
  may naturally conjecture that the same may be true for nets. This
  is, generally, not true. Indeed, this would imply that
  $\mathrm{c}(\mathrm{ru})$ is norm convergence, hence $X$ has a
  strong unit by Proposition~\ref{ceta-top}. Below, we provide a
  specific example of a norm null net in a Banach lattice with no
  order bounded subnets.

  Fix an uncountable set~$\Omega$. Put $X=\ell_1(\Omega)$. Let
  $\Gamma$ be the set of all finite non-empty subsets of~$\Omega$,
  ordered by inclusion. For $\gamma\in\Gamma$, put
  $x_\gamma=\frac{1}{\abs{\gamma}^2}\one_\gamma$, where $\abs{\gamma}$
  is the cardinality of~$\gamma$. Clearly,
  $\norm{x_\gamma}=\frac{1}{\abs{\gamma}}\to 0$.

  Suppose that there is a subnet $(x_{\gamma_\alpha})_{\alpha\in A}$
  of $(x_\gamma)_{\gamma\in\Gamma}$ which is order bounded by some
  $u\in \ell_1(\Omega)_+$. Since the set
  $\{\gamma_\alpha\mid\alpha\in A\}$ is co-final in~$\Gamma$, it
  follows that for every $t\in\Omega$, $\{t\}\subseteq\gamma_\alpha$
  for some~$\alpha$; we conclude that $u(t)\ge
  x_{\gamma_\alpha}(t)>0$. This means that $u$ has uncountable
  support, which is impossible in $\ell_1(\Omega)$.
\end{example}

In the rest of this section, we will discuss when order and uo
convergences are Choquet.

\begin{lemma}\label{CK-Choq}
  Let $K$ be a compact topological space with no isolated points, and
  $F$ a norm dense sublattice of $C(K)$ such that $\one\in F$.  There
  exists a net in $F\cap[0,\one]$ which, viewed as a net in $C(K)$, is
  $\mathrm{co}$-null but not order null.
\end{lemma}

\begin{proof}
  We start by construction an index set. For each $s\in K$, let
  $\mathcal N_s$ be the set of all neighborhoods of~$s$. Put
  \begin{math}
    A=K\times\prod_{s\in K}\mathcal N_s.
  \end{math}
  That is, for every $\alpha\in A$, the first component is some point
  of~$K$, which we will denote by $t_\alpha$, and the second component
  is a family $(U^{\alpha}_s)_{s\in K}$, where
  $U^{\alpha}_s\in\mathcal N_s$ for every~$s$. We pre-order $A$ as
  follows $\alpha_1\le\alpha_2$ if
  $U^{\alpha_2}_s\subseteq U^{\alpha_1}_s$ for every $s\in K$. For
  every $\alpha\in A$, use, e.g., Proposition~4.1 in~\cite{Bilokopytov:23b}
  to find a function $f_\alpha\in F$ such
  that $0\le f_\alpha\le\one$, $f_\alpha(t_\alpha)=1$, and $f_\alpha$
  vanishes outside of $U^{\alpha}_{t_\alpha}$. It is easy to see that
  for every $t\in K$, in every tail of the net there exists an index
  $\alpha$ such that $f_\alpha(t)=1$; it follows that the supremum of
  every tail is $\one$ and, therefore, $f_\alpha$ fails to converge to
  zero in order.

  We will show that $f_\alpha\xrightarrow{\mathrm{co}}0$. Let
  $(f_{\alpha_\beta})$ be a subnet of~$(f_\alpha)$. By compactness
  of~$K$, there is a further subnet $(f_{\alpha_{\beta_\gamma}})$ such
  that $(t_{\alpha_{\beta_\gamma}})$ converges to some~$t_0$. We will
  show that $f_{\alpha_{\beta_\gamma}}\goeso 0$; this will complete
  the proof.  Fix an open non-empty subset $U$
  of~$K$. By Theorem~3.2 in~\cite{Bilokopytov:22}, it suffices to show
  that some tail of $(f_{\alpha_{\beta_\gamma}})$ vanishes on some
  open non-empty subset $V$ of~$U$. Since $K$ has no isolated points,
  we can find $W\in\mathcal N_{t_0}$ and an open non-empty subset $V$
  of $U$ such that $W\cap V=\varnothing$. Define $\alpha_0$ as
  follows: $t_{\alpha_0}=t_0$, and let $U^{\alpha_0}_s$ be $W$ if
  $s\in W$ and $K$ otherwise. Passing to a tail, we may assume that
  $\alpha_{\beta_\gamma}\ge\alpha_0$ and
  $t_{\alpha_{\beta_\gamma}}\in W$ for all~$\gamma$. Then
  $U^{\alpha_{\beta_\gamma}}_{t_{\alpha_{\beta_\gamma}}}\subseteq
  W$. Since $f_{\alpha_{\beta_\gamma}}$ vanishes outside of
  $U^{\alpha_{\beta_\gamma}}_{t_{\alpha_{\beta_\gamma}}}$, it vanishes
  on~$V$.
\end{proof}

\begin{theorem}\label{o-Choq}
 For an Archimedean vector lattice $X$ the following are equivalent:
 \begin{enumerate}
 \item\label{o-Choq-disc} $X$ is discrete;
  \item\label{o-Choq-o} Order convergence on $X$ is Choquet;
  \item\label{o-Choq-uo} Unbounded order convergence on $X$ is Choquet;
  \item\label{o-Choq-top} Unbounded order convergence on $X$ is topological.
 \end{enumerate}
\end{theorem}

\begin{proof}
 \eqref{o-Choq-disc}$\Rightarrow$\eqref{o-Choq-top} If $X$ is discrete
 then uo convergence is topological by Proposition~\ref{discr-uo-p}.

 \eqref{o-Choq-top}$\Rightarrow$\eqref{o-Choq-uo} is trivial.
 
 \eqref{o-Choq-uo}$\Rightarrow$\eqref{o-Choq-o} follows from
 Example~\ref{uob-o} and Proposition~\ref{Choq-born}.

 \eqref{o-Choq-o}$\Rightarrow$\eqref{o-Choq-disc}
 Assume that $X$ is not discrete. Then there exists $0<e\in X$ such
 that $I_e$ is atomless. By Krein-Kakutani's Representation Theorem,
 there exists a lattice isomorphic embedding $T\colon I_e\to C(K)$ for
 some compact Hausdorff space~$K$, such that $F:=\Range T$ is a norm
 dense sublattice and $\one\in F$. Since $F$ is atomless, $K$ has no
 isolated points; see, e.g., Lemma~4.2 in~\cite{Aviles:comp-int}. Let
 $(f_\alpha)$ be a net as in Lemma~\ref{CK-Choq}.

 It is easy to see that $F$ is order dense and, therefore, regular in
 $C(K)$. It follows that order convergences in $F$ and in $C(K)$ agree
 for order bounded nets; see, e.g., Corollary~2.12
 in~\cite{Gao:17}. Therefore, $(f_\alpha)$ is $\mathrm{co}$-null but
 not order null not only in $C(K)$ but also in~$F$. Put
 $x_\alpha=T^{-1}f_\alpha$. Since $T$ is a lattice isomorphism,
 $(x_\alpha)$ is $\mathrm{co}$-null but not order null in~$I_e$; it is
 also order bounded there. Since $I_e$ is regular, we conclude that
 $(x_\alpha)$ is $\mathrm{co}$-null but not order null in~$X$. It
 follows that order convergence on $X$ fails to be Choquet.
\end{proof}

\begin{remark}
  The fact that uo-convergence is topological iff the
  vector lattice is discrete is known: it was shown in
  Theorem~7.5 in~\cite{Taylor:19} that uo-convergence agrees with
  the convergence of a locally convex-solid topology iff $X$ is
  discrete; in the general case, this fact may be deduced by combining
  results of \cite{Ellis:68,Papangelou:65,Winberg:62}. A direct proof
  may be found in~\cite{Taylor:TH} and in~\cite{Aviles:comp-int}.
\end{remark}

\begin{question}
  Characterize Archimedean vector lattices on which relative uniform
  convergence or $\sigma$-order convergence is Choquet.
\end{question}

Some partial answers to this question are already contained in this
article. First, if $X$ has a strong unit, then $\mathrm{ru}$ is given
via a norm, and so in particular it is Choquet. In an order continuous
Banach lattice we have $\mathrm{o}=\mathrm{ru}$ (see,
e.g.,~\cite{Bedingfield:80}), hence in
this case the latter is Choquet if and only if $X$ is
discrete. In Example~\ref{cru-bm}, we observed that relative uniform convergence on the appropriate spaces
of continuous functions may be obtained as a bounded modification of
ucc convergence; since ucc is topological and, therefore, Choquet, we
conclude by Proposition~\ref{Choq-born} that $\mathrm{ru}$ on these
spaces is Choquet as well. Hence, neither
discreteness nor having strong unit are necessary conditions. Is
discreteness sufficient?

If $X$ has countable supremum property, then
$\mathrm{o}=\sigma\mathrm{o}$, and so the latter is Choquet if and
only if $X$ is discrete. Discreteness is not necessary due to
existence of the following class of examples. Recall that a compact
Hausdorff space $K$ is called a $P^\prime$ space or an almost $P$
space (see~\cite{Veksler:73}) if every non-empty closed $G_\delta$ set
has a nonempty interior, or equivalently, $\sigma\mathrm{o}$ on $C(K)$
is given via the supremum norm. Clearly, in this case
$\sigma\mathrm{o}$ is Choquet, but $K$ does not necessarily have a
dense set of isolated points: for a counterexample, take
$K=\beta\mathbb N\setminus\mathbb N$, which has no isolated points.

\section{Weakest and minimal locally solid convergence structures}
\label{sec:minimal}

Let $X$ be a set. All convergence structures on $X$ are naturally
ordered by the ``weaker/stronger'' relation. So it makes sense to ask
whether a family $\mathcal A$ of convergence structures on $X$ has a
least element (the weakest convergence in the family), a greatest
element (the strongest convergence), whether it has minimal elements
(an element of $\mathcal A$ is minimal if no other element of
$\mathcal A$ is weaker than it) or maximal elements (defined
analogously).  It may also be of interest to know whether
$\inf\mathcal A$ and $\sup\mathcal A$ exist. Naturally, when
$\mathcal A$ consists of linear (or locally solid) Hausdorff
convergences, we may want to find $\inf\mathcal A$ and
$\sup\mathcal A$ with respect to the same class.  

We start with the set of all convergence structures on~$X$. The
strongest convergence on $X$ is discrete (only eventually constant
nets are convergent), while the weakest convergence is trivial (every
net converges to every point). Every collection of convergence
structures $\mathcal A$ has a supremum given as follows: a net (or a
filter) converges if it converges with respect to every convergence in
the collection. It follows that if the elements of $\mathcal A$ are
Hausdorff, linear, or locally solid then so is $\sup\mathcal A$.  In
\cite[Definition~1.2.1]{Beattie:02}, this convergence structure is called the
\term{initial convergence structure with respect to the family
  $\mathcal A$}.

Since every set of convergences has a supremum, the set of all
convergence structures on $X$ is a complete lattice. This implies that
$\inf\mathcal A$ always exists. In the special case when $\mathcal A$
is directed downwards, $\inf\mathcal A$ can be easily described as
follows: a net converges with respect to $\inf\mathcal A$ if it
converges with respect to \emph{some} convergence in~$\mathcal A$. It
follows that in this case if elements of $\mathcal A$ are Hausdorff,
linear, or locally solid then so is $\inf\mathcal A$.

Describing $\inf\mathcal A$ for a general $\mathcal A$ requires more
effort. Furthermore, $\inf\mathcal A$ need not preserve whatever nice
properties that elements of $\mathcal A$ may
have. Using~\cite[p.~8]{Beattie:02}, one may characterize $\inf\mathcal A$ in
the language of filters as follows:
$\mathcal F\to x$ whenever
$\mathcal F\supseteq\mathcal F_1\cap\dots\cap\mathcal F_n$, where each
$\mathcal F_i$ converges to $x$ with respect to some convergence in
$\mathcal A$ (dependent on $i$).

Next, we present a net version of this characterization. We use the
concept of a mixing of a family of nets, introduced
in~\cite{OBrien:23}. We recall from~\cite{OBrien:23} that if a net $M$
is the mixing of a family of nets $\Sigma$ then the tail filter of $M$
equals the intersection of the tail filters of the nets
in~$\Sigma$. Furthermore, in a convergence space, if $\Sigma$ is
finite and each net in $\Sigma$ converges to $x$ then so does~$M$.

\begin{proposition}\label{inf}
  Let $X$ be a set, $\mathcal A$ a collection of convergence
  structures on~$X$, and $\lambda=\inf\mathcal A$. Then $x_\alpha\goesl x$ iff
  $(x_\alpha)$ is a quasi-subnet of the mixing of a finite collection
  $S_1,\dots,S_n$ of nets, where each $S_i$ converges to $x$ with
  respect to some convergence in~$\mathcal A$.
\end{proposition}

\begin{proof}
  Let $(x_\alpha)$ be a quasi-subnet of~$M$, where $M$ is the mixing of a
  finite collection $S_1,\dots,S_n$ of nets, where each $S_i$
  converges to $x$ with respect to some convergence in~$\mathcal A$. Then
  clearly $S_i\goesl x$ for each~$i$; it follows that $M\goesl x$ and,
  therefore, $x_\alpha\goesl x$.

  Conversely, suppose that $x_\alpha\goesl x$. Then
  $[x_\alpha]\goesl x$. It follows that
  $[x_\alpha]\supseteq \mathcal F_1\cap\dots\cap\mathcal F_n$, where
  each $\mathcal F_i$ converges with respect to some convergence
  in~$\mathcal A$. For each $i=1,\dots,n$, find a net $S_i$ such that
  $\mathcal F=[S_i]$, and let $M$ be the mixing of $S_1,\dots,
  S_n$. Then each $S_i$ converges to $x$ with respect to some
  convergence structure in~$\mathcal A$. We also have
  \begin{displaymath}
    [x_\alpha]\supseteq \mathcal F_1\cap\dots\cap\mathcal F_n
    =[S_1]\cap\dots\cap[S_n]=[M],
  \end{displaymath}
  so that $(x_\alpha)$ is a quasi-subnet of~$M$.
\end{proof}

Suppose now that $X$ is a vector space. The strongest linear
convergence on $X$ is finite-dimensional convergence, while the weakest is
the trivial one. Let $\mathcal A$ be a set of linear convergence
structures on~$X$. The infimum of $\mathcal A$ in the set of all
convergence structures described above need not be linear. However,
since $\sup\mathcal A$ is linear for every such~$\mathcal A$, the
collection of all linear convergence structures is itself a complete
lattice. Therefore, every collection $\mathcal A$ of linear
convergence structures has an infimum in the set of all linear
convergence structures. In the filter language, this infimum was
described in~\cite[Corollary~3.3.7]{Beattie:02} as follows: a filter
converges to $0$ if it contains $\mathcal F_1+\dots+\mathcal F_n$,
where each $\mathcal F_i$ converges to $0$ with respect to some
convergence in~$\mathcal A$. Analogously to
Proposition~\ref{inf}, we can translate this to the language of nets
as follows:

\begin{proposition}
  A net converges to zero in the infimum of $\mathcal A$ within the
  lattice of linear convergences iff it is a quasi-subnet of
  $S_1+\dots+S_n$, where each $S_i$ converges to $0$ with respect to
  some convergence in~$\mathcal A$.
\end{proposition}

The situation with local solidity is similar. Let $X$ be a vector
lattice. The strongest locally solid convergence is relative
unifrom convergence, while the weakest is the trivial one. The
supremum of a collection of locally solid convergences is locally
solid, and so the set of all locally solid convergences is a complete
lattice. It follows that every set $\mathcal A$ of locally solid
convergences has an infimum in this lattice. In fact, it coincides
with the infimum of $\mathcal A$ in the lattice of linear
convergences. Indeed, denote the latter infimum by~$\to$. If
$\mathcal F\to 0$ then it contains $\mathcal F_1+\dots+\mathcal F_n$,
where each $\mathcal F_i$ and, therefore, also $\Sol(\mathcal F_i)$,
converges to zero with respect to some convergence in~$\mathcal
A$. Hence, $\mathcal F$ contains
$\Sol(\mathcal F_1)+\dots+\Sol(\mathcal F_n)\to 0$. Therefore, $\to$
is locally solid.

The following example implies that the infimum of all Hausdorff
locally solid convergences on $X$ need not be
Hausdorff. That is, there may be no weakest Hausdorff locally solid
convergence on~$X$.

\begin{example}\label{no-least}
  Let $X=C[0,1]$. Consider two convergence structures on $X$: let
  $\mu$ be convergence in measure and let $\eta$ be pointwise
  convergence on $\mathbb Q\cap[0,1]$. It can be easily verified that
  both are locally solid Hausdorff convergences. Let $(q_n)$ be an
  enumeration of $\mathbb Q\cap[0,1]$.  One can easily construct a
  sequence $(x_n)$ in $X$ such that $x_n$ takes the value $1$ at the
  points $q_1,\dots,q_n$, but the measure of the support of $x_n$
  tends to zero. It follows that $x_n\xrightarrow{\mu}0$ but
  $x_n\goeseta\one$. Hence, for every convergence $\lambda$ that is weaker
  than both $\mu$ and $\eta$ we have both $x_n\xrightarrow{\lambda}0$
  and $x_n\xrightarrow{\lambda}\one$, so that any such $\lambda$ fails
  to be Hausdorff.
\end{example}

\begin{proposition}
  Let $X$ be a discrete vector lattice.  Then coordinate-wise
  convergence is the weakest locally solid Hausdorff convergence on~$X$.
\end{proposition}

\begin{proof}
  Let $\eta$ be a locally solid Hausdorff convergence on~$X$. If
  $0\le x_\alpha\goeseta 0$, it follows from $P_ax_\alpha\le x_\alpha$
  that $P_ax_\alpha\to 0$ for every atom $a$ in $X$ and, therefore,
  $x_\alpha\goesp 0$.
\end{proof}

Note that for a discrete vector lattice, the weakest Hausdorff locally
solid convergence is also the weakest Hausdorff and locally solid topology.

It is known that on an order continuous discrete Banach lattice, uo
and un convergences agree with coordinate-wise convergence; see, e.g.,
Lemma~3.1 and Theorem~3.4 in~\cite{Dabboorasad:20}.  We now generalize
this as follows.

\begin{corollary}\label{u-ocont-p}
  Let $X$ be a discrete vector lattice and $\lambda$ a Hausdorff order
  continuous locally solid convergence on~$X$. Then
  $\mathrm{u}\lambda=\mathrm{p}$.
\end{corollary}

\begin{proof}
  Since $\lambda$ is order continuous, we have
  $\lambda\le\mathrm{o}$. It follows that
  $\mathrm{u}\lambda\le\mathrm{uo}$. By Proposition~\ref{discr-uo-p},
  we have $\mathrm{uo}=\mathrm{p}$, hence
  $\mathrm{u}\lambda\le\mathrm{p}$. Minimality of $p$ yields
  $\mathrm{u}\lambda=\mathrm{p}$.
\end{proof}

\begin{remark}\label{ocn-ru-o}
  Recall that on an order continuous Banach lattice, $\mathrm{ru}$ convergence
  agrees with order convergence (see, e.g.,~\cite{Bedingfield:80}).
  Let $X$ be an order continuous discrete Banach lattice and $\eta$ a
  Hausdorff locally solid convergence on~$X$. By
  Proposition~\ref{u-strongest}, we have
  $\eta\le\mathrm{ru}=\mathrm{o}$, so Corollary~\ref{u-ocont-p} yields
  $\mathrm{u}\eta=\mathrm{p}$. In particular, $\eta$ agrees with
  $\mathrm{p}$ on order bounded nets. Furthermore, the bounded
  modification of $\eta$ by the bornology of all order
  bounded sets is order convergence.
\end{remark}

\subsection*{Minimal convergences.}
We observed earlier that for every downward directed set of Hausdorff
linear (or locally solid) convergence structures, its infimum in the
set of all convergence structures always exists and is again Hausdorff
and linear (respectively, locally solid). Using Zorn's lemma, we now
obtain the following:

\begin{proposition}
  Let $X$ be a vector space (or a vector lattice). The set of all
  Hausdorff linear (respectively, locally solid) convergences on $X$
  has a minimal element.
\end{proposition}

\begin{remark}\label{min-dom}
  A similar argument shows that given a Hausdorff linear (or locally
  solid) convergence structure~$\lambda$, there is a minimal Hausdorff
  linear (respectively, locally solid) convergence structure that is
  dominated by~$\lambda$.
\end{remark}

By a \term{minimal Hausdorff locally solid convergence} on a vector
lattice $X$ we mean a minimal element in the set of all Hausdorff
locally solid convergences on~$X$.

\begin{proposition}\label{min-unb-Ch}
  Every minimal Hausdorff locally solid convergence is unbounded and
  Choquet.
\end{proposition}

\begin{proof}
  Let $\eta$ be a minimal locally solid Hausdorff convergence. Then
  $\mathrm{u}\eta$ and $\mathrm{c}\eta$ are Hausdorff locally
  solid convergences
  with $\mathrm{u}\eta\le\eta$ and $\mathrm{c}\eta\le\eta$.
  It follows that $\mathrm{u}\eta=\eta=\mathrm{c}\eta$.
\end{proof}

We are going to prove that every minimal Hausdorff locally solid
convergence is order continuous. We will use the following auxiliary
concept. Given two locally solid convergences $\eta$ and $\lambda$ on
a vector lattice~$X$, we introduce their \term{composition} $\eta\lambda$
as follows: we say that $x_\alpha\xrightarrow{\eta\lambda}x$ if there
exists a net $(v_\gamma)$ in $X_+$ such that $v_\gamma\goesl 0$ and
for every $\gamma$ the net
\begin{math}
  \bigl(\abs{x_\alpha-x}-v_\gamma\bigr)^+\goeseta 0
\end{math}
as a net in~$\alpha$.
We say that $(v_\gamma)$ is a \term{control net} for the convergence.

\begin{proposition}
  If $\eta$ ans $\lambda$ are two locally solid convergence structures
  then so is~$\eta\lambda$. 
\end{proposition}

\begin{proof}
  It suffices to verify that the assumptions of Theorem~\ref{plsc}
  are satisfied for nets in~$X_+$. \eqref{plsc-sub} is obvious because
  if $(y_\beta)$ is a quasi-subnet of $(x_\alpha)$ then for every
  $\gamma$ the net
  $(y_\beta-v_\gamma)^+$ is a quasi-subnet of
  $(x_\alpha-v_\gamma)^+$. \eqref{plsc-dom} follows from the
  observation that if $0\le y_\alpha\le x_\alpha$ then
  $0\le(y_\alpha-v_\gamma)^+\le(x_\alpha-v_\gamma)^+$.

  To verify \eqref{plsc-sum}, let $x_\alpha\xrightarrow{\eta\lambda}0$
  and $y_\alpha\xrightarrow{\eta\lambda}0$ in~$X_+$. Let $(u_\beta)$
  and $(v_\gamma)$ be the corresponding control nets. Then
  $u_\beta+v_\gamma\goesl 0$ and for every $\beta$ and $\gamma$ we have
  \begin{displaymath}
    \bigl((x_\alpha+y_\alpha)-(u_\beta+v_\gamma)\bigr)^+
    \le(x_\alpha-u_\beta)^++(y_\alpha-v_\gamma)^+\goeseta 0,
  \end{displaymath}
  hence $x_\alpha+y_\alpha\xrightarrow{\eta\lambda}0$ with control net
  $(u_\beta+v_\gamma)$.

  To see \eqref{plsc-1n}, just take the control net to be the zero net.
\end{proof}

\begin{proposition}
  Let $\eta$ and $\lambda$ be two locally solid convergence
  structures. Then $\eta\lambda\le\eta$. Furthermore, $\lambda$ 
  dominates $\eta\lambda$ on monotone nets. If $\lambda$ is order
  continuous then so is~$\eta\lambda$. 
\end{proposition}

\begin{proof}
  If $x_\alpha\goeseta x$ then $x_\alpha\xrightarrow{\eta\lambda}x$;
  just take $(v_\gamma)$ to be a constant zero net.
  
  Suppose that $(x_\alpha)$ is monotone and $x_\alpha\goesl x$; we
  need to show that $x_\alpha\xrightarrow{\eta\lambda}x$. Without loss
  of generality, $x_\alpha\downarrow$ and $x=0$. We take $(x_\alpha)$
  itself as a control net. For every fixed~$\beta$, the net
  $(x_\alpha-x_\beta)^+$ is eventually zero, hence is $\eta$-null, so
  that $x_\alpha\xrightarrow{\eta\lambda}0$.

  The last claim now follows from Theorem~\ref{ocn-mon}.
\end{proof}

\begin{proposition}
  The composition $\eta\lambda$ of two locally solid convergences is
  Hausdorff iff both $\eta$ and $\lambda$ are Hausdorff.
\end{proposition}

\begin{proof}
  Suppose that $\eta$ and $\lambda$ are both Hausdorff; we will show
  that so is~$\eta\lambda$. We will use
  Proposition~\ref{Haus}\eqref{Haus-zero}. Let $(x_\alpha)$ be a
  constant zero net such that $x_\alpha\xrightarrow{\eta\lambda}x$ for
  some $x\in X$. Let
  $(v_\gamma)$ be a control net. For every~$\gamma$, the net
  \begin{math}
  \bigl(\bigl(\abs{x_\alpha-x}-v_\gamma\bigr)^+\bigr)
  \end{math}
  is $\eta$-null. However, every term in this net is the constant
  $\bigl(\abs{x}-v_\gamma\bigr)^+$. Since $\eta$ is Hausdorff, we
  conclude that
   $\bigl(\abs{x}-v_\gamma\bigr)^+=0$ and, therefore, $\abs{x}\le
   v_\gamma$. It now follows from $v_\gamma\goesl 0$ and local
   solidity of $\lambda$ that $\abs{x}\goesl 0$. Since $\lambda$ is
   Hausdorff, we conclude that $x=0$.
  
  Suppose that $\eta\lambda$ is Hausdorff, we will now show that $\eta$ and
  $\lambda$ are both Hausdorff. It follows from  $\eta\lambda\le\eta$
  that $\eta$ is Hausdorff. Suppose now that $(x_\alpha)$ is a zero
  net and $x_\alpha\goesl x$. Since this net is monotone, we get
  $x_\alpha\xrightarrow{\eta\lambda}x$ and, therefore, $x=0$.
\end{proof}

We call a convergence structure $\eta$ \term{idempotent} if
$\eta\eta=\eta$ (idempotent convergence structures were introduced and
investigated in~\cite{Bilokopytov:23a}).

\begin{theorem}\label{min-uo-cont}
  Every minimal Hausdorff locally solid convergence is order
  continuous and idempotent, and is weaker than uo-convergence.
\end{theorem}

\begin{proof}
  Let $\eta$ be a minimal Hausdorff locally solid convergence. It
  follows from $\eta\eta\le\eta$ that $\eta\eta=\eta$.

  Recall that $\mathrm{o}$ stand for order convergence. Notice that
  $\eta\mathrm{o}$ is an order continuous Hausdorff locally solid
  convergence, and $\eta \mathrm{o}\le\eta$. It follows that
  $\eta\mathrm{o}=\eta$ and, therefore, $\eta$ is order continuous.

  Since $\eta$ is unbounded and $\eta\le\mathrm{o}$, we conclude that
  $\eta=\mathrm{u}\eta\le\mathrm{uo}$.
\end{proof}

Since every disjoint net is uo-null, we get the following:

\begin{corollary}
  For every minimal Hausdorff locally solid convergence, every
  disjoint net is null.
\end{corollary}

It may be interesting to compare properties of minimal Hausdorff
locally solid convergences to that of minimal Hausdorff locally solid
topologies. In the following theorem, we summarize several known
results about the minimal Hausdorff locally solid topologies.

\begin{theorem}[\cite{Aliprantis:03,Conradie:05,Taylor:19,Taylor:TH}]\label{min-top}
  ~
  \begin{enumerate}
  \item There is
    at most one minimal Hausdorff locally solid
    topology.
  \item A Hausdorff locally solid topology $\tau$ is minimal iff
    $\tau\le\mathrm{uo}$ iff $\tau$ is order continuous and unbounded.
  \item If $\tau$ is minimal, then it extends to a locally solid
    topology on the universal completion $X^u$ of~$X$, which then
    becomes the completion of $(X,\tau)$. In particular, $\tau$ is
    complete iff $X$ is universally complete.
  \end{enumerate}
\end{theorem}

We will now explore to which extent these results remain valid for
convergences.

\begin{example}
  \emph{A minimal Hausdorff locally solid convergence need not be
  unique}. In the notation of Example~\ref{no-least}, we use
  Remark~\ref{min-dom} to find two minimal Hausdorff locally solid
  convergences $\mu'$ and $\eta'$ such that $\mu'\le\mu$ and
  $\eta'\le\eta$. It follows that $x_n\xrightarrow{\mu'}0$ and
  $x_n\xrightarrow{\eta'}\one$. Therefore, $\mu'\ne\eta'$. 
\end{example}

It follows from Remark~\ref{min-dom}
that if there is a unique minimal Hausdorff locally solid
convergence, then it also has to be the least. Again, this fails for
topologies; see, e.g., \cite[Example~3.28]{Taylor:TH}.

\begin{example}
  By Theorem~\ref{min-uo-cont}, every minimal Hausdorff locally solid
  convergence is unbounded and order continuous. The converse is false:
  on $L_p[0,1]$ with $1\le p<\infty$, uo-convergence is unbounded and
  order continuous, yet it is not minimal as it strictly dominates
  un-convergence.
\end{example}

\begin{example}
  According to Theorem~\ref{min-top},
  every Hausdorff locally solid topology that is weaker than
  uo-convergence is itself unbounded. The following example shows that
  this fails for convergences. Let $X=C[0,1]$. For a positive net
  $(f_\alpha)$, we define $f_\alpha\goeseta 0$ if for every non-empty
  open interval $I$ in $[0,1]$ and every $\varepsilon>0$ there exists
  an non-empty open interval $J\subseteq I$ and an index $\alpha_0$
  such that for all $\alpha\ge\alpha_0$ the average of $f_\alpha$ over
  $J$ is less than $\varepsilon$. By Theorem~\ref{plsc}, this yields a
  Hausdorff locally solid convergence on~$X$. By
  \cite[Theorem~3.2]{Bilokopytov:22}, $\eta\le\mathrm{uo}$. We will
  show that, nevertheless, $\eta$ is not unbounded. Let $(f_n)$ be the
  Schauder sequence in $C[0,1]$ as in, say,
  \cite[p.~28]{Carothers:05}. Let
  \begin{math}
    g_n=f_n/\norm{f_n}_{L_1}.
  \end{math}
  Observe that $(g_n)$ is not $\eta$-null because for every interval
  $J$ and every $n$ there exists $m>n$ such that the support of $g_m$
  is contained in $J$ and, therefore, the average of $g_m$ over $J$ is
  greater that one. On the other hand, for every positive
  $u\in C[0,1]$, $g_n\wedge u\goeseta 0$ because for every
  interval~$J$, the average of $g_n\wedge u$ over $J$ converges to
  zero as $n\to\infty$.
\end{example}

Recall that, by~\cite[Theorem~17]{Azouzi:19} or \cite[Theorem~4.14]{Taylor:TH},
$X$ is universally complete iff it is uo-complete.

\begin{proposition}\label{univ-compl}
  If $X$ is complete with respect to a Hausdorff locally solid
  convergence $\eta$ with $\eta\le\mathrm{uo}$ then $X$ is universally complete.
\end{proposition}

\begin{proof}
  It suffices to prove that $X$ is uo-complete. Let
  $(x_\alpha)$ be a uo-Cauchy net in~$X$. Then it is $\eta$-Cauchy,
  hence $x_\alpha\goeseta x$ for some~$x$. We will show that
  $x_\alpha\goesuo x$. Fix $u\ge 0$. Since  $(x_\alpha)$ is
  uo-Cauchy, $x_\alpha-x_\beta\goesuo 0$. Hence
  there exists a net $v_\gamma$
  such that $v_\gamma\downarrow 0$ and for every $\gamma$ there exists
  $\alpha_0$ such that $\abs{x_\alpha-x_\beta}\wedge u\le v_\gamma$ for
  all $\alpha,\beta\ge\alpha_0$. Passing to the limit in~$\eta$, we
  conclude that $\abs{x_\alpha-x}\wedge u\le v_\gamma$ for
  all $\alpha\ge\alpha_0$. It follows that $x_\alpha\goesuo x$.
\end{proof}

\begin{corollary}\label{min-compl-univ}
  If $X$ is complete with respect to a minimal Hausdorff locally solid
  convergence then $X$ is universally complete.
\end{corollary}

We feel that we only scratched the surface in the subject of minimal
Hausdorff locally solid convergences. We end the section with a list
of questions that may suggest further developments of this topic.

\begin{question}
  Is every minimal Hausdorff locally solid convergence on a
  universally complete vector lattice necessarily complete? Does every
  minimal locally solid convergence on $X$ extend to $X^u$?
\end{question}

Assume that $X$ has the countable supremum property, and $\tau$ is the
minimal Hausdorff locally solid topology. By Theorem~\ref{min-top},
$\tau$ is unbounded and order continuous, so by
Proposition~\ref{uo-top-un},
$\tau=\mathrm{u}\tau=\mathrm{t}(\mathrm{uo})$. This motivates the
following question:

\begin{question}
  Suppose that $X$ admits a minimal Hausdorff locally solid topology.
  Does it agree with the topological or Choquet modification of
  uo-convergence? In particular, if the topological modification of
  uo-convergence is Hausdorff, is it a minimal Hausdorff locally
  solid topology?
\end{question}

\begin{question}
  Proposition~\ref{min-unb-Ch} and Theorem~\ref{min-uo-cont} motivate
  the following question: Is the Choquet modification of
  uo-convergence nec- essarily a minimal Hausdorff locally solid
  convergence?
\end{question}

\begin{question}
  Is every minimal Hausdorff locally solid topology automatically
  a minimal Hausdorff locally solid convergence?
\end{question}

\begin{question}
  It is known that for a Hausdorff locally solid topology, the
  following two properties are equivalent: every increasing order
  bounded net is Cauchy, and, every order bounded disjoint net is null;
  see \cite[Theorem~3.22]{Aliprantis:03}. Such topologies are called
  pre-Lebesgue. For locally solid convergences, the forward
  implication remains true (just consider the net of partial sums). Is
  the converse implication true for convergences?
\end{question}

\section{Adherence of an ideal}
\label{sec:u-ideal}

Throughout this section, ``$\to$'' is a Hausdorff locally solid
convergence structure on a vector lattice~$X$. Recall that for a
subset $A$ of~$X$, we write $\overline{A}^1$ for the adherence of~$A$,
$\overline{A}^2$ for the adherence of $\overline{A}^1$, etc. $J$ will
stand for an ideal of~$X$. For $x\in X^+$, the set $[0,x]\cap J$ may
be viewed as an increasing net in~$X$.

\begin{lemma}\label{ideal-adh}
  Let $y\in X_+$. Then $y\in\overline{J}^1$ iff $y=\lim[0,y]\cap J$. 
\end{lemma}

\begin{proof}
  Suppose that $y\in\overline{J}^1$. Find a net $(x_\alpha)$ in $J$
  with $x_\alpha\to y$. Replacing $x_\alpha$ with
  $y\wedge x_\alpha^+$, we may assume without loss of generality that
  $x_\alpha\in [0,y]\cap J$. Proposition~\ref{sset-mon-conv} yields that
  $y=\lim[0,y]\cap J$. 

  The converse implication is trivial.
\end{proof}

\begin{corollary}
  The adherence of an ideal $J$ is not changed by replacing the given
  locally solid convergence by its unbounded or Choquet modification,
  or its bounded modification by a locally solid bornology.
\end{corollary}

\begin{proof}
  For the unbounded modification, this follows from the fact that the
  net $[0,y]\cap J$ is order bounded. The same works for a bounded
  modification with a locally solid bornology, because such a
  bornology contains order bounded sets. For the Choquet modification,
  this follows from the fact that the net $[0,y]\cap J$ is increasing
  and Proposition~\ref{Ch-ls}.
\end{proof}

\begin{proposition}\label{adhJpb-intlim}
  $\overline{J}^1$ is a projection band iff the net $[0,x]\cap J$ is
  convergent for every $x\in X_+$.
\end{proposition}

\begin{proof}
  Suppose that $\overline{J}^1$ is a projection band; let $x\in
  X_+$. Then $y:=\sup[0,x]\cap\overline{J}^1$ exists. Clearly,
  $y\le x$ and $y\in \overline{J}^1$. It is easy to see that
  $[0,x]\cap J=[0,y]\cap J$. It follows from Lemma~\ref{ideal-adh}
  that $y=\lim[0,y]\cap J=\lim[0,x]\cap J$.

  Conversely, suppose that $\lim[0,x]\cap J$ exists for every
  $x\in X_+$.  Fix $x\in X_+$. Put $y:=\lim[0,x]\cap J$. Clearly,
  $y\in\overline{J}^1$ and $y\le x$. By Proposition~\ref{MCT}, we 
  have $y=\sup[0,x]\cap J$.

  We claim that $x-y\perp J$. To see this, let $u\in[0,x]\cap J$ (in
  particular, $u\le y$) and $z\in[0,x-y]\cap J$. Then $u+z\le x$, so
  that $u+z\in[0,x]\cap J$ and, therefore, $u+z\le y$. It follows that
  $u\le y-z$. Taking the supremum over all $u\in[0,x]\cap J$, we get
  $y\le y-z$, hence $z=0$. This yields $[0,x-y]\cap J=\{0\}$, hence
  $x-y\perp J$.

  It follows that $x-y\perp J^{dd}$. Since $J^{dd}$ is closed by
  Proposition~\ref{Jd-closed}, we have $x-y\perp \overline{J}^1$. So
  we have a decomposition $x=y+(x-y)$ with $y\in \overline{J}^1$ and $x-y\perp
  \overline{J}^1$. We conclude that  $\overline{J}^1$ is a
  projection band.
\end{proof}

\begin{lemma}\label{uJu-int}
  Let $\mathrm{u}_J(\mathrm{ru})$ be the unbounded modification of
  relative uniform convergence with respect to~$J$. Then
  $[0,x]\cap J\xrightarrow{\mathrm{u}_J(\mathrm{ru})}x$ for every
  $x\in X_+$.
\end{lemma}

\begin{proof}
  Fix $e\in J_+$. For every $n\in\mathbb N$, we have
  \begin{displaymath}
    (x-x\wedge ne)\wedge e\le(x-ne)^+\wedge(e-\tfrac1nx)^++\tfrac1nx
    =\tfrac1nx.
  \end{displaymath}
  It follows that
  $(x-x\wedge ne)\wedge e\goesu x$. This sequence is contained, as a
  subset, in the decreasing positive net $(x-y)\wedge e$ indexed by
  $y\in[0,x]\cap J$. By Proposition~\ref{sset-mon-conv}, the latter net
  $\mathrm{ru}$-converges to zero. Since  $e\in J_+$ is arbitrary, we
  get the conclusion.
\end{proof}

\begin{proposition}\label{uI-unbdd}
  Let $\eta$ be a Hausdorff locally solid convergence on $X$ and $J$
  an ideal in~$X$. If $\mathrm{u}_J\eta=\eta$ then $\overline{J}^{1,\eta}=X$. 
\end{proposition}

\begin{proof}
  Let $x\in X_+$. Then
  \begin{math}
    [0,x]\cap J\xrightarrow{\mathrm{u}_J(\mathrm{ru})}x,
  \end{math}
  so that $x\in\overline{J}^{1,\mathrm{u}_J(\mathrm{ru})}$.
  Since ru is the strongest locally solid convergence, we get
  $x\in\overline{J}^{1,\mathrm{u}_J\eta}=\overline{J}^{1,\eta}$.
\end{proof}

\begin{question}
  Is the converse true? That is, if $\eta$ is unbounded and
  $\overline{J}^{1,\eta}=X$, does this imply $\mathrm{u}_J\eta=\eta$? 
\end{question}

The answer to this question is ``yes'' if $\eta$ is
idempotent; see Proposition~3.1(v) in \cite{Bilokopytov:23a}.

\medskip

If $B$ is a band in a Banach lattice~$X$, then the restriction
$(\mathrm{un})_{|B}$ of un-convergence in $X$ to $B$ does not
necessarily equal the ``native'' un-convergence of~$B$,
$\mathrm{u}(\textrm{n}_{|B})$; see Example~4.2
in~\cite{Kandic:17}. If $B$ is a projection band, we do have equality
by Theorem~4.3 in~\cite{Kandic:17}. In Question~4.7
in~\cite{Kandic:17} it was asked whether the converse is true: if
$(\mathrm{un})_{|B}=\mathrm{u}(\mathrm{n}_{|B})$, is $B$ a projection
band? The following theorem answers this question in the affirmative;
just take $\eta$ to be the norm convergence and $J=B=\overline{J}^1$.

\begin{theorem}\label{ideal-comm}
  Let $(X,\eta)$ be a complete locally solid convergence vector
  lattice, and $J$ an ideal in~$X$. Suppose that
  $(\mathrm{u}\eta)_{|J}=\mathrm{u}(\eta_{|J})$, that is, the
  operation of taking the unbounded modification commutes with that of
  taking the restriction to~$J$. Then $\overline{J}^1$ is a projection
  band.
\end{theorem}

\begin{proof}
  Fix $x\in X_+$. By Proposition~\ref{adhJpb-intlim} and since $\eta$
  is complete, it suffices to prove that the net $[0,x]\cap J$ is
  $\eta$-Cauchy. By Lemma~\ref{uJu-int},
  $[0,x]\cap J\xrightarrow{\mathrm{u}_J(\mathrm{ru})}x$. By
  Proposition~\ref{u-strongest},
  $[0,x]\cap J\xrightarrow{\mathrm{u}_J\eta}x$. It follows that this
  net is $\mathrm{u}_J\eta$-Cauchy and, therefore,
  $\mathrm{u}(\eta_{|J})$-Cauchy. By assumption, we conclude that the
  net is $(\mathrm{u}\eta)_{|J}$-Cauchy, hence
  $\mathrm{u}\eta$-Cauchy. Finally, this net is bounded, thus 
  $\eta$-Cauchy.
\end{proof}

\begin{remark}
  The condition $(\mathrm{u}\eta)_{|J}=\mathrm{u}(\eta_{|J})$ in the
  theorem is equivalent to $(\mathrm{u}\eta)_{|J}$ being
  unbounded. The forward implication is obvious. For the converse, if
  $(\mathrm{u}\eta)_{|J}$ is unbounded then we have
  \begin{math}
    \mathrm{u}(\eta_{|J})
    \le(\mathrm{u}\eta)_{|J}
    =\mathrm{u}\bigl((\mathrm{u}\eta)_{|J}\bigr)
    \le\mathrm{u}(\eta_{|J}).
  \end{math}
\end{remark}

Let $X$ be a Banach lattice and $(x_\alpha)$ a net in~$X$. Since $X$
may be viewed as a closed sublattice of~$X^{**}$,  if
$x_\alpha\goesun 0$ in $X^{**}$ then $x_\alpha\goesun 0$ in~$X$. That is,
the un-convergence on $X$ is weaker than (or equal to) the un-convergence
on $X^{**}$ reduced to~$X$.

\begin{corollary}
  Let $X$ be an order continuous Banach lattice. The following are equivalent:
  \begin{enumerate}
  \item $X$ is a KB-space;
  \item\label{KB-iff} $x_\alpha\goesun 0$ in $X$ iff
    $x_\alpha\goesun 0$ in $X^{**}$ for every net $(x_\alpha)$ in~$X$.
  \end{enumerate}
\end{corollary}

\begin{proof}
  Since $X$ is order continuous, $X$ is a closed ideal
  in~$X^{**}$. $X$ is a KB-space iff $X$ is a projection band
  in~$X^{**}$. By Theorem~\ref{ideal-comm}, the latter is equivalent
  to~\eqref{KB-iff}.
\end{proof}

The theorem fails without the order continuity assumption; $C[0,1]$ is a
counterexample, because un-convergence agrees with norm
convergence on both $C[0,1]$ and $C[0,1]^{**}$. 

\bigskip

\textbf{Acknowledgements.} We would like to thank K.~Abela,
M.~O'Brien, M.~Taylor, M.~Wortel, and O.~Zabeti for valuable discussions.

\end{document}